\documentclass{amsart}

\usepackage[utf8]{inputenc}

\usepackage[T1]{fontenc}

\usepackage[margin=1.5in]{geometry}
\usepackage{graphicx}
\usepackage[centertags]{amsmath}
\usepackage{amsfonts}
\usepackage{amssymb}
\usepackage{amsthm}
\usepackage{newlfont}
\usepackage{hyperref}

\makeatletter
\newcommand{\newreptheorem}[2]{
  \newtheorem*{rep@#1}{\rep@title}
  \newenvironment{rep#1}[1]{
    \def\rep@title{#2 \ref*{##1}}\begin{rep@#1}}
  {\end{rep@#1}}
  }
\makeatother

\newtheorem{thm}{Theorem}[section]
\newreptheorem{thm}{Theorem}
\newtheorem{cor}[thm]{Corollary}
\newreptheorem{cor}{Corollary}
\newtheorem{lem}[thm]{Lemma}
\newtheorem{prop}[thm]{Proposition}

\theoremstyle{definition}
\newtheorem{defn}[thm]{Definition}

\newtheorem{question}[thm]{Question}
\newtheorem{rem}[thm]{Remark}
\newtheorem{assume}[thm]{Convention}

\numberwithin{equation}{section}

\newcommand{\N}{\mathbb{N}}
\newcommand{\Z}{\mathbb{Z}} 
\newcommand{\R}{\mathbb{R}} 
\newcommand{\C}{\mathbb{C}} 
\newcommand{\Qmath}{\mathbb{Q}} 
\newcommand{\acts}{\curvearrowright} 
\newcommand{\pmp}{p{$.$}m{$.$}p{$.$}} 
\newcommand{\Sym}{\mathrm{Sym}} 
\newcommand{\sH}{\mathrm{H}} 
\newcommand{\AM}{\mathcal{M}} 
\newcommand{\avg}{\mathbb{E}} 
\newcommand{\var}{\mathrm{Var}} 
\newcommand{\cost}{\mathcal{C}} 
\newcommand{\costsup}{\mathcal{C}_{\mathrm{sup}}} 
\newcommand{\cP}{\mathcal{P}} 
\newcommand{\cQ}{\mathcal{Q}} 
\newcommand{\cJ}{\mathcal{J}} 
\newcommand{\Map}{\mathrm{Map}}
\newcommand{\Mat}[3]{\mathrm{Mat}_{#1 \times #2}(#3)} 
\newcommand{\KK}{\mathbb{K}} 
\newcommand{\id}{\mathrm{id}} 
\newcommand{\dom}{\mathrm{dom}} 
\newcommand{\Img}{\mathrm{Im}} 
\newcommand{\res}{\upharpoonright} 
\newcommand{\Cay}{\mathrm{Cay}} 
\newcommand{\Fix}[2]{\mathrm{Fix}_{#1}(#2)} 
\newcommand{\Borel}{\mathcal{B}} 
\newcommand{\calF}{\mathcal{F}} 
\newcommand{\calR}{\mathcal{R}} 
\newcommand{\rh}[3]{h^{\mathrm{Rok}}(#1 \acts #2, #3)} 
\newcommand{\toph}[3]{h_{\mathrm{top}}^{#1}(#2 \acts #3)} 
\newcommand{\meash}[4]{h_{\mathrm{meas}}^{#1}(#2 \acts #3, #4)} 

\usepackage{color}

\begin{document}

\title{Cost, $\ell^2$-Betti numbers and the sofic entropy of some algebraic actions}

\author{Damien Gaboriau}
\address{Unit{\'e} de Math{\'e}matiques Pures et Appliqu{\'e}es, ENS-Lyon, CNRS, Université de Lyon, 46 all{\'e}e d'Italie, 69007 Lyon, France}
\email{damien.gaboriau@ens-lyon.fr}

\author{Brandon Seward}
\address{Einstein Institute of Mathematics, The Hebrew University of Jerusalem, Givat Ram, Jerusalem 91904, Israel}
\email{b.m.seward@gmail.com}

\date{\today}

\begin{abstract}
In 1987, Ornstein and Weiss discovered that the Bernoulli $2$-shift over the rank two free group factors onto the seemingly larger Bernoulli $4$-shift. With the recent creation of an entropy theory for actions of sofic groups (in particular free groups), their example shows the surprising fact that entropy can increase under factor maps. In order to better understand this phenomenon, we study a natural generalization of the Ornstein--Weiss map for countable groups. We relate the increase in entropy to the cost and to the first $\ell^2$-Betti number of the group. More generally, we study coboundary maps arising from simplicial actions and, under certain assumptions, relate $\ell^2$-Betti numbers to the failure of the Yuzvinsky addition formula. This work is built upon a study of entropy theory for algebraic actions. We prove that for actions on profinite groups via continuous group automorphisms, topological sofic entropy is equal to measure sofic entropy with respect to Haar measure whenever the homoclinic subgroup is dense. For algebraic actions of residually finite groups we find sufficient conditions for the sofic entropy to be equal to the supremum exponential growth rate of periodic points.
\end{abstract}
\maketitle

\newpage
\section{Introduction}

Entropy is a fundamental invariant of dynamical systems.  
It was first defined for probability measure preserving actions of $\Z$ by Kolmogorov in 1958 \cite{Kol58,Kol59} and then extended to actions of countable amenable groups by Kieffer in 1975 \cite{Ki75}.
The notion was transferred to continuous actions of $\Z$ by Adler-Konheim-McAndrew in 1965 \cite{AKM65}.
 Despite evidence suggesting that entropy theory could not be extended beyond actions of amenable groups, groundbreaking work by Bowen in 2008 \cite{B10b}, together with improvements by Kerr and Li \cite{KL11a}, created a definition of entropy for probability measure preserving actions of sofic groups. This definition was also generalized to the topological setting by Kerr and Li \cite{KL11a}. This new notion of entropy is an extension of its classical counterpart, as when the acting sofic group is amenable the two notions coincide \cite{Ba,KL}, however it displays surprising behavior that violates some of the fundamental properties of classical entropy theory. In this paper, we study this strange behavior of entropy by drawing connections to $\ell^2$-Betti numbers and cost. In particular, our work uncovers instances of a close connection between entropy and sequences of normalized Betti numbers computed over finite fields.

\medskip
A classical result in entropy theory is the \emph{Yuzvinsky addition formula}. This formula states the following:
If a countable amenable group $G$ acts on a compact metrizable group $H$ by continuous group automorphisms, and $N \lhd H$ is a closed normal $G$-invariant subgroup, so that we have a $G$-equivariant exact sequence
$$1\to N\to H\to H/N\to 1,$$
then the entropies are related by the formula:
$$h(G \acts H)=h(G \acts N)+h(G \acts H/N).$$
This can be viewed as an analogue of the rank-nullity theorem from linear algebra. In this form, this result is due independently to Li \cite{Li11} and Lind--Schmidt \cite{LS09} and holds true both for topological entropy and for measured entropies with respect to Haar probability measures. The case $G = \Z$ was originally obtained by Yuzvinsky \cite{J65}, the case $G = \Z^d$ by Lind, Schmidt, and Ward \cite{LSW90}, and other special cases were obtained by Miles \cite{M08} and Bj\"{o}rklund--Miles \cite{BM09}. Outside of algebraic actions and algebraic maps, a weak form of the Yuzvinsky addition formula exists. Specifically, for actions of amenable groups entropy is monotone decreasing under factor maps: the entropy of a factor action can be at most the entropy of the original action.

For actions of non-amenable groups, problems with Yuzvinsky's addition formula and entropy theory in general were first evidenced by an example of Ornstein and Weiss in 1987 \cite{OW87}. They considered, for the rank $2$ free group $F_2$ with free generating set  $\{a, b\}$, the $F_2$-equivariant continuous group homomorphism 
 $$\theta: (\Z / 2 \Z)^{F_2}\to (\Z / 2 \Z \times \Z / 2 \Z)^{F_2}$$ defined by $$\theta(x)(f) = \Big( x(f) - x(fa), \ x(f) - x(f b) \Big) \mod 2.$$
They observed that it factors the Bernoulli shift $(\Z / 2 \Z)^{F_2}$ onto the ``larger'' Bernoulli shift $(\Z / 2 \Z \times \Z / 2 \Z)^{F_2}$.
Since for an amenable group $G$ the Bernoulli shift $G \acts \{1, 2,\cdots, n\}^G$ has entropy $\log(n)$, their discovery suggested that entropy theory could not be extended to actions of $F_2$. Specifically, if $(\Z / 2 \Z)^{F_2}$ and $(\Z / 2 \Z \times \Z / 2 \Z)^{F_2}$ were to have entropies $\log(2)$ and $\log(4)$ as one would expect, then this example would both violate the Yuzvinsky addition formula and violate the monotonicity property of entropy. Today an entropy theory for actions of sofic groups, including $F_2$, does exist, and indeed this example demonstrates that both the Yuzvinsky addition formula and the monotonicity property can fail.

We mention that modifications of the Ornstein--Weiss factor map have played important roles in later work on the failure of monotonicity. Specifically, Ball proved that for every non-amenable group $G$ there is $n \in \N$ such that $(\{0, \cdots, n-1\}^G, u_n^G)$ factors onto all other measure-theoretic Bernoulli shifts over $G$ \cite{Ba05}, where $u_n$ is the normalized counting measure on $\{0, \cdots, n-1\}$. Similarly, Bowen proved that if $F_2 \leq G$ then all measure-theoretic Bernoulli shifts over $G$ factor onto one another \cite{Bo11}. Finally, the second author proved that for every non-amenable group $G$ there is $n \in \N$ such that every action (either topological or measure-theoretic) of $G$ is the factor of an action having entropy at most $\log(n)$ \cite{S13}.

In this paper we study a generalization of the Ornstein--Weiss factor map which is distinct from those used in \cite{Ba05}, \cite{Bo11}, and \cite{S13}. If $S$ is any generating set for $G$, not necessarily finite, and $K$ is a finite additive abelian group, then we define the generalized Ornstein--Weiss map $\theta^{ow}_S : K^G \rightarrow (K^S)^G$ by
$$\theta^{ow}_S(x)(g)(s) = x(g) - x(g s) \in K.$$
If we identify the kernel of $\theta^{ow}_S$ (the set of constant functions in $K^G$) with $K$, then the map $\theta^{ow}_S$ coincides (modulo an isomorphism between the respective images) with the quotient map $K^G \rightarrow K^G / K$. In particular, up to an isomorphism $\theta^{ow}_S$ does not depend on the choice of generating set $S$.

We mention that the image of the generalized Ornstein--Weiss map, $K^G / K$, has also been studied by Meesschaert--Raum--Vaes \cite{MRV} and by Popa \cite{Po06} (in both cases the connection to the Ornstein--Weiss map was coincidental). Meeschaert, Raum, and Vaes proved that if $G$ is the free product of amenable groups then $(K^G / K, \mathrm{Haar})$ is isomorphic to a Bernoulli shift, while the work of Popa gives as a consequence that for many non-amenable groups, such as property (T) groups, $(K^G / K, \mathrm{Haar})$ is not isomorphic to a Bernoulli shift.

\medskip
Our first result relates the entropy of the image of the generalized Ornstein--Weiss map to the first $\ell^2$-Betti number of $G$, denoted $\beta_{(2)}^1(G)$, and to the supremum-cost of $G$. 
The $\ell^2$-Betti numbers were defined for arbitrary countable groups by Cheeger and Gromov \cite{CG86}.
The notion of cost in ergodic theory was introduced by Levitt in \cite{Lev95} and studied by the first author in \cite{G00}. 
The supremum-cost of a group $G$, denoted $\costsup(G)$, is the supremum of the costs of all free probability measure preserving ({\pmp}) actions of $G$ (see Section~\ref{sec:prelim}). 

\medskip
The definition of topological sofic entropy for a continuous action $G \acts X$ on a compact metrizable space $X$ will be recalled in Section~\ref{sec:tent}, while the definition of measured sofic entropy for a {\pmp} action $G\acts (X,\mu)$ on a standard probability space will be recalled in Section~\ref{sec:ment}. For now we simply mention that these entropies require $G$ to be sofic and rely upon the choice of a sofic approximation $\Sigma$ to $G$. They are denoted respectively by
\begin{equation*}
\toph{\Sigma}{G}{X} \ \textrm{ and } \ \meash{\Sigma}{G}{X}{\mu}.
\end{equation*}

\begin{repthm}{thm:cost}[Measured and topological entropy vs cost and $\ell^2$-Betti number] 
Let $G$ be a countably infinite sofic group, let $\Sigma$ be a sofic approximation to $G$, and let $\KK$ be a finite field. Then
\begin{align*}
& \ \meash{\Sigma}{G}{\KK^G / \KK}{\mathrm{Haar}} = \toph{\Sigma}{G}{\KK^G / \KK} \leq \costsup(G) \cdot \log |\KK|.\\
\intertext{Furthermore, if $G$ is finitely generated then}
\Big( 1 + \beta_{(2)}^1(G) \Big) \cdot \log |\KK| \leq & \ \meash{\Sigma}{G}{\KK^G / \KK}{\mathrm{Haar}} = \toph{\Sigma}{G}{\KK^G / \KK}.
\end{align*}
\end{repthm}

Thus, since $\meash{\Sigma}{G}{\KK^G}{\mathrm{Haar}} = \toph{\Sigma}{G}{\KK^G} = \log |\KK|$, the generalized Ornstein--Weiss map increases entropy, and hence violates monotonicity, when $G$ is a finitely generated sofic group and $\beta_{(2)}^1(G) > 0$. Examples of groups with $\beta_{(2)}^1(G) > 0$ are free groups, all free products of non-trivial groups with one factor having at least $3$ elements, amalgamated free products $G_1 *_H G_2$ with $H$ infinite amenable and $\beta_{(2)}^1(G_1) + \beta_{(2)}^1(G_2) > 0$ \cite{CG86}, surface groups of genus at least $2$, surface groups of genus at least $2$ modulo a single relation \cite{DL07}, and one relator groups with more than $3$ generators \cite{DL07}. See \cite{PT11} for more examples. We mention that there exist finitely generated residually finite torsion groups with $\beta_{(2)}^1(G) > 0$ (see \cite{LO11}).

On the other hand, the generalized Ornstein--Weiss map preserves entropy whenever $\costsup(G) = 1$ or, equivalently, whenever $G$ has fixed price $1$. Examples of groups with fixed price $1$ are infinite amenable groups \cite{OW80}, infinite-conjugacy-class (icc) inner amenable groups \cite{T14}, direct products $G \times H$ where $H$ is infinite and $G$ contains a fixed price $1$ subgroup \cite{G00}, groups with a normal fixed price $1$ subgroup \cite{G00}, amalgamated free products of fixed price $1$ groups over infinite amenable subgroups, Thompson's group F \cite{G00}, $\mathrm{SL}(n, \Z)$ for $n \geq 3$ \cite{G00}, non-cocompact arithmetic lattices in connected semi-simple algebraic Lie groups of  $\Qmath$-rank at least $2$ \cite{G00}, and groups generated by chain-commuting infinite order elements \cite{G00} (i.e. infinite order elements whose graph of commutation is connected).

\begin{rem}
Two famous open problems are the cost versus first $\ell^2$-Betti number problem, which asks if $1 + \beta_{(2)}^1(G) = \cost(G)$ (see \cite{G02} where the inequality $\leq$ is proved), and the fixed price problem \cite{G00}, which asks if $\cost(G) = \costsup(G)$. If both of these problems have a positive answer then $1 + \beta_{(2)}^1(G) = \costsup(G)$ and hence when $G$ is finitely generated the inequalities in Theorem \ref{thm:cost} are actually equalities. It is known that $1 + \beta_{(2)}^1(G) = \costsup(G)$ if $G$ is a free group, a surface group, an amalgamated free product $G_1 *_H G_2$ with $1 + \beta_{(2)}^1(G_i) = \costsup(G_i)$ and $H$ amenable, if $G$ has fixed price $1$ \cite{G02}, or if $\costsup(G)$ is realized by a treeable action \cite{G02} (a treeable Bernoulli shift would be sufficient when $G$ is finitely generated \cite{AW13}).
\end{rem}
In the measure-theoretic case, we actually prove something stronger than the first inequality in Theorem \ref{thm:cost}. Recall that Rokhlin's generator theorem states that for a free ergodic {\pmp} action $\Z \acts (X, \mu)$, the classical Kolmogorov--Sinai entropy $\meash{}{\Z}{X}{\mu}$ is equal to the infimum of the Shannon entropies of countable generating partitions \cite{Roh67} (for relevant definitions, see Section \ref{sec:ment}). Rokhlin's generator theorem continues to hold for free ergodic actions of countable amenable groups \cite{ST14}. Drawing upon these facts, in \cite{S14} the second author defined the Rokhlin entropy $\rh{G}{X}{\mu}$ of a {\pmp} ergodic action $G \acts (X, \mu)$ of a general countable group $G$ to be the infimum of the Shannon entropies of countable generating partitions. When the group is sofic and the action is ergodic, Rokhlin entropy satisfies the inequality \cite{B10b}
$$\meash{\Sigma}{G}{X}{\mu} \leq \rh{G}{X}{\mu}.$$
Now we may state our stronger version of the first inequality of Theorem~\ref{thm:cost}.
\begin{repthm}{thm:rokcost}[Rokhlin entropy and cost]
Let $G$ be a countably infinite group, not necessarily sofic, and let $K$ be a finite abelian group. Then
$$ \rh{G}{K^G / K}{\mathrm{Haar}} \leq \costsup(G) \cdot \log |K|.$$
\end{repthm}

We remark that there may exist deeper connections between first $\ell^2$-Betti numbers, cost, and entropy. In \S\ref{sec:disc} we discuss the possibility of relating the first $\ell^2$-Betti number and the cost of a {\pmp} countable Borel equivalence relation $R$ to Rokhlin entropy and sofic entropy via inequalities similar to those in Theorems \ref{thm:cost} and \ref{thm:rokcost}.

\medskip
The Ornstein--Weiss map can be further generalized by observing that it corresponds to the coboundary map on the Cayley graph of $G$. This observation leads us to consider coboundary maps coming from any simplicial action of $G$. Below we write $\beta^p_{(2)}(L : G)$ for the $p^{\text{th}}$ $\ell^2$-Betti number of a free simplicial action $G \acts L$ (see Section \ref{sec:prelim}).

\begin{repthm}{thm:betti}
Let $G$ be a sofic group with sofic approximation $\Sigma$, and let $G$ act freely on a simplicial complex $L$. Consider the coboundary maps with coefficients in a finite field $\KK$:
\begin{equation*}
\begin{matrix}
C^{p-1}(L,\KK) & \overset{\delta^{p}}{\longrightarrow} & C^{p}(L,\KK)  & \overset{\delta^{p+1}}{\longrightarrow} & C^{p+1}(L,\KK) .
\end{matrix}
\end{equation*}
If $p \geq 1$ and the action of $G$ on the $p$-skeleton of $L$ is cocompact then 
\begin{equation*}
\toph{\Sigma}{G}{C^{p-1}(L, \KK)} + \beta^p_{(2)} ( L : G ) \cdot \log |\KK| \leq \toph{\Sigma}{G}{\ker(\delta^p)} + \toph{\Sigma}{G}{\ker(\delta^{p+1})}.
\end{equation*}
\end{repthm}

The above theorem specializes to a result previously obtained by Elek in the case where $G$ is amenable (this follows from \cite{E02} but is not explicitly stated; one must combine the first sentence in the proof of \cite[Prop 9.1]{E02} with the remark following \cite[Prop. 9.3]{E02}). In the non-amenable setting, a new novel feature of the above formula is that it leads to relating $\ell^2$-Betti numbers to the failure of the Yuzvinsky addition formula as follows.

\begin{repcor}{cor:juzv}
With the notation and assumptions of Theorem \ref{thm:betti}, if we furthermore have $\Img(\delta^{p})=\ker(\delta^{p+1})$ (equivalently $H^{p}(L,\KK)=0$), then
\begin{equation*}
\toph{\Sigma}{G}{C^{p-1}(L, \KK)} + \beta^p_{(2)} ( L : G ) \cdot \log |\KK| \leq \toph{\Sigma}{G}{\ker(\delta^p)} + \toph{\Sigma}{G}{\Img(\delta^p)}.
\end{equation*}
\end{repcor}

Under stronger assumptions (needed to relate  the measured and topological sofic entropies), we obtain a similar inequality in the measure-theoretic case.

\begin{cor}[Abbreviated version of Corollary \ref{cor:juzv2}]
Let $G$, $L$, $\KK$, and $p$ be as in Theorem \ref{thm:betti}. Assume that either \\
(1) $p > 1$ and $H^{p-1}(L, \KK) = H^p(L, \KK) = 0$, or 
\\
(2) $p = 1$, $H^1(L,\KK) = 0$, and $\toph{\Sigma}{G}{\ker(\delta^1)} = 0$.
\\
 If $\beta^p_{(2)}(L : G) > 0$, then $\delta^p$ violates the Yuzvinsky addition formula for measured sofic entropy. In fact, for $p > 1$ we have
\begin{align*}
 \ \meash{\Sigma}{G}{C^{p-1}(L, \KK)}{\mathrm{Haar}}& + \beta^p_{(2)}(L : G) \cdot \log |\KK|\\
\leq & \ \meash{\Sigma}{G}{\ker(\delta^p)}{\mathrm{Haar}} + \meash{\Sigma}{G}{\Img(\delta^p)}{\mathrm{Haar}}.
\end{align*}
\end{cor}

\begin{rem}
In general, topological entropy is either $- \infty$ or else non-negative. However, any action with a fixed point has non-negative topological entropy. In particular, if $H$ is a compact metrizable group and $G \acts H$ by continuous group automorphisms then $\toph{\Sigma}{G}{H} \geq 0$ for all sofic approximations $\Sigma$ since $1_H$ is a fixed point. Thus, all entropies appearing in Theorem \ref{thm:betti} and Corollary \ref{cor:juzv} are non-negative.
\end{rem}

When $G$ is residually finite and $\Sigma$ comes from a sofic chain of finite-index subgroups $(G_n)$, our methods reveal (see Corollary \ref{cor:finitebetti}) a close connection between the Yuzvinsky addition formula and sequences of normalized Betti numbers computed over $\KK$:
$$\frac{\dim_{\KK} H^{p}(G_n \backslash L, \KK)}{|G : G_n|}.$$
The convergence of such sequences, the significance of the limit value, and relation of the limit to $\beta_{(2)}^{p}(L : G)$ are important open questions (for example, see \cite{EL14}). Our work shows that if the entropy $\toph{\Sigma}{G}{\ker(\delta^p)}$ is achieved as a limit rather than a limit-supremum, then the inequality in Theorem \ref{thm:betti} becomes an equality after you replace $\beta^p_{(2)}(L : G)$ with the limit-supremum of the normalized Betti numbers computed over $\KK$:
\begin{align*}
\limsup_{n \rightarrow \infty} & \frac{\dim_\KK H^p(G_n \backslash L, \KK)}{|G : G_n|} \cdot \log|\KK|
\\
& \qquad \qquad \qquad = \toph{\Sigma}{G}{\ker(\delta^p)} + \toph{\Sigma}{G}{\ker(\delta^{p+1})}-\toph{\Sigma}{G}{C^{p-1}(L, \KK)} .
\end{align*}
This draws a connection between finite-field homology problems and sofic entropy problems, such as, for example, whether entropy depends upon the choice of sofic approximation $\Sigma$. In analogy with the work of Elek \cite{E02}, these observations lead us to introduce (in the context of Theorem \ref{thm:betti}) the 
\emph{$p$-th sofic entropy Betti number over the field $\KK$ of the action of $G$ on $L$}:
\begin{equation*}
\beta^{p,\Sigma}_{\KK} ( L : G ) : = \frac{1}{\log |\KK|}\Bigl[ \toph{\Sigma}{G}{\ker(\delta^{p+1})} 
 - \bigl[ \toph{\Sigma}{G}{C^{p-1}(L, \KK)} -
\toph{\Sigma}{G}{\ker(\delta^p)} \bigl]\Bigl].
\end{equation*}

\medskip
It remains an open question if the Yuzvinsky addition formula fails for every non-amenable sofic group. In the measure-theoretic setting, the addition formula has been shown to fail for free groups (and groups containing a free subgroup) by Ornstein--Weiss \cite{OW87} and groups with defined and non-zero $\ell^2$-torsion by Hayes \cite{H14}. In the topological setting, the addition formula was shown to fail for groups having a non-zero Euler characteristic by Elek \cite{E99}. Hayes also showed that in the topological setting the class of groups that violate the Yuzvinsky addition formula is closed under passing to supergroups \cite{H14}. Through our work we uncover another class of groups for which the addition formula fails.

\begin{repthm}{thm:fail}[Failure of the Yuzvinsky formula, topological]
Let $G$ be a sofic group containing an infinite subgroup $\Gamma$ with some non-zero $\ell^2$-Betti number $\beta^p_{(2)}(\Gamma) > 0$. Then $G$ admits an algebraic action and an algebraic factor map that simultaneously violates the Yuzvinsky addition formula for topological sofic entropy for all sofic approximations to $G$. 
\end{repthm}

In regard to measure-theoretic entropy, we obtain the following.

\begin{repthm}{thm:fail2}
Let $G$ be a sofic group containing an infinite subgroup $\Gamma$ such that $\Gamma$ has some non-zero $\ell^2$-Betti number $\beta^p_{(2)}(\Gamma) > 0$ and admits a free cocompact action on $p$-connected simplicial complex (for instance, if $\Gamma$ has a finite classifying space). Then $G$ admits an algebraic action and an algebraic factor map that simultaneously violates the Yuzvinsky addition formula for both measured (with respect to Haar probability measures) and topological sofic entropy for all sofic approximations to $G$.
\end{repthm}

For {\pmp} actions of finite rank free groups there is an alternate entropy-like quantity called the f-invariant which was introduced by Bowen \cite{B10a}. We remark for completeness that Bowen and Gutman have shown that the f-invariant satisfies the Yuzvinsky addition formula for a large class of algebraic actions \cite{BG}.

In proving these results, we obtain a few specialized formulas for the topological entropy of the natural shift-action of $G$ on $G$-invariant compact subgroups $X \subseteq K^G$, where $K$ is either a finite or profinite group. In certain cases, this entropy is precisely given by the exponential growth rate of the number of periodic points in $X$. This extends the classical theme connecting entropy to the growth rate of periodic points to the sofic setting. The definition of a subshift/subgroup of finite type is recalled in Section \ref{sec:prelim}.

\begin{repthm}{thm:period}[Topological entropy and fixed points]
Let $G$ be a residually finite group, let $(G_n)$ be a sofic chain of finite-index subgroups, and let $\Sigma=\bigl(\sigma_n:G\to \Sym(G_n \backslash G)\bigr)$ be the associated sofic approximation. Let $K$ be a finite group. If $X \subseteq K^G$ is a $G$-invariant compact subgroup of finite type then
$$\toph{\Sigma}{G}{X} = \limsup_{n \rightarrow \infty} \frac{1}{|G : G_n|} \cdot \log \Big| \Fix{G_n}{X} \Big|,$$
where $\Fix{G_n}{X}$ is the set of $G_n$-periodic elements of $X$.
\end{repthm}

\begin{rem}
In view of the above theorem, under the stronger assumption that the action of $G$ on the $(p+1)$-skeleton of $L$ is cocompact, Theorem \ref{thm:betti} and Corollary \ref{cor:juzv} can be interpreted in terms of growth rates of periodic points. Specifically, a basic symbolic dynamics argument implies that for any compact group $K$, finite group $L$, and any continuous $G$-equivariant group homomorphism $\phi : K^G \rightarrow L^G$, the kernel $\ker(\phi)$ is of finite type (see Lemma \ref{lem:ker}). So the action of $G$ on each of $C^{p-1}(L, \KK)$, $\ker(\delta^p)$, $\ker(\delta^{p+1})$, and, in the case of Corollary \ref{cor:juzv}, $\Img(\delta^p) = \ker(\delta^{p+1})$ are all of finite type and thus Theorem \ref{thm:period} may be applied to each.
\end{rem}

\begin{rem}
For actions of amenable groups, particularly $\Z$, there are numerous established connections between entropy and periodic points in various contexts. However we mention that for actions of non-amenable sofic groups, the only previously known connections between sofic entropy and periodic points were in the setting of principal algebraic actions, i.e. for fixed $f \in \Z[G]$ the natural shift action of $G$ on the Pontryagin dual $X_f$ of $\Z[G] / \Z[G] f$. Under various assumptions on $f$, work of Bowen \cite{B11}, Bowen--Li \cite{BL}, and Kerr--Li \cite{KL11a} has shown that the sofic entropy of the corresponding principal algebraic action is equal to the exponential growth rate of the number of periodic points (or the number of connected components of $\Fix{G_n}{X_f}$ where appropriate).
\end{rem}

Theorems \ref{thm:betti} and \ref{thm:period} and Corollary \ref{cor:juzv} are stated in terms of topological entropy, however the theorem below implies that under stronger assumptions those results also hold for measured entropies computed with respect to Haar probability measures. Recall that if $G$ acts on a compact metrizable group $H$ by continuous group automorphisms, then a point $h \in H$ is called \emph{homoclinic} if $g_n \cdot h \rightarrow 1_H$ for every injective sequence $(g_n)$ of elements of $G$. The set of homoclinic points forms a subgroup of $H$ called the \emph{homoclinic group} of $H$. For example, if $K$ is a discrete group and $G \acts K^G$ by left-shifts, then the homoclinic group of $K^G$ is the set of all functions $x : G \rightarrow K$ that satisfy $x(g) = 1_K$ for all but finitely many $g \in G$. For more on homoclinic points and their role in algebraic dynamics, see the nice discussion in \cite[Section 6]{Lind15}.

\begin{repthm}{thm:meas}[Haar measure and topological sofic entropy] 
Let $G$ be a sofic group, let $H$ be a profinite group, and let $G$ act on $H$ by continuous group automorphisms. If the homoclinic group of $H$ is dense, then
$$\meash{\Sigma}{G}{H}{\mathrm{Haar}} = \toph{\Sigma}{G}{H}$$
for every sofic approximation $\Sigma$ of $G$.
\end{repthm}

\begin{rem}
Both profiniteness and the property of having dense homoclinic group are preserved under taking continuous $G$-equivariant quotients. In particular, if $K$ is finite or profinite, then every continuous $G$-equivariant quotient of $K^G$ satisfies the assumptions of Theorem \ref{thm:meas}.
\end{rem}

\begin{rem}
In classical ergodic theory there is a close connection between homoclinic groups and the entropy of algebraic actions. To mention one of the more recent results, Chung and Li have shown that if $G$ is polycyclic-by-finite and the algebraic action $G \acts H$ is expansive, then the action has positive entropy if and only if the homoclinic group is non-trivial, and it has completely positive entropy (i.e. all measure-theoretic factors have positive entropy) if and only if the homoclinic group is dense \cite{ChLi}. Currently it is still unknown what role the homoclinic group may play in the development of the sofic entropy theory of algebraic actions.
\end{rem}

For algebraic actions of countable amenable groups, a result of Deninger states that the topological entropy is always equal to the measured entropy computed with Haar probability measure \cite{De06}. In the setting of sofic entropy, previous work of Bowen--Li \cite{BL}, Hayes \cite{H14}, and Kerr--Li \cite{KL11a} has uncovered instances of algebraic actions, specifically principal algebraic actions, where the topological entropy agrees with measured entropy computed using Haar probability measure. However, all prior instances of this were essentially discovered indirectly as the equalities were observed by computing both entropies explicitly. In \cite[Problem 7.7]{KL11a}, Kerr and Li asked if there is any direct way to see equality between topological and measured entropy for algebraic actions, and they asked under what conditions the two entropies coincide. Theorem \ref{thm:meas} is a partial answer to their question. We remark that the equality in Theorem \ref{thm:meas} is obtained directly, as we do not know the value of either of the two entropies.

We also obtain a second result asserting equality of measure-theoretic and topological entropy. This corollary follows from Theorem \ref{thm:period} after applying a result of Bowen \cite{B11}.

\begin{repcor}{cor:meas2}[Measured entropy and fixed points]
Let $G$ be a residually finite group, let $(G_n)$ be a chain of finite-index normal subgroups with trivial intersection, and let $\Sigma=\bigl(\sigma_n:G\to \Sym(G_n \backslash G)\bigr)$ be the associated sofic approximation. Let $K$ be a finite group and let $X \subseteq K^G$ be a $G$-invariant compact subgroup. If $X$ is of finite type, $\Fix{G_n}{X}$ converges to $X$ in the Hausdorff metric, and the Haar probability measure on $X$ is ergodic then
$$\meash{\Sigma}{G}{X}{\mathrm{Haar}} = \toph{\Sigma}{G}{X} = \limsup_{n \rightarrow \infty} \frac{1}{|G : G_n|} \cdot \log \Big| \Fix{G_n}{X} \Big|.$$
\end{repcor}

We mention that this corollary may be contrasted with either Theorem \ref{thm:meas} or Theorem \ref{thm:period}. In contrast to Theorem \ref{thm:meas}, the assumption that $\Fix{G_n}{X}$ converges to $X$ in the Hausdorff metric does not seem to imply that the homoclinic group of $X$ is dense. Thus the assumption of a dense homoclinic group is weakened (while several other new assumptions are imposed). In contrast to Theorem \ref{thm:period}, the above corollary relates measured entropy to periodic points while Theorem \ref{thm:period} relates topological entropy to periodic points. The assumptions of the above corollary are strictly stronger: we require the subgroups $G_n$ to be normal, require $\Fix{G_n}{X}$ to converge to $X$, and we require ergodicity of the Haar measure.

\vspace{3mm}
\noindent\emph{Update.} After this paper was completed, some new results appeared in the literature that imply that the Yuzvinsky addition formula fails (for some action) for all non-amenable sofic groups. This consequence has not yet been explicitly observed in the literature, so we record it in detail here. Specifically, in \cite{Bar16a} Bartholdi proves that for every countable non-amenable group $G$, there is a finite field $\KK$, $n \geq 2$, and a continuous $G$-equivariant linear map $\phi : (\KK^n)^G \rightarrow (\KK^{n-1})^G$ which is ``pre-injective'' (meaning, if $\phi(x) = \phi(y)$ and $x$ and $y$ differ at only finitely many coordinates, then $x = y$). In a second paper \cite{Bar16b}, given any field $\KK$, any countable group $G$, any $r, s \in \N$, and any continuous $G$-equivariant linear map $\psi : (\KK^r)^G \rightarrow (\KK^s)^G$, Bartholdi defined a dual, also a continuous $G$-equivariant linear map, $\psi^* : (\KK^s)^G \rightarrow (\KK^r)^G$, satisfying $(\psi^*)^* = \psi$, and he proved that $\psi$ is pre-injective if and only if $\psi^*$ is surjective. Applying Bartholdi's two results, it follows that for every countable non-amenable group $G$ there is a finite field $\KK$, $n \geq 2$, and a continuous $G$-equivariant linear map $\theta : (\KK^{n-1})^G \rightarrow (\KK^n)^G$ which is surjective (namely $\theta=\phi^*$). If $G$ is sofic, then for every sofic approximation $\Sigma$ of $G$ we have $\toph{\Sigma}{G}{(\KK^{n-1})^G} = (n-1) \log |\KK|$ and $\toph{\Sigma}{G}{(\KK^n)^G} = n \log|\KK|$. So $\theta$ violates the Yuzvinsky addition formula for topological sofic entropy for every choice of $\Sigma$. Furthermore, $\theta$, being a surjective group homomorphism, pushes the Haar measure on $(\KK^{n-1})^G$ forward to the Haar measure on $(\KK^n)^G$. Writing $u_\KK$ for the normalized counting measure on $\KK$, the Haar probability measure on $(\KK^r)^G$ is $(u_\KK^r)^G$ and $G \acts ((\KK^r)^G, (u_\KK^r)^G)$ is a Bernoulli shift of entropy $\meash{\Sigma}{G}{(\KK^r)^G}{(u_\KK^r)^G} = r \log |\KK|$ for every choice of $\Sigma$. Thus $\theta$ also violates the Yuzvinsky addition formula for measured sofic entropy for every choice of $\Sigma$.

\vspace{3mm}
\noindent\emph{Acknowledgments.} This work is the result of communication between the authors that was initiated at the conference ``Geometric and Analytic Group Theory'' held in Ventotene, Italy in 2013. The authors would like to thank the organizers for a wonderful conference. 
We are also grateful to Bruno Sevennec for valuable discussions.
The first author was supported by the CNRS, by the ANR project GAMME (ANR-14-CE25-0004) and by the LABEX MILYON (ANR-10-LABX-0070) of Université de Lyon, within the program ``Investissements d'Avenir" (ANR-11-IDEX-0007) operated by the French National Research Agency (ANR).
The second author was partially supported by NSF Graduate Student Research Fellowship grant DGE 0718128, NSF RTG grant 1045119, and ERC grant 306494.

\newpage
\section{Preliminaries} \label{sec:prelim}

\noindent\emph{Generating partitions and functions.} Let $G$ be a countably infinite group, and let $G \acts X$ be a Borel action on a standard Borel space $X$. A countable Borel partition $\cP$ of $X$ is \emph{generating} if for all $x \neq y \in X$ there is $g \in G$ such that $\cP$ separates $g \cdot x$ and $g \cdot y$. Generally, it will be more convenient for us to work instead with generating functions. A Borel function $\alpha : X \rightarrow \N$ is \emph{generating} if the partition $\cP = \{\alpha^{-1}(k) : k \in \N\}$ is generating. Equivalently, $\alpha$ is generating if for all $x \neq y \in X$ there is $g \in G$ with $\alpha(g \cdot x) \neq \alpha(g \cdot y)$.

Suppose now that $\mu$ is a $G$-invariant Borel probability measure on $X$. A countable Borel partition $\cP$ (or a Borel function $\alpha : X \rightarrow \N$) is \emph{generating (mod $\mu$)} if there is a $G$-invariant conull set $X' \subseteq X$ such that $\cP \res X'$ (respectively $\alpha \res X'$) is generating for $G \acts X'$. Although this is not the standard definition of generating, it is equivalent (see for example \cite[Lemma 2.1]{S12}). When there is no danger of confusion, we will simply say generating instead of generating mod $\mu$.

\vspace{3mm}
\noindent\emph{Bernoulli shifts and subshifts.} For a metrizable space $K$ and a countable or finite set $V$, we write $K^{V} = \prod_{v \in V} K$ for the set of all functions $x : V \rightarrow K$ equipped with the product topology or, equivalently, the topology of point-wise convergence. When $V = G$, we call $K^G$ the $K$-Bernoulli shift over $G$ and equip it with the $G$-shift-action
\begin{equation*}
\textrm{for } x \in K^G \textrm{ and } g \in G, \qquad (g \cdot x)(h) = x(g^{-1} h), \textrm{ for all } h \in G.
\end{equation*}
The Bernoulli shift $K^G$ comes with the \emph{tautological generating function} $\alpha$ defined by evaluation at the identity element $1_G$:
\begin{equation*}
\alpha:\begin{pmatrix}
K^G &\rightarrow &K\\
x & \mapsto & x(1_G)
\end{pmatrix}
\end{equation*}
Note that $\alpha(g \cdot x) = (g \cdot x)(1_G) = x(g^{-1})$. For our entropy computations it will be convenient to explore the values of a function $x \in K^G$ by considering the values $\alpha(g \cdot x)$ for $g \in G$. Observe that the collection of values $\alpha(g \cdot x)$, for all $g \in G$, uniquely determines $x$ (i.e. $\alpha$ is a generating function).

A subset $X \subseteq K^G$ is a \emph{subshift} if it is closed and $G$-invariant. It is a subshift \emph{of finite type} if there is a finite set $W \subseteq G$ and a closed set of patterns $P \subseteq K^W$ such that
$$X = \{x \in K^G : \forall g \in G \ \exists p \in P \ \forall w \in W \ \alpha(w g \cdot x) = p(w) \}.$$
In other words, $X$ is the largest subshift with the property that, for every $x \in X$, the map $w \mapsto \alpha(w \cdot x)$ lies in $P$. We will call such a set $W$ a \emph{test window} for $X$. If $K$ is a group, then $X \subseteq K^G$ is an \emph{algebraic subshift} if $X$ is a subshift and is also a subgroup of $K^G$.

\vspace{3mm}
\noindent\emph{Cost.} 
Let $(X, \mu)$ be a standard probability space and let $G \acts (X, \mu)$ be a measure-preserving action. Denote by $\calR_G^X$ the equivalence relation given by the orbits:
$$\calR_G^X = \{(x, y) \in X \times X : \exists g \in G \ g \cdot x = y\}.$$
A \emph{graphing} on $X$ is a countable collection $\Phi = \{\phi_i : A_i \rightarrow B_i : i \in I\}$ of partial isomorphisms, i.e. of Borel bijections between Borel subsets $A_i, B_i \subseteq X$. Let $\calR_\Phi$ denote the smallest equivalence relation containing all the pairs $(x, \phi_i(x))$ for $i \in I$ and $x \in A_i$. 
We call $\Phi$ a \emph{graphing of} $\calR_G^X$ if $\calR_G^X = \calR_\Phi$ on a co-null set of $X$ (this forces the $\phi_i$ to preserve the measure), and the \emph{cost of} $\Phi$ is defined to be $\sum_{i \in I} \mu(A_i)$. 
The \emph{cost} of $\calR_G^X$ is defined to be the infimum of the costs of all graphings $\Phi$ of $\calR_G^X$.
The \emph{cost} of a countable group $G$, denoted $\cost(G)$, is defined to be the infimum of the costs of all $\calR_G^Y$ for free probability-measure-preserving actions $G \acts (Y, \nu)$. Similarly, the \emph{supremum-cost} of $G$, denoted $\costsup(G)$, is defined in the same manner, except one takes the supremum over all free probability-measure-preserving actions of $G$.
When $G$ is finitely generated, the supremum-cost is realized by any non trivial Bernoulli shift action (Ab{\'e}rt-Weiss~\cite{AW13}).

\vspace{3mm}
\noindent\emph{Sofic groups}

\begin{defn}[Sofic approximation] \label{defn:sofic}
A countable group $G$ is \emph{sofic} if there is a sequence of finite sets $D_n$ and maps $\sigma_n : G \rightarrow \Sym(D_n)$ (not necessarily homomorphisms) such that $\sigma_n(1_G) = 1_{\Sym(D_n)}$ and:
\begin{enumerate}
\item[\rm (i)] (asymptotically free) for every $g \in G$ with $g \neq 1_G$ we have
$$\lim_{n \rightarrow \infty} \frac {|\{ i \in D_n : \sigma_n(g)(i) \neq i\}|}{\vert D_n \vert} = 1.$$

\item[\rm (ii)] (asymptotically an action) for every $g, h \in G$ we have
$$\lim_{n \rightarrow \infty} \frac {|\{ i \in D_n : \sigma_n(g) \circ \sigma_n(h)(i) = \sigma_n(g h)(i)\}|}{\vert D_n \vert} = 1.$$
\end{enumerate}

A sequence $\Sigma = \{\sigma_n : n \in \N\}$ with $\vert D_n \vert \rightarrow \infty$ and satisfying the above conditions will be called a \emph{sofic approximation} to $G$. Of course, if $G$ is infinite and the above conditions are satisfied, then $\vert D_n \vert$ must tend to infinity. For a nice survey of sofic groups and other equivalent definitions of soficity, see \cite{P08}.
\end{defn}

\begin{rem}
The class of sofic groups contains the (countable) amenable groups and, as discussed below, the residually finite groups (such as free groups and linear groups). It is a well known open problem whether all countable groups are sofic.
\end{rem}

\vspace{3mm}
\noindent\emph{Residually finite and profinite groups}.
An important class of sofic groups are the residually finite groups. A countable group $G$ is \emph{residually finite} if 
the intersection of all its normal subgroups of finite index is trivial.
 Equivalently, a group $G$ is residually finite if and only if it admits a sofic chain, as defined below:
\begin{defn}[Sofic chain]\label{def:sofic chain}
A decreasing sequence $(G_n)$ of finite-index (not necessarily normal) subgroups of a countable group $G$ is called a \emph{chain}.
We call it a \emph{sofic chain} if it satisfies moreover the following asymptotic condition:
\begin{equation}\label{eq:sofic cond. for chains}
\forall g \in G \setminus \{1_G\} \qquad \lim_{n \rightarrow \infty} \ \frac{{|\{G_n u : u \in G, \textrm{ s. t. } \ g \in u^{-1} G_n u\}|}}{[G:G_n]} = 0.
\end{equation}
In particular, a normal chain (i.e. a chain of normal subgroups) with trivial intersection is a sofic chain.
\end{defn}
Consider a countable group $G$ and a chain $(G_n)$. Right multiplication by the inverse defines a natural left action of $G$ on the set $D_n:=G_n \backslash G = \{G_n u : u \in G\}$ of right-cosets of $G_n$ defined by $g * G_n u = G_n u g^{-1}$. This action $G \acts D_n$ produces a homomorphism $\sigma_n : G \rightarrow \Sym(D_n)$, from $G$ to the symmetric group on the set $D_n$.
\begin{lem}
The sequence $\Sigma = \{\sigma_n : G \rightarrow \Sym(G_n \backslash G) : n \in \N\}$ is a sofic approximation to $G$ if and only if the chain $(G_n)$ is a sofic chain.
\end{lem}

\begin{proof}
Since each $\sigma_n$ is a homomorphism,  condition (ii) of Definition \ref{defn:sofic} is automatically satisfied. 
As for condition (i), it coincides with condition~(\ref{eq:sofic cond. for chains}) since 
$g$ fixes $G_n u$ if and only if $g\in u^{-1} G_n u$.
\end{proof}
\begin{rem}\label{rem:no base point in D}
In this framework, $D_n=G_n \backslash G$ is considered just as a homogeneous $G$-space, i.e. a transitive $G$-space without any particular choice of an origin or a root. It follows that $\sigma_n$ only retains the conjugacy class of $G_n$ instead of $G_n$ itself. Indeed, each choice of a root $\delta\in D_n$ delivers its stabilizer $\mathrm{Stab}_{G\acts D_n}(\delta)$. Changing the root modifies the stabilizer by a conjugation.
\end{rem}
Observe  that condition~(\ref{eq:sofic cond. for chains}) is equivalent to 
the so-called \emph{Farber condition}:
\begin{equation*}
\forall g \in G \setminus \{1_G\} \qquad \lim_{n \rightarrow \infty} \ \frac{|\{u^{-1} G_n u : u \in G, \ g \in u^{-1} G_n u\}|}{|\{u^{-1} G_n u : u \in G\}|} = 0,
\end{equation*}
introduced in \cite{Far98} and analyzed in~\cite{BG04} (see in particular \cite[prop.~2.6]{BG04}).

\medskip
A topological group $K$ is \emph{profinite} if it is isomorphic (as a topological group) to an inverse limit of discrete finite groups. More specifically, $K$ is profinite if there is a sequence $K_m$ of finite groups and homomorphisms $\beta_{m, \ell} : K_\ell \rightarrow K_m$ (for $\ell \geq m$) and $\beta_m : K \rightarrow K_m$ such that
\begin{enumerate}
\item[\rm (i)] $\beta_{m, \ell} \circ \beta_\ell = \beta_m$ for all $\ell \geq m$;
\item[\rm (ii)] for all $k \neq k' \in K$ there is an $m$ with $\beta_m(k) \neq \beta_m(k')$;
\item[\rm (iii)] $K$ has the weakest topology making every homomorphism $\beta_m$ continuous.
\end{enumerate}
In this situation we write $K = \varprojlim K_m$ (the maps $\beta_m$ will be clear from the context). It is immediate from the definition that profinite groups must be compact, Hausdorff, metrizable, and totally disconnected. In particular, every profinite group admits a (unique) Haar probability measure.

\vspace{3mm}
\noindent\emph{$\ell^2$-Betti numbers.} Let $L$ be a simplicial complex and let $G \acts L$ be a free cocompact action. For $p \geq 0$ let $C^p(L, \C)$ be the space of $p$-cochains over $\C$ and let $\delta^{p+1} : C^p(L, \C) \rightarrow C^{p+1}(L, \C)$ be the coboundary map. The coboundary map $\delta^{p+1}$ restricts to a continuous operator $\delta^{p+1}_{(2)} : C^p_{(2)}(L) \rightarrow C^{p+1}_{(2)}(L)$ defined on the Hilbert space of square-summable $p$-cochains $C^p_{(2)}(L)$. The \emph{reduced $\ell^2$-cohomology} group $\overline{H}^p_{(2)}(L):=\ker(\delta^{p+1}_{(2)}) / \overline{\Img(\delta^{p}_{(2)})}$ is defined as the quotient of the kernel of $\delta^{p+1}_{(2)}$ by the closure of the image of $\delta^{p}_{(2)}$. The \emph{$p^{\text{th}}$ $\ell^2$-Betti number of the action $G \acts L$} is defined as the von Neumann $G$-dimension of the Hilbert $G$-module $\overline{H}^p_{(2)}(L)$
$$\beta_{(2)}^p(L : G) := \dim_G \overline{H}^p_{(2)}(L).$$
If $L$ is $n$-connected, then for all $p \leq n$ the $\ell^2$-Betti number $\beta_{(2)}^p(L : G)$ depends only on $G$ and not on $L$, and is called the \emph{$p^{\text{th}}$ $\ell^2$-Betti number of $G$}, denoted $\beta_{(2)}^p(G)$.

In general, we must consider the case of a free action $G \acts L$ on a simplicial complex $L$ where the action is possibly not cocompact.
We write $L$ as an increasing union of $G$-invariant cocompact subcomplexes $L=\cup_{i} \nearrow L_i$.
The inclusions $L_i\subset L_j$, $i\leq j$, induce $G$-equivariant morphisms $C^{p}_{(2)}(L_j)\to C^{p}_{(2)}(L_i)$ (restriction)
for cochains and for reduced cohomologies $\bar{H}^{p}_{(2)}(L_j)\to \bar{H}^{p}_{(2)}(L_i)$. The $p$-th $\ell^2$-Betti number of the action of $G$ on $L$ is defined as the (increasing) limit in $i$ of the (decreasing) limit in $j$ of the von Neumann $G$-dimension of the closure of the image of these morphisms: 
\begin{equation*}
\beta^{p}_{(2)}(L : G):=\lim_{i\to \infty}\nearrow \lim_{i\leq j, j\to \infty} \searrow  \dim_{G} \overline{\Img}\, \bigl(\bar{H}^{p}_{(2)}(L_j)\to \bar{H}^{p}_{(2)}(L_i)\bigr).
\end{equation*}
It does not depend on the choice of the exhaustion of $L$.
Indeed, it remains unchanged under subsequences and for any two exhaustions, there is a third one containing subsequences in both.
Moreover, if $L$ is $n$-connected then we get a number that is independent of $L$ for all $p \leq n$, and this number is called the \emph{$p^{\text{th}}$ $\ell^2$-Betti number of $G$}:
\begin{equation*}
\beta^{p}_{(2)}(G):=\beta^{p}_{(2)}(L : G).
\end{equation*}

\medskip
We will make use of the following fundamental theorem. This result is due to L\"{u}ck \cite{Luc94} for normal chains and due to Farber \cite{Far98} in the generality stated here.

\begin{thm}[L\"{u}ck approximation theorem] \label{thm:luck}
Let $L$ be a connected simplicial complex and let $G \acts L$ be a free cocompact action. If $G$ is residually finite and $(G_n)$ is a sofic chain then
$$\beta_{(2)}^p(L : G) = \lim_{n \rightarrow \infty} \frac{\dim_\C H^p(G_n \backslash L, \C)}{|G : G_n|}.$$
\end{thm}

We will also use a result of Thom \cite[Th. 4.2]{Tho08} (improving a similar result by Elek and Szab\'{o} \cite{ES04}) that extends the L\"{u}ck approximation theorem to the sofic framework (see formula~(\ref{eqn:dim3})).

\newpage
\section{Definition of topological sofic entropy} \label{sec:tent}

Let $G$ be a sofic group, let $\Sigma = (\sigma_n : G \rightarrow \Sym(D_n))$ be a sofic approximation to $G$, and let $G \acts X$ be a continuous action on a compact metrizable space $X$. Let $\rho : X \times X \rightarrow [0, \infty)$ be a continuous \emph{pseudo-metric} on $X$ that is \emph{generating}; i.e. $\rho$ is symmetric, satisfies the triangle inequality, and for all $x \neq y \in X$ there is $g \in G$ with $\rho(g \cdot x, g \cdot y) > 0$. For maps $\phi, \phi' : D_n \rightarrow X$ we define
\begin{equation*}
\rho_2(\phi, \phi') := \left( \frac{1}{\vert D_n \vert} \sum_{i\in D_n} \rho(\phi(i), \phi'(i))^2 \right)^{1 / 2},
\hskip40pt 
\rho_\infty(\phi, \phi') := \max_{i \in D_n} \rho(\phi(i), \phi'(i)).
\end{equation*}
For finite $F \subseteq G$ and $\delta > 0$, let
\begin{equation*}
\Map(\rho, F, \delta, \sigma_n):=\Bigl\{\phi : D_n \rightarrow X:
\forall f\in F,\ \ \rho_2 \Big( \phi \circ \sigma_n(f), \ f \cdot \phi \Big) < \delta\Bigr\}.
\end{equation*}
Roughly speaking, it is the collection of maps $\phi \in X^{D_n}$ that are {\em almost-equivariant} 
(up to $\delta$, under the ``pseudo-action'' $\sigma_n$ restricted to the finite set $F\subset G$).
Finally, we let 
\begin{equation}
N_\kappa(\Map(\rho, F, \delta, \sigma_n), \rho_\infty)
\end{equation}
denote the maximum cardinality of a \emph{$(\rho_\infty, \kappa)$-separated set}; i.e. a set such that for every pair of elements $\phi$ and $\phi'$ we have $\rho_\infty(\phi, \phi') \geq \kappa$. The $\Sigma$-entropy of $G \acts X$ is then defined to be
$$\toph{\Sigma}{G}{X} = \sup_{\kappa > 0} \inf_{\delta > 0} \inf_{\substack{F \subseteq G \\ F \text{ finite}}} \limsup_{n \rightarrow \infty} \ \frac{1}{\vert D_n \vert} \cdot \log N_\kappa \Big( \Map(\rho, F, \delta, \sigma_n), \rho_\infty \Big).$$
This definition comes from \cite[Def. 2.3]{KL}. Kerr and Li \cite{KL} proved that the value $\toph{\Sigma}{G}{X}$ does not depend on the choice of the continuous generating pseudo-metric $\rho$ (though it may depend on $\Sigma$). Observe that either $\toph{\Sigma}{G}{X} \geq 0$ or else $\toph{\Sigma}{G}{X} = -\infty$.

We mention for completeness that when the acting sofic group is amenable, the topological sofic entropy is identical to classical topological entropy for all choices of $\Sigma$ \cite{KL}.

We will soon present an alternative formula for topological entropy, but first we need the following technical lemma. Below we write $\lfloor t \rfloor$ for the greatest integer less than or equal to~$t$.

\begin{lem} \label{lem:choose}
If $0 < \kappa < 1$ then
$$\lim_{n \rightarrow \infty} \frac{1}{n} \cdot \log \binom{n}{\lfloor \kappa n \rfloor} = - \kappa \cdot \log(\kappa) - (1 - \kappa) \cdot \log(1 - \kappa).$$
\end{lem}

\begin{proof}
This follows quickly from Stirling's formula.
\end{proof}

In the remainder of this section, we present an alternate formula for entropy and establish a few lemmas that will be useful for the types of actions studied in this paper.

Let $G \acts X$ be a continuous action on a compact metrizable space. Let $K$ be a finite discrete set and $\alpha : X \rightarrow K$ be a continuous function (equivalently, a clopen partition). The $G$-action delivers, for $g\in G$, the shifted functions $x\mapsto \alpha(g 
\cdot x)$.
Given a finite subset $F$ of the group $G$ and a point $x\in X$, the \emph{$F$-itinerary} of $x$ (i.e. the following the pattern)
\begin{equation*}
\Bigl(\begin{array}{lll} F & \to  & K \\f & \mapsto & \alpha(f \cdot x) \end{array}\Bigr)\in K^F
\end{equation*}
is the sequence of pieces traveled by the $F$-iterates of $x$.

The topological entropy of the action, roughly speaking, is the exponential growth rate of the number of ways in which the pseudo-action $\sigma_n : G \rightarrow \Sym(D_n)$ on $D_n$ can mimic the $G$-action on $X$. In the alternative formula for entropy below, we will be interested in partitions $a:D_n\to K$ that are plausible in the sense that the pseudo-$F$-itinerary 
\begin{equation*}
p:\Bigl(\begin{array}{lll} F & \to  & K \\f & \mapsto & a(\sigma_n(f)(i)) \end{array}\Bigr)\in K^F
\end{equation*}
of the points $i\in D_n$ are  \emph{witnessed in $X$}, i.e. coincide with a genuine $F$-itinerary in $X$.
Indeed, we will require this to happen for most of the points $i\in D_n$, and we will then count the number of such good partitions $a\in K^{D_n}$.
We say that a property $\mathcal{P}$ is satisfied for \emph{$\epsilon$-almost every $i\in D_n$} if 
\begin{equation*}
\frac{1}{\vert D_n\vert} \Bigl\vert \{i\in D_n: i \textrm{ satisfies } \mathcal{P}\}\Bigr\vert \geq 1-\epsilon.
\end{equation*}
We define:
\begin{equation*}
\begin{split}\AM_X(\alpha, F, \epsilon, \sigma_n):=\Bigl\{a\in K^{D_n}:\textrm{for } \epsilon&\textrm{-almost every } i\in D_n \textrm{ the pattern }\\
& p:\Bigl(\begin{array}{lll} F & \to  & K \\f & \mapsto & a(\sigma_n(f)(i)) \end{array}\Bigr) \textrm{ is witnessed in } X\Bigr\}.
\end{split}
\end{equation*}

\begin{prop}[Clopen partition] \label{prop:topent2}
Let $G \acts X$ be a continuous action on a compact metrizable space. If $K$ is finite and $\alpha : X \rightarrow K$ is a continuous generating function, then
$$\toph{\Sigma}{G}{X} = \inf_{\epsilon > 0} \inf_{\substack{F \subseteq G\\F \text{ finite}}} \limsup_{n \rightarrow \infty} \frac{1}{\vert D_n \vert} \cdot \log \Big| \AM_X(\alpha, F, \epsilon, \sigma_n) \Big|.$$
\end{prop}

In this paper we will be mostly interested in the case where $K$ is finite, however we will occasionally need to work in the more general setting where $K = \varprojlim K_m$. For example, we will need to work in the profinite setting in order to prove Theorem \ref{thm:cost}. Below in Lemma \ref{lem:topent} we present a formula for topological entropy when $K = \varprojlim K_m$ is an inverse limit of finite spaces, and Proposition \ref{prop:topent2} will easily be seen as an immediate special case of that lemma.

\begin{assume}
For the remainder of this section, we let $K = \varprojlim K_m$ be an inverse limit of finite spaces $K_m$ with maps $\beta_m : K \rightarrow K_m$ and $\beta_{m, \ell} : K_\ell \rightarrow K_m$, $\ell \geq m$, satisfying $\beta_{m, \ell}\circ \beta_{\ell}=\beta_m$.
We place on $K$ the weakest topology making each map $\beta_m$ continuous. We also let $G$ be a sofic group with sofic approximation $\Sigma = (\sigma_n : G \rightarrow \Sym(D_n))$.
\end{assume}

Let $G \acts X$ be a continuous action on a compact metrizable space, and suppose that $\alpha : X \rightarrow K$ is a continuous generating function. For $\ell \in \N$, set $\alpha_\ell = \beta_\ell \circ \alpha$. For $\ell \in \N$, finite $F \subseteq G$, and $Y \subseteq X$, call a pattern $p \in K_\ell^F$ \emph{witnessed in $Y$} if there is some $y \in Y$ with $\alpha_\ell(f \cdot y) = p(f)$ for all $f \in F$.

We define:
\begin{equation}\label{eq:def M Y}
\begin{split}\AM_Y(\alpha_{\ell}, F, \epsilon, \sigma_n):=\Bigl\{a\in K_{\ell}^{D_n}:\textrm{for } \epsilon&\textrm{-almost every } i\in D_n \textrm{ the pattern }\\
& p:\Bigl(\begin{array}{lll} F & \to  & K_{\ell} \\f & \mapsto & a(\sigma_n(f)(i)) \end{array}\Bigr) \textrm{ is witnessed in } Y\Bigr\}.
\end{split}
\end{equation}
Given $\ell \geq m$, every partition $a\in K_{\ell}^{D_n}$ reduces through $\beta_{m,\ell}$ to a partition 
\begin{equation*}
\beta_{m,\ell}\circ a:\Bigl(\begin{array}{ccccc}  D_n & \overset{a}{\longrightarrow}  & K_{\ell} & \overset{\beta_{m, \ell}}{\longrightarrow} & K_m\end{array}\Bigr)\in K_{m}^{D_n}.
\end{equation*}

The lemma below contains the formula for entropy that we will use.

\begin{lem} \label{lem:topent}
With the notation above, we have
\begin{equation} \label{eqn:topent}
\toph{\Sigma}{G}{X} = \sup_{m \in \N} \inf_{\ell \geq m} \inf_{\epsilon > 0} \inf_{\substack{F \subseteq G\\ F \text{ finite}}} \limsup_{n \rightarrow \infty} \ \frac{1}{\vert D_n \vert} \cdot \log \Big| \beta_{m, \ell} \circ \AM_X(\alpha_\ell, F, \epsilon, \sigma_n) \Big|.
\end{equation}
\end{lem}

\begin{proof}
Since $\alpha$ is a continuous generating function, the function $\rho : X \times X \rightarrow \R$ defined by
$$\rho(x, y) = \begin{cases}
2^{-m} & \text{if } m \text{ is least with } \alpha_m(x) \neq \alpha_m(y) \\
0 & \text{otherwise}
\end{cases}$$
is a continuous generating pseudo-metric.
\vspace{0.2in}

\noindent\underline{Claim:} If $\kappa \geq 2^{-m}$ and $2^{2 \ell} \cdot |F| \cdot \delta^2 < \epsilon$, then
$$N_\kappa( \Map(\rho, F, \delta, \sigma_n), \rho_\infty ) \leq |\beta_{m, \ell} \circ \AM_X(\alpha_\ell, F, \epsilon, \sigma_n)|.$$

\noindent{\it Proof of Claim:} 
Define
\begin{equation*}
r_{\ell} :\left(\begin{array}{ccl} X^{D_n}& \to & K_{\ell}^{D_n} 
\\
\phi & \mapsto & r_{\ell}(\phi)=\alpha_{\ell}\circ \phi\end{array}\right)
\end{equation*}
If $\phi, \phi' \in \Map(\rho, F, \delta, \sigma_n)$ and $\rho_\infty(\phi, \phi') \geq 2^{-m}$, then by definition there is $i \in D_n$ with $\alpha_m \circ \phi(i) \neq \alpha_m \circ \phi'(i)$ and hence $\beta_{m, \ell}\circ r_{\ell}(\phi) \neq \beta_{m, \ell} \circ r_{\ell}(\phi')$. Thus $\phi\mapsto \beta_{m, \ell} \circ r_{\ell}(\phi)$ is injective whenever restricted to a $(\rho_\infty, \kappa)$-separated set. Thus the claim will follow once we show that $r_{\ell}(\phi) \in \AM_X(\alpha_\ell, F, \epsilon, \sigma_n)$.

Fix $\phi \in \Map(\rho, F, \delta, \sigma_n)$. For fixed $f \in F$ we have
$$\Bigl\vert\{i\in D_n : \alpha_\ell \circ \phi \, (\sigma_n(f)(i) ) \neq \alpha_\ell ( f \cdot \phi(i) ) \}\Bigr\vert \leq 2^{2 \ell} \vert D_n \vert \cdot \rho_2(\phi \circ \sigma_n(f), f \cdot \phi)^2 < 2^{2 \ell} \vert D_n \vert \cdot \delta^2.$$
Therefore the set
$$B = \{i\in D_n : \exists f \in F \ \alpha_\ell \circ \phi\, (\sigma_n(f)(i) ) \neq \alpha_\ell ( f \cdot \phi(i) )\}$$
has cardinality less than $\epsilon \cdot \vert D_n \vert$. If $i \not\in B$ then for every $f \in F$ 
$$r_{\ell}(\phi)(\sigma_n(f)(i)) = \alpha_\ell \circ \phi ( \sigma_n(f)(i) ) = \alpha_\ell (f \cdot \phi(i)).$$
Thus for $i \not\in B$ the pattern $f \in F \mapsto r_{\ell}(\phi)(\sigma_n(f)(i))$ is witnessed since $\phi(i) \in X$. $\blacksquare$
\vspace{0.2in}

\noindent\underline{Claim:} If $1_G \in F$, $\kappa \leq 2^{-m}$, and $2^{-\ell - 1} < \delta$ then
$$\inf_{\epsilon > 0} \limsup_{n \rightarrow \infty} \frac{1}{\vert D_n \vert} \cdot \log |\beta_{m,\ell}^{D_n} \circ \AM_X(\alpha_\ell, F, \epsilon, \sigma_n)| \leq \limsup_{n \rightarrow \infty} \frac{1}{\vert D_n \vert} \cdot \log N_\kappa( \Map(\rho, F, \delta, \sigma_n), \rho_\infty ).$$

\noindent{\it Proof of Claim:}
For each witnessed pattern $p \in K_\ell^F$ fix $x_p \in X$ with $\alpha_\ell (f \cdot x_p) = p(f)$ for all $f \in F$. For unwitnessed patterns $p \in K_\ell^F$, fix an arbitrary point $x_p \in X$. 
Define
\begin{equation*}
\phi :\left(\begin{array}{ccl} K_{\ell}^{D_n}& \to & X^{D_n} 
\\
a & \mapsto & \phi_{a}: (i\mapsto x_{p(a,i)})\end{array}\right)
\end{equation*}
by associating to $a$ and each $i\in D_n$ first the $F$-pattern 
\begin{equation*}
p(a,i): f\mapsto a(\sigma_n(f)(i))
\end{equation*}
and then the point $x_{p(a,i)}$ chosen for this pattern. Let $W_a$ be the set of $i$ for which $p(a,i)$ is witnessed.

Assume now $a \in \AM_X(\alpha_\ell, F, \epsilon, \sigma_n)$. Then $|W_a| \geq (1 - \epsilon) \cdot \vert D_n \vert$. 
It is immediate from the definitions that for all $i \in W_a$ and all $f \in F$
\begin{equation*}
\alpha_\ell (f \cdot \phi_a(i)) = \alpha_\ell (f \cdot x_{p(a,i)}) = p(a,i)(f) = a(\sigma_n(f)(i)).
\end{equation*}
In particular, 
\begin{equation}\label{i in W-a}
\alpha_\ell \circ \phi_a(i) = a(i)  \textrm{ for } i \in W_a
\end{equation}
since $1_G \in F$ and $\sigma_n(1_G)=1_{\Sym(D_n)}$. So if both $i, \sigma_n(f)(i) \in W_a$ then $\rho(\phi_a \circ \sigma_n(f)(i), f \cdot \phi_a(i))\leq 2^{-\ell-1}$ since
$$\alpha_\ell (\phi_a \circ \sigma_n(f)(i)) \overset{\textrm{if }\sigma_n(f)(i) \in W_a}{=} a(\sigma_n(f)(i)) \overset{\textrm{if } i \in W_a}{=} \alpha_\ell (f \cdot \phi_a(i)).$$
For each $f\in F$, both $i, \sigma_n(f)(i) \in W_a$ for $(2\epsilon)$-almost every $i\in D_n$.
This implies $$\rho_2(\phi_a \circ \sigma_n(f), f \cdot \phi_a) \leq \sqrt{2^{-2 \ell - 2} + \frac{1}{\vert D_n \vert} \cdot 2 \epsilon \cdot \vert D_n \vert}.$$
So when $\epsilon$ is sufficiently small $\rho_2(\phi_a \circ \sigma_n(f), f \cdot \phi_a) < \delta$ and hence $\phi_a \in \Map(\rho, F, \delta, \sigma_n)$.

If $a, a' \in \AM_X(\alpha_\ell, F, \epsilon, \sigma_n)$ and $\rho_\infty(\phi_a, \phi_{a'}) < \kappa \leq 2^{-m}$ (*)  (thus $\alpha_m(\phi_a(i)) = \alpha_m(\phi_{a'}(i))$ for every $i$), then for all $i \in W_a \cap W_{a'}$
$$\beta_{m, \ell} \circ a(i) \overset{(\ref{i in W-a})}{=} \underbrace{\beta_{m, \ell} \circ \alpha_{\ell}}_{\alpha_{m}} (\phi_a(i)) 
\overset{(*)}{=}
 \underbrace{\beta_{m, \ell} \circ \alpha_\ell}_{\alpha_{m}} (\phi_{a'}(i)) \overset{(\ref{i in W-a})}{=} \beta_{m, \ell} \circ a'(i).$$
Note that $|W_a \cap W_{a'}| \geq (1 - 2 \epsilon) \cdot \vert D_n \vert$, so that 
$$\frac{1}{\vert D_n \vert} \cdot |\{i \in D_n : \beta_{m, \ell} \circ a(i) \neq \beta_{m, \ell} \circ a'(i)\}| \leq 2 \epsilon.$$

For $b \in K_m^{D_n}$ let $B(b, 2 \epsilon)$ be the set of those $b' \in K_m^{D_n}$ with $\frac{1}{\vert D_n \vert} \cdot \{i \in D_n : b(i) \neq b'(i)\}| \leq 2 \epsilon$. Note that $|B(b, 2 \epsilon)| \leq \binom{\vert D_n \vert}{\lfloor 2 \epsilon \vert D_n \vert \rfloor} \cdot |K_m|^{2 \epsilon \vert D_n \vert}$. Pick a maximal $(\rho_{\infty},\kappa)$-separated set $\phi_{a_1}, \phi_{a_2}, \cdots, \phi_{a_R}$ with $a_j\in \AM_X(\alpha_\ell, F, \epsilon, \sigma_n)$.
Then $R\leq N_\kappa(\Map(\rho, F, \delta, \sigma_n), \rho_\infty)$.
Each $a_j$ gives rise to a $b_j:=\beta_{m,\ell}\circ a_j$.
Every extra $a \in \AM_X(\alpha_\ell, F, \epsilon, \sigma_n)$ satisfies $\rho_{\infty}(\phi_a,\phi_{a_j})<\kappa$ for some $j$, so that the associated $b=\beta_{m,\ell}\circ a$ belongs to the set $B(b_j,2\epsilon)$. Thus, the $B(b_j,2\epsilon)$ form a covering of $\beta_{m,\ell} \circ \AM_X(\alpha_\ell, F, \epsilon, \sigma_n)$.
It follows that
$$N_\kappa(\Map(\rho, F, \delta, \sigma_n), \rho_\infty) \geq \binom{\vert D_n \vert}{\lfloor 2 \epsilon \vert D_n \vert \rfloor}^{-1} \cdot |K_m|^{- 2 \epsilon \vert D_n \vert} \cdot |\beta_{m,\ell}^{D_n} \circ \AM_X(\alpha_\ell, F, \epsilon, \sigma_n)|.$$
Now apply Lemma \ref{lem:choose} and let $\epsilon$ tend to $0$. $\blacksquare$
\vspace{0.2in}

To complete the proof, set $h = \toph{\Sigma}{G}{X}$ and let $h'$ be equal to the right-hand side of (\ref{eqn:topent}). It is important to note that monotonicity considerations imply that in the definitions of $h$ and $h'$ one can require $1_G \in F$ and replace the suprema and infima with limits where $\kappa, \delta, \epsilon \rightarrow 0$ and $\ell, m \rightarrow \infty$. Also observe that one can exchange consecutive infima like $\inf_{\epsilon > 0}$ and $ \inf_{\substack{F \subseteq G\\ F \text{ finite}}}$. Using the inequality from the first claim, first take the infimum over $\delta$ and $F$, then take the infimum over $\epsilon$ and $\ell$ and the supremum over $m$. Finally, take the supremum over $\kappa$ to obtain $h \leq h'$. Now using the inequality from the second claim, we first take the infimum over $F$ and $\ell$, then take the infimum over $\delta$ and supremum over $\kappa$. Finally, taking the supremum over $m$ gives $h' \leq h$.
\end{proof}

If $G \acts X$ is a continuous action on a compact metrizable space $X$ and $\alpha : X \rightarrow K$ is a continuous generating function, then $\alpha$ naturally produces a $G$-equivariant homeomorphism between $X$ and a compact invariant set $X' \subseteq K^G$. It is natural in this setting to approximate $X'$ by slightly larger compact subsets of $K^G$. Such approximations provide a means to compute the entropy of $G \acts X$.

\begin{lem} \label{lem:tapprox}
Let $X \subseteq K^G$ be a $G$-invariant compact set. If $(Y_r)_{r \in \N}$ is a decreasing sequence of compact subsets of $K^G$ with
$$X = \bigcap_{r \in \N} \bigcap_{g \in G} g \cdot Y_r,$$
then
\begin{equation} \label{eqn:tapprox}
\toph{\Sigma}{G}{X} = \sup_{m \in \N} \inf_{\ell \geq m} \inf_{r \in \N} \inf_{\substack{F \subseteq G\\F \text{ finite}}} \inf_{\epsilon > 0} \limsup_{n \rightarrow \infty} \frac{1}{\vert D_n \vert} \cdot \log \Big| \beta_{m, \ell}^{D_n} \circ \AM_{Y_r}(\alpha_\ell, F, \epsilon, \sigma_n) \Big|.
\end{equation}
\end{lem}

\begin{proof}
Since $X \subseteq Y_r$ we have
$$\AM_X(\alpha_\ell, F, \epsilon, \sigma_n) \subseteq \AM_{Y_r}(\alpha_\ell, F, \epsilon, \sigma_n).$$
Therefore $\toph{\Sigma}{G}{X}$ is bounded above by the right-hand side of (\ref{eqn:tapprox}).

For $r, c \in \N$ and finite $T \subseteq G$ define the set
$$V_r^{c, T} = \{ z \in K^G : \exists y \in Y_r \ \forall t \in T \ \alpha_c(t \cdot z) = \alpha_c(t \cdot y)\}.$$
Note that each set $V_r^{c, T}$ is clopen and contains $Y_r$. Since $K^G \setminus Y_r$ is open, we have
$$Y_r = \bigcap_{c \in \N} \bigcap_{T \subseteq G} V_r^{c, T}.$$
Indeed, a base of neighborhoods for $z_0\in K^G$ are given by the sets $W^{c,T}:=\{z\in K^G: \forall t\in T, \ \alpha_c(t \cdot z) = \alpha_c(t \cdot z_0)\}$.

Fix a finite $F \subseteq G$, $m \leq \ell \in \N$, and $\epsilon > 0$. Define
$$U = \{z \in K^G : \exists x \in X \ \forall f \in F \ \alpha_\ell(f \cdot z) = \alpha_\ell(f \cdot x)\}.$$
We have
$$\bigcap_{r \in \N} \bigcap_{g \in G} \bigcap_{c \in \N} \bigcap_{T \subseteq G} g \cdot V_r^{c, T} = \bigcap_{r \in \N} \bigcap_{g \in G} g \cdot Y_r = X \subseteq U.$$
By taking complements, we see that the complement of $U$, which is compact, is covered by the union of the complements of the sets $g \cdot V_r^{c, T}$, which are open. Since $V_{r'}^{c', T'} \subseteq V_r^{c, T}$ whenever $r' \geq r$, $c' \geq c$, and $T' \supseteq T$, by compactness we get $r, c \in \N$ and finite sets $T, S \subseteq G$ such that $c \geq \ell$ and
$$X \subseteq \bigcap_{g \in S} g \cdot V_r^{c, T} \subseteq U.$$

Since $\Sigma$ is a sofic approximation to $G$, if $n$ is sufficiently large then there are at most $(\epsilon / 2) \cdot \vert D_n \vert$ many $i \in D_n$ with $\sigma_n(t s^{-1})(i) \neq \sigma_n(t) \circ \sigma_n(s^{-1})(i)$ for some $s \in S$, $t \in T$. Fix such a value of $n$.  Also fix $0 < \delta < \epsilon / (2 |S|)$. We claim that
$$\beta_{\ell, c} \circ \AM_{Y_r}(\alpha_c, T, \delta, \sigma_n) \subseteq \AM_X(\alpha_\ell, F, \epsilon, \sigma_n).$$
Fix $a \in \AM_{Y_r}(\alpha_c, T, \delta, \sigma_n)$. For $i \in D_n$ define the pattern $p_i \in K_c^T$ by $p_i(t) = a(\sigma_n(t)(i))$. Define $B$ to be the set of $i \in D_n$ such that either: (1) there are $s \in S$, $t \in T$ with $\sigma_n(t s^{-1})(i) \neq \sigma_n(t) \circ \sigma_n(s^{-1})(i)$; or (2) there is $s \in S$ such that the pattern $t\mapsto p_{\sigma_n(s^{-1})(i)}(t)$ is not witnessed in $Y_r$. By our choice of $n$ and $\delta$, we have $|B| < \epsilon \cdot \vert D_n \vert$.

Consider $i \in D_n$ with $i \not\in B$. Since $|B| < \epsilon \cdot \vert D_n \vert$ and $i \not\in B$, it suffices to show that the pattern $q \in K_\ell^F$ defined by $q(f) = \beta_{\ell, c} \circ a(\sigma_n(f)(i))$ is witnessed in $X$. Fix $z \in K^G$ satisfying $\alpha_c(g \cdot z) = a(\sigma_n(g)(i))$ for all $g \in G$. By definition, $z$ lies in $U$ if and only if $q$ is witnessed in $X$. Towards a contradiction, suppose that $q$ is not witnessed. Then we have $z \not\in U$. In particular, there is $s \in S$ with $s^{-1} \cdot z \not\in V_r^{c, T}$. Set $j = \sigma_n(s^{-1})(i)$. Since $i \not\in B$, by condition (1) we have 
that for all $t \in T$
$$\alpha_c(t s^{-1} \cdot z) = a(\sigma_n(t s^{-1})(i)) = a(\sigma_n(t)(j)).$$
Since $s^{-1} \cdot z \not\in V_r^{c, T}$, the pattern $p_j$ is not witnessed in $Y_r$. This contradicts condition (2). This proves the claim.

Since $\beta_{m, \ell} \circ \beta_{\ell, c} = \beta_{m, c}$, for $\ell\geq m$, it follows from the above claim that
$$\beta_{m, c}^{D_n} \circ \AM_{Y_r}(\alpha_c, T, \delta, \sigma_n) \subseteq \beta_{m, \ell}^{D_n} \circ \AM_X(\alpha_\ell, F, \epsilon, \sigma_n).$$
Therefore
$$\inf_{c \geq \ell} \inf_{r \in \N} \inf_{T \subseteq G} \inf_{\delta > 0} \limsup_{n \rightarrow \infty} \frac{1}{\vert D_n \vert} \cdot \log \Big| \beta_{m, c}^{D_n} \circ \AM_{Y_r}(\alpha_c, T, \delta, \sigma_n) \Big|$$
$$\leq \limsup_{n \rightarrow \infty} \frac{1}{\vert D_n \vert} \cdot \log \Big| \beta_{m, \ell}^{D_n} \circ \AM_X(\alpha_\ell, F, \epsilon, \sigma_n) \Big|.$$
Now by taking the infimum over $F$, $\epsilon$, and $\ell$ and then the supremum over $m$ (and noting that $\sup_m \inf_{\ell \geq m} \inf_{c \geq \ell}$ is the same as $\sup_m \inf_{c \geq m}$), we see that $\toph{\Sigma}{G}{X}$ is bounded below by the right-hand side of (\ref{eqn:tapprox}).
\end{proof}

\begin{cor}[Decreasing continuity] \label{cor:intsct}
Suppose that $K$ is a finite set. If $(X_n)_{n  \in \N}$ is a decreasing sequence of compact $G$-invariant subsets of $K^G$, then the action of $G$ on $X = \bigcap_{n \in \N} X_n$ has entropy
$$\toph{\Sigma}{G}{X} = \inf_{n \in \N} \toph{\Sigma}{G}{X_n}.$$
\end{cor}

\begin{proof}
We will apply Lemma \ref{lem:tapprox} with $Y_n = X_n$. Since $K$ is finite, we can assume that every $K_m = K$. Then $\alpha_\ell = \alpha$ and $\beta_{m, \ell} = \id$ for all $m \leq \ell \in \N$. It follows that in (\ref{eqn:tapprox}) the $\sup_m$ and $\inf_\ell$ terms can be removed, and the right-hand side of (\ref{eqn:tapprox}) becomes $\inf_{r \in \N} \toph{\Sigma}{G}{X_r}$ by Proposition \ref{prop:topent2}.
\end{proof}

\newpage
\section{Topological entropy of algebraic subshifts} \label{sec:talg}

In this section we develop various formulas for the topological entropy of algebraic subshifts $X \subseteq K^G$, where $K$ is either a finite or a profinite group.

\begin{assume}
Throughout this section, $G$ will denote a sofic group with sofic approximation $\Sigma = (\sigma_n : G \rightarrow \Sym(D_n))$, and $K = \varprojlim K_m$ will be a profinite group with homomorphisms $\beta_m : K \rightarrow K_m$ and $\beta_{m, \ell} : K_\ell \rightarrow K_m$, $\ell \geq m$. Also, $G \acts K^G$ will be the Bernoulli shift action, $\alpha : K^G \rightarrow K$ will be the tautological generating function, and $\alpha_m = \beta_m \circ \alpha$.
\end{assume}

\begin{prop} \label{prop:tsubgroup}
If $Y \subseteq K^G$ is a subgroup (not necessarily compact or $G$-invariant), then for all $m \leq \ell \in \N$ and finite $F \subseteq G$, 
\begin{equation} \label{eqn:tsubgroup}
\inf_{\epsilon > 0} \limsup_{n \rightarrow \infty} \frac{1}{\vert D_n \vert} \cdot \log \Big| \beta_{m, \ell}^{D_n} \circ \AM_Y(\alpha_\ell, F, \epsilon, \sigma_n) \Big| = \limsup_{n \rightarrow \infty} \frac{1}{\vert D_n \vert} \cdot \log \Big| \beta_{m, \ell}^{D_n} \circ \AM_Y(\alpha_\ell, F, 0, \sigma_n) \Big|,
\end{equation}
\end{prop}

\begin{proof}
Clearly the left-hand side of (\ref{eqn:tsubgroup}) is greater than or equal to the right-hand side.

Fix a finite set $F \subseteq G$, $m, n \in \N$, $\ell \geq m$, and $\epsilon > 0$. Let $P$ be the set of patterns $p \in K_\ell^F$ that are witnessed in $Y$, i.e. there is some $y \in Y$ with $\alpha_\ell(f \cdot y) = p(f)$ for all $f \in F$. Since $Y$ is a subgroup of $K^G$ and $\alpha_\ell$ is a homomorphism, we see $P$ is a subgroup of $K_\ell^F$. Recall that $a \in \AM_Y(\alpha_\ell, F, 0, \sigma_n)$ if and only if for every $i\in D_n$ the pattern $p^a_i \in K_\ell^F$, defined by $p^a_i(f) = a(\sigma_n(f)(i))$, is witnessed in $Y$. So $a \in \AM_Y(\alpha_\ell, F, 0, \sigma_n)$ if and only if $p^a_i \in P$ for every $i\in D_n$. It follows that $\AM_Y(\alpha_\ell, F, 0, \sigma_n)$ is a subgroup of $K_\ell^{D_n}$, and hence $\beta_{m, \ell} \circ \AM_Y(\alpha_\ell, F, 0, \sigma_n)$ is a subgroup of $K_m^{D_n}$.

Consider the map $\phi : K_\ell^{D_n} \rightarrow (K_\ell^F / P)^{D_n}$ defined by
$$\phi(a)(i) = p^a_i \cdot P \in K_\ell^F / P.$$
For $a, a' \in K_\ell^{D_n}$ we have $\phi(a) = \phi(a')$ if and only if $a^{-1} a' \in \AM_Y(\alpha_\ell, F, 0, \sigma_n)$. If $a \in \AM_Y(\alpha_\ell, F, \epsilon, \sigma_n)$ then by definition $p^a_i \in P$ for at least $(1 - \epsilon) \cdot \vert D_n\vert$ many values of $i\in D_n$. Hence
\begin{equation} \label{eqn:binom}
\Big| \phi(\AM_Y(\alpha_\ell, F, \epsilon, \sigma_n)) \Big| \leq \binom{\vert D_n \vert}{\lfloor \epsilon \cdot \vert D_n \vert \rfloor} \cdot |K_\ell^F / P|^{\epsilon \cdot \vert D_n \vert}.
\end{equation}

For each $b \in \beta_{m, \ell} \circ \AM_Y(\alpha_\ell, F, \epsilon, \sigma_n)$, pick $a(b) \in \AM_Y(\alpha_\ell, F, \epsilon, \sigma_n)$ with $\beta_{m, \ell}(a(b)) = b$. If $\phi(a(b)) = \phi(a(b'))$ then $a(b)^{-1} a(b') \in \AM_Y(\alpha_\ell, F, 0, \sigma_n)$. Since $\beta_{m, \ell}$ is a group homomorphism, this implies that
$$b^{-1} b' \in \beta_{m, \ell} \circ \AM_Y(\alpha_\ell, F, 0, \sigma_n).$$
Therefore
$$\Big| \beta_{m, \ell} \circ \AM_Y(\alpha_\ell, F, \epsilon, \sigma_n) \Big| \leq \Big| \beta_{m, \ell} \circ \AM_Y(\alpha_\ell, F, 0, \sigma_n) \Big| \cdot \Big| \phi(\AM_Y(\alpha_\ell, F, \epsilon, \sigma_n)) \Big|.$$
So by (\ref{eqn:binom}) and Lemma \ref{lem:choose}
\begin{align*}
\limsup_{n \rightarrow \infty} \frac{1}{\vert D_n \vert} \cdot \log & \Big| \beta_{m, \ell} \circ \AM_Y(\alpha_\ell, F, \epsilon, \sigma_n) \Big|\\
\leq & \limsup_{n \rightarrow \infty} \frac{1}{\vert D_n \vert} \cdot \log \Big| \beta_{m, \ell} \circ \AM_Y(\alpha_\ell, F, 0, \sigma_n) \Big|
\\ 
& - \epsilon \cdot \log(\epsilon) - (1 - \epsilon) \cdot \log(1 - \epsilon) + \epsilon \cdot \log |K_\ell^F / P|.
\end{align*}
Now take the infimum over $\epsilon$.
\end{proof}

From the above proposition, it follows that for algebraic subshifts $X \subseteq K^G$ we may remove one quantifier from the definition of $\toph{\Sigma}{G}{X}$.

\begin{cor}[Algebraic subshift, profinite]
If $X \subseteq K^G$ is an algebraic subshift then
$$\toph{\Sigma}{G}{X} = \sup_{m \in \N} \inf_{\ell \geq m} \inf_{\substack{F \subseteq G\\F \text{finite}}} \limsup_{n \rightarrow \infty} \frac{1}{\vert D_n \vert} \cdot \log \Big| \beta_{m, \ell}^{D_n} \circ \AM_X(\alpha_\ell, F, 0, \sigma_n) \Big|.$$
\end{cor}

When $K$ is finite the formula for the topological entropy now becomes quite nice.

\begin{cor}[Algebraic subshift, finite]
If $K$ is a finite group and $X \subseteq K^G$ is an algebraic subshift, then
$$\toph{\Sigma}{G}{X} = \inf_{\substack{F \subseteq G \\ F \text{ finite}}} \limsup_{n \rightarrow \infty} \frac{1}{\vert D_n \vert} \cdot \log \Big| \AM_X(\alpha, F, 0, \sigma_n) \Big|.$$
\end{cor}

The entropy formula further simplifies for algebraic subshifts of finite type.

\begin{cor}[Algebraic subshift of finite type]
 \label{cor:ft}
If $K$ is a finite group and $X \subseteq K^G$ is an algebraic subshift of finite type with test window $W \subseteq G$, then
$$\toph{\Sigma}{G}{X} = \limsup_{n \rightarrow \infty} \frac{1}{\vert D_n \vert} \cdot \log \Big| \AM_X(\alpha, W, 0, \sigma_n) \Big|.$$
\end{cor}

\begin{proof}
Define
$$C = \{z \in K^G : \exists x \in X \ \forall w \in W \ \alpha(w \cdot z) = \alpha(w \cdot x)\}.$$
Then $C$ is compact. Since $X$ is a subshift of finite type with test window $W$, we have $X = \bigcap_{g \in G} g \cdot C$. By Lemma \ref{lem:tapprox} and Proposition \ref{prop:tsubgroup} we have
$$\toph{\Sigma}{G}{X} = \inf_{\substack{F \subseteq G\\F \text{ finite}}} \limsup_{n \rightarrow \infty} \frac{1}{\vert D_n \vert} \cdot \log \Big| \AM_C(\alpha, F, 0, \sigma_n) \Big|.$$
Note that by monotonicity properties we can require that $F \supseteq W$ above. Finally, from the definition of $C$ we have that 
$$\AM_C(\alpha, F, 0, \sigma_n) = \AM_C(\alpha, W, 0, \sigma_n) = \AM_X(\alpha, W, 0, \sigma_n)$$
whenever $F$ contains $W$.
\end{proof}

The relation between topological entropy and periodic points now follows quickly.

\begin{thm}[Topological entropy and fixed points]\label{thm:period}
Suppose that $G$ is residually finite, $(G_n)$ is a sofic chain of finite-index subgroups, and $\Sigma = (\sigma_n : G \rightarrow \Sym(G_n \backslash G))$ is the associated sofic approximation. If $K$ is a finite group and $X \subseteq K^G$ is an algebraic subshift of finite type, then
$$\toph{\Sigma}{G}{X} = \limsup_{n \rightarrow \infty} \frac{1}{|G : G_n|} \cdot \log \Big| \Fix{G_n}{X} \Big|,$$
where $\Fix{G_n}{X}$ is the set of $G_n$-periodic elements of $X$.
\end{thm}

\begin{proof}
Let $W$ be a test window for $X$ and let $P$ be the $W$-patterns defining $X$:
$$P = \{p \in K^W : \exists x \in X \ \forall w \in W \ p(w) = \alpha(w \cdot x)\}$$
The natural map $a \mapsto \bar{a}$ from $K^{G_n \backslash G}$ to $K^G$ obtained by setting $\bar{a}(g) = a(G_n g)$ (induced by $G\to {G_n \backslash G}$) is a bijection
 between $K^{G_n \backslash G}$ and the $G_n$-periodic elements of $K^G$. 
 
 Moreover, by definition, for $g \in G$ and $w \in W$ we have the equality
$$\alpha(w g \cdot \bar{a}) = \bar{a}(g^{-1} w^{-1}) = a(G_n g^{-1} w^{-1}) = a \big( \sigma_n(w)(G_n g^{-1}) \big).$$
Now (1) $\bar{a}$ lies in $X$ if and only if $(w \mapsto \alpha(w g \cdot \bar{a})) \in P$ for all $g \in G$, and

{}\hspace{8pt} (2) $a \in \AM_X(\alpha, W, 0, \sigma_n)$ if and only if $(w \mapsto a(\sigma_n(w)(G_n g))) \in P$ for all $G_n g \in G_n \backslash G$.
\\ 
 It follows that $\bar{a} \in X$ if and only if $a \in \AM_X(\alpha, W, 0, \sigma_n)$. Thus $a \mapsto \bar{a}$ is a bijection between $\AM_X(\alpha, W, 0, \sigma_n)$ and $\Fix{G_n}{X}$. Now the claim follows from Corollary \ref{cor:ft}.
\end{proof}

We mention for completeness the following fact.

\begin{lem} \label{lem:ker}
If $K$ is a compact group, $L$ is a finite group, and $\phi : K^G \rightarrow L^G$ is a continuous $G$-equivariant group homomorphism, then $\ker(\phi) \subseteq K^G$ is an algebraic subshift of finite type.
\end{lem}

\begin{proof}
Clearly $\ker(\phi)$ is a closed $G$-invariant subgroup of $K^G$. Let $\gamma$ be the tautological generating function for $L^G$. Since $\gamma \circ \phi$ is continuous and finite valued, the set $(\gamma \circ \phi)^{-1}(1_L) \subseteq K^G$ is compact and open and hence the union of a finite number of cylinder sets. So there is a finite $W \subseteq G$ and $P \subseteq K^W$ such that for all $z \in K^G$, $\gamma \circ \phi(z) = 1_L$ if and only if the pattern $w \mapsto \alpha(w \cdot z)$ lies in $P$. Therefore
$$z \in \ker(\phi) \Leftrightarrow \forall g \in G \ \ \  \gamma \circ \phi(g \cdot z) = 1_L \Leftrightarrow \forall g \in G \ (w \mapsto \alpha(w g \cdot z)) \in P.$$
Thus $\ker(\phi)$ is of finite type (with test window $W$) as claimed.
\end{proof}

\bigskip

Recall that for a group $G$, the \emph{group ring} $\Z[G]$ is the set of all formal sums $\sum_{g \in G} f^g \cdot g$ where $f^g \in \Z$ and $f^g = 0$ for all but finitely many $g \in G$. If $\KK$ is a finite field and $f \in \Z[G]$, then the \emph{right-convolution} of $f$ on $\KK^G$ is defined as  
\begin{equation*}
f^{\KK}: \left(\begin{array}{ccc}\KK^G & \to & \KK^G \\x & \mapsto  & x*f\end{array}\right) \ \ \ \text{ where }(x * f)(g) = \sum_{h \in G} x(g h^{-1}) \cdot f^h, \text{ for all } g \in G.
\end{equation*}

More generally, if $M = \bigl( m_{i,j}\bigr)_{i,j}=\bigl( \sum m_{i,j}^g \cdot g\bigr)_{i,j} \in \Mat{r}{s}{\Z[G]}$, then \emph{right-convolution} by $M$ produces a map 
\begin{equation*}
M^{\KK}: \left(\begin{array}{ccc}(\KK^r)^G & \to & (\KK^s)^G \\x & \mapsto  & M^{\KK}(x)\end{array}\right),
\end{equation*}
where for $g\in G$ and  $1\leq j \leq s$ we have
\begin{equation} \label{eqn:conv}
M^\KK(x)(g)(j) = (x * M)(g)(j) = \sum_{i = 1}^r \Big( [x(\cdot)(i)] * m_{i,j} \Big)(g) = \sum_{i = 1}^r \sum_{h \in G} x(g h^{-1})(i) \cdot m_{i,j}^h.
\end{equation}

The map $M^\KK$ is easily seen to be a continuous $\KK$-linear homomorphism. It is furthermore $G$-equivariant since
\begin{align*}
M^\KK(u \cdot x)(g)(j) & = \sum_{i = 1}^r \sum_{h \in G} (u \cdot x)(g h^{-1})(i) \cdot m_{i,j}^h\\
 & = \sum_{i = 1}^r \sum_{h \in G} x(u^{-1} g h^{-1})(i) \cdot m_{i,j}^h = M^\KK(x)(u^{-1} g)(j) = [u \cdot M^\KK(x)](g)(j).
\end{align*}
The convolution operation of $M$ also induces maps 
\begin{equation*}
\left.\begin{array}{cc}
M^\Z : \begin{array}{ccc} (\Z^r)^G &\rightarrow &(\Z^s)^G\end{array}
\\
M^{(2)}: \begin{array}{ccc}\ell^2(G)^r &\rightarrow &\ell^2(G)^s\end{array}\end{array}\right.
\end{equation*}
by the same formula. Note that $M^{(2)}$ is a continuous $G$-equivariant operator.

Given a sofic approximation $\Sigma = (\sigma_n : G \rightarrow \Sym(D_n))$ of $G$, we extend each map $\sigma_n$ linearly to a map $\sigma_n : \Mat{r}{s}{\Z[G]} \rightarrow \Mat{r}{s}{\Z[\Sym(D_n)]}$. Just as in the previous paragraph, if $\tilde{M} \in \Mat{r}{s}{\Z[\Sym(D_n)]}$ then \emph{right-convolution} by $\tilde{M}$ produces a map 
\begin{equation*}
\tilde{M}^\KK: \left(\begin{array}{ccc}(\KK^r)^{D_n} &\rightarrow &(\KK^s)^{D_n} \\a & \mapsto  & \tilde{M}^\KK(a)\end{array}\right)
\end{equation*}
where for all $\delta \in D_n$ and $1 \leq j \leq s$ we have
\begin{equation*}
\tilde{M}^\KK(a)(\delta)(j) = \sum_{i = 1}^r \sum_{\sigma \in \Sym(D_n)} a(\sigma(\delta))(i) \cdot \tilde{m}_{i,j}^\sigma.
\end{equation*}
In particular, from these definitions we have that for $M \in \Mat{r}{s}{\Z[G]}$
$$\sigma_n(M)^\KK(a)(\delta)(j) = \sum_{i = 1}^r \sum_{h \in G} a \Big( \sigma_n(h)(\delta) \Big)(i) \cdot m_{i,j}^h.$$
Since labellings $a:D_n\to \KK^r$ are intended to be approximations of $\alpha : (\KK^r)^G \rightarrow \KK^r$, this formula should feel more natural if one rewrites (\ref{eqn:conv}) by using $x(g h^{-1})(i) \cdot m_{i,j}^h = \alpha(h g^{-1} \cdot x)(i) \cdot m_{i,j}^h$.

\begin{lem} \label{lem:kerent}
Let $\KK$ be a finite field, let $r, s \in \N$, and let $M \in \Mat{r}{s}{\Z[G]}$. Then
$$\toph{\Sigma}{G}{\ker(M^\KK)} = \limsup_{n \rightarrow \infty} \frac{1}{\vert D_n \vert} \cdot \log \Big| \ker(\sigma_n(M)^\KK) \Big|.$$
\end{lem}

\begin{proof}
Set $Y = \{x \in (\KK^r)^G : M^\KK(x)(1_G) = 0_{\KK^s}\}$. It is immediate that
$$\ker(M^\KK) = \bigcap_{g \in G} g \cdot Y.$$
Let $W \subseteq G$ be finite but sufficiently large so that $W = W^{-1}$ contains the support of every component of $M$. By Lemma \ref{lem:tapprox} and Proposition \ref{prop:tsubgroup} we have
$$\toph{\Sigma}{G}{\ker(M^\KK)} = \inf_{\substack{F \subseteq G\\F \text{ finite}}} \limsup_{n \rightarrow \infty} \frac{1}{\vert D_n \vert} \cdot \log \Big| \AM_Y(\alpha, F, 0, \sigma_n) \Big|.$$
Note that by monotonicity properties we can require that $F \supseteq W$ above. From the definition of $Y$ we have that 
$$\AM_Y(\alpha, F, 0, \sigma_n) = \AM_Y(\alpha, W, 0, \sigma_n)$$
whenever $F$ contains $W$. Thus
$$\toph{\Sigma}{G}{\ker(M^\KK)} = \limsup_{n \rightarrow \infty} \frac{1}{\vert D_n \vert} \cdot \log \Big| \AM_Y(\alpha, W, 0, \sigma_n) \Big|.$$
This equation is quite similar to the conclusion of Corollary \ref{cor:ft}, the only difference being the use of $\AM_Y$ in place of $\AM_{\ker(M^\KK)}$.

To complete the proof, we claim that
$$\AM_Y(\alpha, W, 0, \sigma_n) = \ker(\sigma_n(M)^\KK).$$
It suffices to fix $\delta \in D_n$ and argue that $\sigma_n(M)^\KK(a)(\delta) = 0_{\KK^s}$ if and only if the pattern $w \in W \mapsto a(\sigma_n(w)(\delta))$ is witnessed in $Y$. Fix $a \in (\KK^r)^{D_n}$ and $\delta \in D_n$. Fix any $z \in (\KK^r)^G$ satisfying $\alpha(w \cdot z) = a(\sigma_n(w)(\delta))$ for all $w \in W$. Then for every $1 \leq j \leq s$ we have
\begin{align*}
\sigma_n(M)^\KK(a)(\delta)(j) & = \sum_{i = 1}^r \sum_{h \in W} a \Big( \sigma_n(h)(\delta) \Big)(i) \cdot m_{i,j}^h\\
 & = \sum_{i = 1}^r \sum_{h \in W} \alpha(h \cdot z)(i) \cdot m_{i,j}^h\\
 & = \sum_{i = 1}^r \sum_{h \in W} z(h^{-1})(i) \cdot m_{i,j}^h\\
 & = M^\KK(z)(1_G)(j).
\end{align*}
Thus $\sigma_n(M)^\KK(a) = 0_{\KK^s}$ if and only if $z \in Y$. This completes the proof.
\end{proof}

\begin{lem} \label{lem:dim}
Let $G$ be a sofic group with sofic approximation $\Sigma$, let $\KK$ be a finite field, and let $M \in \Mat{r}{s}{\Z[G]}$. Then
$$\toph{\Sigma}{G}{\ker(M^\KK)} \geq (\dim_G \ker(M^{(2)})) \cdot \log |\KK|.$$
\end{lem}

We remark that we do not know of any example where the inequality is strict.

\begin{proof}
Say $\Sigma = (\sigma_n : G \rightarrow \Sym(D_n))$. Let $\sigma_n(M)^\C : (\C^r)^{D_n} \rightarrow (\C^s)^{D_n}$ be the map given by the convolution operation of $\sigma_n(M)$. By Lemma \ref{lem:kerent} we have
\begin{equation} \label{eqn:dim1}
\toph{\Sigma}{G}{\ker(M^\KK)} = \limsup_{n \rightarrow \infty} \frac{1}{\vert D_n \vert} \cdot \log | \ker(\sigma_n(M)^\KK) |.
\end{equation}
Also, observe that
\begin{equation} \label{eqn:dim2}
\log \Big| \ker(\sigma_n(M)^\KK) \Big| = \Big( \dim_\KK \ker(\sigma_n(M)^\KK) \Big) \cdot \log |\KK| \geq \Big( \dim_\C \ker(\sigma_n(M)^\C) \Big) \cdot \log |\KK|.
\end{equation}
Finally, by a theorem of Thom \cite[Th. 4.2]{Tho08} that extends the L\"{u}ck approximation theorem to the sofic framework (see also the work of Elek and Szab\'{o} \cite{ES04}), we have
\begin{equation} \label{eqn:dim3}
\lim_{n \rightarrow \infty} \frac{\dim_\C \ker(\sigma_n(M)^\C)}{\vert D_n \vert} = \dim_G \ker(M^{(2)}).
\end{equation}
Thus, by (\ref{eqn:dim1}), (\ref{eqn:dim2}), and (\ref{eqn:dim3}) we have $\toph{\Sigma}{G}{\ker(M^\KK)} \geq \dim_G(\ker(M^{(2)})) \cdot \log |\KK|$.
\end{proof}

\newpage
\section{Yuzvinsky's addition formula and \texorpdfstring{$\ell^2$-Betti}{L2-Betti} numbers} \label{sec:Juzv}

In this section we study the Yuzvinsky addition formula and in particular its connection to $\ell^2$-Betti numbers. We first consider the case of residually finite groups where the relationships and results are particularly natural. From Theorem \ref{thm:period} we obtain the following inequality.

\begin{cor} \label{cor:rfj}
Let $G$ be a residually finite group, let $(G_n)$ be a sofic chain of finite-index subgroups, and let $\Sigma = (\sigma_n : G \rightarrow \Sym(G_n \backslash G))$ be the associated sofic approximation. If $K$ and $L$ are finite groups, $f : K^G \rightarrow L^G$ is a continuous $G$-equivariant group homomorphism, and $X \subseteq L^G$ is an algebraic subshift of finite type with $\Img(f) \subseteq X$, then
$$\toph{\Sigma}{G}{K^G} + \limsup_{n \rightarrow \infty} \frac{1}{|G : G_n|} \cdot \log \frac{|\Fix{G_n}{X}|}{|f(\Fix{G_n}{K^G})|} \leq \toph{\Sigma}{G}{\ker(f)} + \toph{\Sigma}{G}{X}.$$
\end{cor}

\begin{proof}
Since $\Fix{G_n}{K^G}$ is a finite subgroup of $K^G$, $f$ is a group homomorphism, and $\ker(f \res \Fix{G_n}{K^G}) = \Fix{G_n}{\ker(f)}$, for every $n$ we have
\begin{align*}
\log|K| & = \frac{1}{|G : G_n|} \cdot \log \Big| \Fix{G_n}{K^G} \Big|\\
 & = \frac{1}{|G : G_n|} \cdot \log \Big| \Fix{G_n}{\ker(f)} \Big| + \frac{1}{|G : G_n|} \cdot \log \Big| f(\Fix{G_n}{K^G}) \Big|.
\end{align*}
As $\toph{\Sigma}{G}{K^G} = \log |K|$ is equal to the above expression for every $n$, it follows that $\toph{\Sigma}{G}{K^G}$ is equal to
$$\limsup_{n \rightarrow \infty} \frac{1}{|G : G_n|} \cdot \log \Big| \Fix{G_n}{\ker(f)} \Big| + \liminf_{n \rightarrow \infty} \frac{1}{|G : G_n|} \cdot \log \Big| f(\Fix{G_n}{K^G}) \Big|.$$
By our assumption on $X$ and Lemma \ref{lem:ker}, both $X$ and $\ker(f)$ are subshifts of finite type. So by Theorem \ref{thm:period} we have
\begin{align}
& \toph{\Sigma}{G}{\ker(f)} + \toph{\Sigma}{G}{X} - \toph{\Sigma}{G}{K^G}\nonumber\\
& \qquad = \limsup_{n \rightarrow \infty} \frac{1}{|G : G_n|} \cdot \log \Big| \Fix{G_n}{X} \Big| - \liminf_{n \rightarrow \infty} \frac{1}{|G : G_n|} \cdot \log \Big| f(\Fix{G_n}{K^G}) \Big|\label{eqn:supinf}\\
& \qquad \geq \limsup_{n \rightarrow \infty} \frac{1}{|G : G_n|} \cdot \log \frac{|\Fix{G_n}{X}|}{|f(\Fix{G_n}{K^G})|}.\nonumber\qedhere
\end{align}
\end{proof}

The above corollary is most interesting when $X = \Img(f)$ as then it relates to the Yuzvinsky addition formula. In fact, by inspecting the above proof we arrive at the following.

\begin{cor} \label{cor:rfj2}
Let $G$, $\Sigma$, and $f$ be as in Corollary \ref{cor:rfj}, and suppose that $\Img(f)$ is of finite type. Then $f$ satisfies the Yuzvinsky addition formula for $\Sigma$ if and only if
$$\limsup_{n \rightarrow \infty} \frac{1}{|G : G_n|} \cdot \log \Big| \Fix{G_n}{\Img(f)} \Big| = \liminf_{n \rightarrow \infty} \frac{1}{|G : G_n|} \cdot \log \Big| f(\Fix{G_n}{K^G}) \Big|.$$
\end{cor}

\begin{proof}
This is immediate from equation (\ref{eqn:supinf}) with $X = \Img(f)$.
\end{proof}

Currently it is not known if the limsup's that appear in the definition of sofic topological entropy may be replaced with limits (i.e. its not known if the sequences converge). What is peculiar is that this seemingly mundane detail manifests itself in the study of the Yuzvinsky addition formula.

\begin{cor}
Let $G$, $\Sigma$, and $f$ be as in Corollary \ref{cor:rfj}, and suppose that $\Img(f)$ is of finite type. If
$$\limsup_{n \rightarrow \infty} \frac{1}{|G : G_n|} \cdot \log \Big| \Fix{G_n}{\Img(f)} \Big| \neq \liminf_{n \rightarrow \infty} \frac{1}{|G : G_n|} \cdot \log \Big| \Fix{G_n}{\Img(f)} \Big|$$
then $f$ fails to satisfy the Yuzvinsky addition formula for $\Sigma$.
\end{cor}

\begin{proof}
The quantity
$$\liminf_{n \rightarrow \infty} \frac{1}{|G : G_n|} \cdot \log \Big| \Fix{G_n}{\Img(f)} \Big|$$
lies between the quantities being compared in Corollary \ref{cor:rfj2}.
\end{proof}

In the case where $G$ is residually finite, the connection between $\ell^2$-Betti numbers and the Yuzvinsky addition formula now follows quickly. Specifically, let $G$ be a residually finite group, let $(G_n)$ be a sofic chain of finite index subgroups, and let $G$ act freely and cocompactly on a simplicial complex $L$. Fix a finite field $\KK$ and consider the coboundary maps
$$C^{p-1}(L, \KK) \overset{\delta^p}{\longrightarrow} C^p(L, \KK) \overset{\delta^{p+1}}{\longrightarrow} C^{p+1}(L, \KK).$$
Note that $C^{p-1}(L, \KK) \cong (\KK^r)^G$ and $C^p(L, \KK) \cong (\KK^s)^G$, where $r$ and $s$ are the number of $G$-orbits of $(p-1)$-simplices and $p$-simplices, respectively. By Lemma \ref{lem:ker}, $\ker(\delta^p)$ and $\ker(\delta^{p+1})$ are of finite type so we may apply Corollary \ref{cor:rfj}. Since the restriction of $\delta^i$ to $\Fix{G_n}{C^{i-1}(L, \KK)}$ coincides with the quotient coboundary map $\delta^i_n : C^{i-1}(G_n \backslash L, \KK) \rightarrow C^i(G_n \backslash L, \KK)$, we have
$$\log \frac{|\Fix{G_n}{\ker(\delta^{p+1})}|}{|\delta^p(\Fix{G_n}{C^{p-1}(L, \KK)})|} = \log \frac{|\ker(\delta^{p+1}_n)|}{|\Img(\delta^p_n)|} = \dim_\KK H^p(G_n \backslash L, \KK) \cdot \log |\KK|.$$
We automatically have $\dim_\KK H^p(G_n \backslash L, \KK) \geq \dim_\C H^p(G_n \backslash L, \C)$, and by the Farber--L\"{u}ck Approximation Theorem \ref{thm:luck} we obtain
$$\lim_{n \rightarrow \infty} \frac{\dim_\C H^p(G_n \backslash L, \C)}{|G : G_n|} = \beta^p_{(2)}(L : G).$$
Therefore by Corollary \ref{cor:rfj}
$$\toph{\Sigma}{G}{C^{p-1}(L, \KK)} + \beta^p_{(2)}(L : G) \cdot \log |\KK| \leq \toph{\Sigma}{G}{\ker(\delta^p)} + \toph{\Sigma}{G}{\ker(\delta^{p+1})}.$$
In particular, when $\Img(\delta^p) = \ker(\delta^{p+1})$, or equivalently $H^p(L, \KK) = 0$, the $\ell^2$-Betti number $\beta^p_{(2)}(L : G)$ provides a lower bound to the failure of the Yuzvinsky addition formula.

In fact, we obtain the following connection with normalized Betti numbers computed over finite fields. This equation was previously discovered for amenable groups $G$ by Elek \cite{E02}.

\begin{cor} \label{cor:finitebetti}
let $G$ be a residually finite group, let $(G_n)$ be a sofic chain of finite index subgroups, and let $\KK$ be a finite field. Let $G$ act freely and cocompactly on a simplicial complex $L$, let $\delta^i : C^{i-1}(L, \KK) \rightarrow C^i(L, \KK)$ be the coboundary map, and let $\delta^i_n : C^{i-1}(G_n \backslash L, \KK) \rightarrow C^i(G_n \backslash L, \KK)$ denote the quotient coboundary map. Let $p \geq 1$ and assume that $\frac{1}{|G : G_n|} \cdot \log(|\ker(\delta^{p}_n)|)$ converges as $n \rightarrow \infty$. Then
\begin{align*}
\toph{\Sigma}{G}{C^{p-1}(L, \KK)} + \limsup_{n \rightarrow \infty} & \frac{\dim_\KK H^p(G_n \backslash L, \KK)}{|G : G_n|} \cdot \log|\KK|\\
& \qquad \qquad \qquad = \toph{\Sigma}{G}{\ker(\delta^p)} + \toph{\Sigma}{G}{\ker(\delta^{p+1})}.
\end{align*}
\end{cor}

\begin{proof}
We will rely on Corollary \ref{cor:rfj} with $K^G = (\KK^r)^G = C^{p-1}(L, \KK)$ (here $r$ is the number of orbits of $(p-1)$-simplices), $f = \delta^p$, and $X = \ker(\delta^{p+1})$. By inspecting the proof, specifically expression \ref{eqn:supinf} and the line following it, we see that the inequality in the statement of Corollary \ref{cor:rfj} will be an equality provided that $\frac{1}{|G : G_n|} \log |f(\Fix{G_n}{K^G})|$ converges. In our case,
$$\frac{1}{|G : G_n|} \log \Big| f(\Fix{G_n}{K^G}) \Big| = \frac{1}{|G : G_n|} \log \Big| \Img(\delta^p_n) \Big| = \log|\KK| - \frac{1}{|G : G_n|} \log \Big| \ker(\delta^p_n) \Big|$$
converges by assumption. Thus, for our choice of $K^G$, $f$, and $X$, Corollary \ref{cor:rfj} provides an equality rather than an inequality. Now substituting $K^G = C^{p-1}(L, \KK)$, $f = \delta^p$, $X = \ker(\delta^{p+1})$, and
$$\frac{1}{|G : G_n|} \cdot \log \frac{|\Fix{G_n}{X}|}{|f(\Fix{G_n}{K^G})|} = \frac{1}{|G : G_n|} \cdot \log \frac{|\ker(\delta^{p+1}_n)|}{|\Img(\delta^p_n)|} = \frac{\dim_\KK H^p(G_n \backslash L, \KK)}{|G : G_n|} \cdot \log|\KK|,$$
completes the proof.
\end{proof}

\vspace{3mm}
We now turn our attention to sofic groups $G$ that are not necessarily residually finite. We draw an important corollary from our work in the previous section.

\begin{cor} \label{cor:ineq}
Let $G$ be a sofic group with sofic approximation $\Sigma$. Let $\KK$ be a finite field, let $f \in \Mat{r}{s}{\Z[G]}$, let $g \in \Mat{s}{t}{\Z[G]}$. Consider the convolution maps given by $f$ and $g$,
$(\KK^r)^G\overset{f^{\KK}}{\to}(\KK^s)^G\overset{g^{\KK}}{\to}(\KK^t)^G$,
 and suppose that $f g = 0$ (i.e. $g^{\KK}\circ f^{\KK}=0$). Then
\begin{align*}
\toph{\Sigma}{G}{(\KK^r)^G} + \dim_G \Big( \ker(g^{(2)}) / \overline{\Img(f^{(2)})} \Big) & \cdot \log |\KK|\\
 & \leq \toph{\Sigma}{G}{\ker(f^\KK)} + \toph{\Sigma}{G}{\ker(g^\KK)}.
\end{align*}
\end{cor}

\begin{proof}
Using Lemma \ref{lem:dim}, we have
\begin{align*}
\toph{\Sigma}{G}{\ker(f^\KK)} & + \toph{\Sigma}{G}{\ker(g^\KK)} - \toph{\Sigma}{G}{(\KK^r)^G}\\
& = \toph{\Sigma}{G}{\ker(f^\KK)} + \toph{\Sigma}{G}{\ker(g^\KK)} - \log |\KK| \cdot \dim_G(\ell^2(G)^r)\\
& \geq \Big( \dim_G(\ker(f^{(2)})) + \dim_G(\ker(g^{(2)})) - \dim_G(\ell^2(G)^r) \Big) \cdot \log |\KK|
\\
&=\Big( \dim_G(\ker(g^{(2)})) -  \dim_G(\overline{\Img(f^{(2)})}) \Big) \cdot \log |\KK|
\\
& = \dim_G \Big( \ker(g^{(2)}) / \overline{\Img(f^{(2)})} \Big) \cdot \log |\KK|. 
\end{align*}
The last two equalities hold since the von Neumann dimension satisfies the rank-nullity theorem.
\end{proof}

Before we can apply this corollary we need two simple lemmas.

\begin{lem} \label{lem:equiv}
If $\psi : (\Z^r)^G \rightarrow (\Z^s)^G$ is a continuous $G$-equivariant group homomorphism, then there is a matrix $M \in \Mat{r}{s}{\Z[G]}$ such that $\psi = M^\Z$.
\end{lem}

\begin{proof}
Define $\phi : (\Z^r)^G \rightarrow \Z^s$ by $\phi(x) = \psi(x)(1_G)$. As $\phi$ is a continuous group homomorphism, there is a neighborhood around $0_{(\Z^r)^G}$ that $\phi$ maps to $0_{\Z^s}$. So there is a finite set $T \subseteq G$ such that $\phi(x) = 0$ whenever $\forall t \in T \ x(t) = 0_{\Z^r}$. As $\phi$ is a homomorphism, it follows that there is a map $\tilde{\phi} : (\Z^r)^T \rightarrow \Z^s$ such that $\phi(x) = \tilde{\phi}(x \res T)$ for all $x \in (\Z^r)^G$. Now $\tilde{\phi}$ is a homomorphism from the finite rank free abelian group $(\Z^r)^T$ to $\Z^s$. Thus there are $\{m_{i,j}^{t^{-1}} \in \Z : 1 \leq i \leq r, \ 1 \leq j \leq s, \ t \in T\}$ such that for all $p \in (\Z^r)^T$
$$\tilde{\phi}(p)(j) = \sum_{i = 1}^r \sum_{t \in T} p(t)(i) \cdot m_{i,j}^{t^{-1}}.$$
Therefore $\phi(x)(j) = \sum_{i = 1}^r \sum_{t \in T}  x(t)(i) \cdot m_{i,j}^{t^{-1}}$ for all $x \in (\Z^r)^G$. Letting $M \in \Mat{r}{s}{\Z[G]}$ be the matrix with entries $m_{i,j} = \sum_{t \in T} m_{i,j}^{t^{-1}} \cdot t^{-1}$, we have
\begin{align*}
\psi(x)(g)(j) & = \phi(g^{-1} \cdot x)(j)\\
 & = \sum_{i = 1}^r \sum_{t \in T}  (g^{-1} \cdot x)(t)(i) \cdot m_{i,j}^{t^{-1}}\\
 & = \sum_{i = 1}^r \sum_{t \in T}  x(g t)(i) \cdot m_{i,j}^{t^{-1}}\\
 & = M^\Z(x)(g)(j).\qedhere
\end{align*}
\end{proof}

\begin{lem} \label{lem:infbetti}
Let $G$ act on a simplicial complex $L$ and let $p \geq 1$. Suppose that the action of $G$ on the $p$-skeleton of $L$ is cocompact. Let $L_i$ be an increasing sequence of $G$-invariant cocompact subcomplexes each containing the $p$-skeleton of $L$ and satisfying $L = \cup_i L_i$. Then
$$\beta^p_{(2)}(L : G) = \inf_i \beta^p_{(2)}(L_i : G).$$
\end{lem}

\begin{proof}
For $j \geq i$ let $\pi^j_i : C^p_{(2)}(L_j) \rightarrow C^p_{(2)}(L_i)$ be the map (restriction) induced by the inclusion $L_i \subseteq L_j$. Then $\pi^j_i$ descends to a map $\bar{H}^p_{(2)}(L_j) \rightarrow \bar{H}^p_{(2)}(L_i)$. By definition we have
$$\beta^p_{(2)}(L : G) = \sup_{i \in \N} \inf_{j \geq i} \dim_G \overline{\Img}\Big( \bar{H}^p_{(2)}(L_j) \rightarrow \bar{H}^p_{(2)}(L_i) \Big).$$
Since each $L_i$ contains the $p$-skeleton of $L$, all of the $L_i$'s share the same $p$-skeleton. Thus every $\pi^j_i : C^p_{(2)}(L_j) \rightarrow C^p_{(2)}(L_i)$ is the identity map and every induced map $\bar{H}^p_{(2)}(L_j) \rightarrow \bar{H}^p_{(2)}(L_i)$ is an inclusion. Therefore
\begin{equation*}
\beta^p_{(2)}(L : G) = \inf_j \dim_G \bar{H}^p_{(2)}(L_j) = \inf_j \beta^p_{(2)}(L_j : G).\qedhere
\end{equation*}
\end{proof}

Now we can relate $\ell^2$-Betti numbers to entropy in the general case of sofic groups.

\begin{thm} \label{thm:betti}
Let $G$ be a sofic group with sofic approximation $\Sigma$, and let $G$ act freely on a simplicial complex $L$. Let $\KK$ be a finite field and let $\delta^i : C^{i-1}(L, \KK) \rightarrow C^i(L, \KK)$ be the coboundary map. If $p \geq 1$ and the action of $G$ on the $p$-skeleton of $L$ is cocompact then
$$\toph{\Sigma}{G}{C^{p-1}(L, \KK)} + \beta^p_{(2)}(L : G) \cdot \log |\KK| \leq \toph{\Sigma}{G}{\ker(\delta^p)} + \toph{\Sigma}{G}{\ker(\delta^{p+1})}.$$
\end{thm}

\begin{proof}
First assume that $L$ itself is cocompact. Let $r$, $s$, and $t$ be the number of $G$-orbits of $(p-1)$, $p$, and $(p+1)$ simplices in $L$, respectively. By fixing representatives from every $G$-orbit, we obtain $G$-equivariant isomorphisms between $C^{p-1}(L, \Z)$, $C^p(L, \Z)$, $C^{p+1}(L, \Z)$ and $(\Z^r)^G$, $(\Z^s)^G$, and $(\Z^t)^G$, respectively. By Lemma \ref{lem:equiv} there are matrices
$$M_p \in \Mat{r}{s}{\Z[G]}, \quad M_{p+1} \in \Mat{s}{t}{\Z[G]}$$
whose convolution maps coincide with $\delta^p$ and $\delta^{p+1}$. It is clear that in the $\ell^2$ setting we have $\delta^p = M_p^{(2)}$ and $\delta^{p+1} = M_{p+1}^{(2)}$. Also $M_p M_{p+1} = 0$ (recall that we defined the convolution operation on the right-hand side).

Now consider coefficients in $\KK$. Up to isomorphism we have $C^i(L, \KK) = C^i(L, \Z) \otimes \KK$ and $\delta^i_\KK = \delta^i_\Z \otimes \id_\KK$, where we included subscripts to clarify the coefficients being used. Using the same representatives from the $G$-orbits of $(p-1)$, $p$, and $(p+1)$ simplices as before, we have $G$-equivariant isomorphisms between $C^{p-1}(L, \KK)$, $C^p(L, \KK)$, $C^{p+1}(L, \KK)$ and $(\KK^r)^G$, $(\KK^s)^G$, and $(\KK^t)^G$. With this identification we have $\delta^p_\KK = M_p^\KK$ and $\delta^{p+1}_\KK = M_{p+1}^\KK$. By applying Corollary \ref{cor:ineq} we obtain
\begin{align*}
& \ \toph{\Sigma}{G}{C^{p-1}(L, \KK)} + \dim_G \Big( \ker(M_{p+1}^{(2)}) / \overline{\Img}(M_p^{(2)}) \Big) \cdot \log |\KK|\\
\leq & \ \toph{\Sigma}{G}{\ker \delta^p} + \toph{\Sigma}{G}{\ker \delta^{p+1}}.
\end{align*}
So we are done, since by definition
\begin{equation*}
\beta^p_{(2)}(L : G) = \dim_G( \ker \delta^{p+1}_{(2)} / \overline{\Img}(\delta^p_{(2)})) = \dim_G (\ker(M_{p+1}^{(2)}) / \overline{\Img}(M_p^{(2)}) ).
\end{equation*}

Now in the general case write $L$ as an increasing union $L = \cup_i L_i$ of cocompact subcomplexes $L_i$ each of which contains the $p$-skeleton of $L$. By the above argument we have
\begin{equation} \label{eqn:noncmpct}
\toph{\Sigma}{G}{C^{p-1}(L_i, \KK)} + \beta^p_{(2)}(L_i : G) \cdot \log |\KK| \leq \toph{\Sigma}{G}{\ker(\delta^p_{L_i})} + \toph{\Sigma}{G}{\ker(\delta^{p+1}_{L_i})}.
\end{equation}
Of course, $C^{p-1}(L_i, \KK) = C^{p-1}(L, \KK)$ and $\delta^p_{L_i} = \delta^p$. Each $\ker(\delta^{p+1}_{L_i})$ is a subset of $C^p(L_i, \KK) = C^p(L, \KK)$ and since the $p+1$-cells of $L$ are just the union of those of the $L_i$,
it is easily seen that
$$\ker(\delta^{p+1}) = \bigcap_i \ker(\delta^{p+1}_{L_i}).$$
So taking the infimum over $i$ of both sides of (\ref{eqn:noncmpct}) and applying Corollary \ref{cor:intsct} and Lemma \ref{lem:infbetti} completes the proof.
\end{proof}

\begin{cor} \label{cor:juzv}
Let $G$, $L$, $\KK$, and $p$ be as in Theorem \ref{thm:betti}. If $\Img(\delta^p) = \ker(\delta^{p+1})$, or equivalently $H^p(L, \KK) = 0$, then
$$\toph{\Sigma}{G}{C^{p-1}(L, \KK)} + \beta^p_{(2)}(L : G) \cdot \log |\KK| \leq \toph{\Sigma}{G}{\ker(\delta^p)} + \toph{\Sigma}{G}{\Img(\delta^p)}.$$
\end{cor}

By applying one of our later results, Theorem \ref{thm:meas} from \S\ref{sec:malg}, we obtain one additional corollary.

\begin{cor} \label{cor:juzv2}
Let $G$, $L$, $\KK$, and $p$ be as in Theorem \ref{thm:betti}.
\begin{enumerate}
\item[\rm (1)] If $p > 1$, $\Img(\delta^{p-1}) = \ker(\delta^p)$ and $\Img(\delta^p) = \ker(\delta^{p+1})$ (equivalently $H^{p-1}(L, \KK) = 0$ and $H^p(L, \KK) = 0$), then
\begin{align*}
 \ \meash{\Sigma}{G}{C^{p-1}(L, \KK)}{\mathrm{Haar}}& + \beta^p_{(2)}(L : G) \cdot \log |\KK|\\
\leq & \ \meash{\Sigma}{G}{\ker(\delta^p)}{\mathrm{Haar}} + \meash{\Sigma}{G}{\Img(\delta^p)}{\mathrm{Haar}}.
\end{align*}
\item[\rm (2)] If $p = 1$, $\Img(\delta^1) = \ker(\delta^2)$ (equivalently $H^1(L, \KK) = 0$), and $\toph{\Sigma}{G}{\ker(\delta^1)} = 0$, then
$$\meash{\Sigma}{G}{C^0(L, \KK)}{\mathrm{Haar}} + \beta^1_{(2)}(L : G) \cdot \log |\KK| \leq \meash{\Sigma}{G}{\Img(\delta^1)}{\mathrm{Haar}}$$
and $\meash{\Sigma}{G}{\ker(\delta^1)}{\mathrm{Haar}} \in \{-\infty, 0\}$.
\end{enumerate}
In either case, if $\beta^p_{(2)}(L : G) > 0$ then $\delta^p$ violates the Yuzvinsky addition formula for measured sofic entropy.
\end{cor}

\begin{proof}
From Corollary \ref{cor:juzv} we obtain
\begin{equation} \label{eqn:juzvold}
\toph{\Sigma}{G}{C^{p-1}(L, \KK)} + \beta^p_{(2)}(L : G) \cdot \log |\KK| \leq \toph{\Sigma}{G}{\ker(\delta^p)} + \toph{\Sigma}{G}{\Img(\delta^p)}.
\end{equation}
We will apply Theorem \ref{thm:meas} which states that when $G$ acts by continuous group automorphisms on a profinite group having dense homoclinic subgroup, the measured sofic entropy and topological sofic entropy agree. Each of $C^{p-1}(L, \KK)$, $\ker(\delta^p)$, and $\Img(\delta^p)$ are profinite groups, and $C^{p-1}(L, \KK)$ and $\Img(\delta^p)$ are easily seen to have dense homoclinic groups. Thus
\begin{align}
\toph{\Sigma}{G}{C^{p-1}(L, \KK)} & = \meash{\Sigma}{G}{C^{p-1}(L, \KK)}{\mathrm{Haar}} \quad \text{and}\label{eqn:juzvold2}\\
\toph{\Sigma}{G}{\Img(\delta^p)} & = \meash{\Sigma}{G}{\Img(\delta^p)}{\mathrm{Haar}}\label{eqn:juzvold3}
\end{align}
by Theorem \ref{thm:meas}.

When $p > 1$, $\Img(\delta^{p-1})$ also has a dense homoclinic group. So the assumption $H^{p-1}(L, \KK) = 0$ in (1) implies that $\ker(\delta^p) = \Img(\delta^{p-1})$ has dense homoclinic group. So with assumption (1) Theorem \ref{thm:meas} gives
$$\toph{\Sigma}{G}{\ker(\delta^p)} = \meash{\Sigma}{G}{\ker(\delta^p)}{\mathrm{Haar}}$$
and the corollary follows from the above equation and (\ref{eqn:juzvold}, \ref{eqn:juzvold2}, \ref{eqn:juzvold3}).

Now consider $p = 1$. With assumption (2) we have that $\toph{\Sigma}{G}{\ker(\delta^1)} = 0$. Plugging this into (\ref{eqn:juzvold}) we obtain 
$$\toph{\Sigma}{G}{C^{p-1}(L, \KK)} + \beta^p_{(2)}(L : G) \cdot \log |\KK| \leq \toph{\Sigma}{G}{\Img(\delta^p)}.$$
By applying (\ref{eqn:juzvold2}, \ref{eqn:juzvold3}) we get the same inequality with measured sofic entropies. Finally, from the variational principle \cite{KL11a} (recalled in Theorem \ref{thm:varp} below) it follows that $\meash{\Sigma}{G}{\ker(\delta^1)}{\mathrm{Haar}} \in \{-\infty, 0\}$.
\end{proof}

\begin{rem}
In regard to item (2) of Corollary \ref{cor:juzv2}, for a simplicial complex $L$ it is easily seen that $\ker(\delta^1)$ consists of those functions that are constant on each $0$-skeleton of every connected component of $L$. This means that the action $G \acts \ker(\delta^1)$ will have large stabilizers provided $L$ has large connected components, and work of Meyerovitch \cite{Me15} shows that if all stabilizers are infinite then the topological sofic entropy is $0$ for every sofic approximation (see also \cite{S14b} for an alternate proof). Thus one will have $\toph{\Sigma}{G}{\ker(\delta^1)} = 0$ in many natural situations.
\end{rem}

\begin{rem} \label{rem:polygon}
In the case of $2$-dimensional spaces, one can work with the notion of a polygonal complex (i.e. every $2$-cell is a polygon) rather than a simplicial complex. The proof of Theorem \ref{thm:betti} still works in this context. Thus Theorem \ref{thm:betti} and Corollaries \ref{cor:juzv} and \ref{cor:juzv2} are true for $2$-dimensional polygonal complexes.
\end{rem}

\newpage
\section{Groups that fail the Yuzvinsky addition formula} \label{sec:groups}

We now use the results of the previous section to investigate which groups admit actions violating the Yuzvinsky addition formula. For this, we will need the notions of an induced action and a coinduced action.

If $\Gamma \leq G$ and $\Gamma \acts L$ is any action, then we construct the \emph{induced action} $G \acts L \times (G / \Gamma)$ as follows. Let $r : G / \Gamma \rightarrow G$ choose a representative from every $\Gamma$ coset, so that $r(a \Gamma) \Gamma = a \Gamma$. We will also write $r(a)$ for $r(a \Gamma)$. Let $\sigma : (G / \Gamma) \times G \rightarrow \Gamma$ be the cocycle defined by
$$\sigma(a \Gamma, g) = r(g a)^{-1} g r(a).$$
Then $G$ acts on $L \times (G / \Gamma)$ by
$$g \cdot (x, a \Gamma) = (\sigma(a \Gamma, g) \cdot x, g a \Gamma).$$
When $L$ is furthermore a simplicial complex and $\Gamma$ acts on $L$ simplicially, then we turn $L \times (G / \Gamma)$ into a simplicial complex as well by viewing it as a disjoint union of copies of $L$ indexed by $G / \Gamma$. It is easily checked that the induced action is simplicial. We leave it to the reader to verify that, up to a $G$-equivariant bijection, this action does not depend on the choice of $r : G / \Gamma \rightarrow G$.

Now we define coinduced actions. Again let $\Gamma \leq G$ and let $\Gamma \acts X$ be an action. Let $r : G / \Gamma \rightarrow G$ and $\sigma: (G / \Gamma) \times G \rightarrow \Gamma$ be as before. The \emph{coinduced action} is the action of $G$ on $X^{G / \Gamma}$, where for $f \in X^{G / \Gamma}$ and $g \in G$ we define $g \cdot f$ by
$$(g \cdot f)(a \Gamma) = \sigma(a \Gamma, g^{-1})^{-1} \cdot f(g^{-1} a \Gamma).$$
When $X$ is a topological space and $\Gamma$ acts continuously, we give $X^{G / \Gamma}$ the product topology. Again, we leave it to the reader to verify that, up to a $G$-equivariant homeomorphism, this action does not depend on the choice of $r$.

\begin{lem} \label{lem:induce}
Let $G$ be a countable group, $\Gamma$ a subgroup, and $\Gamma \acts L$ a simplicial action. Let $G \acts L \times (G / \Gamma)$ be the induced action.
\begin{enumerate}
\item[\rm (i)] If $L$ is cocompact for $\Gamma$, then $L \times (G / \Gamma)$ is cocompact for $G$.
\item[\rm (ii)] If $\Gamma$ acts freely on $L$ then $G$ acts freely on $L \times (G / \Gamma)$.
\item[\rm (iii)] If $\KK$ is any field and $H^p(L, \KK) = 0$ then $H^p(L \times (G / \Gamma), \KK) = 0$ as well.
\item[\rm (iv)] For every $p$ we have $\beta^p_{(2)}(L : \Gamma) = \beta^p_{(2)}(L \times (G / \Gamma) : G)$.
\item[\rm (v)] For every $p$, $G \acts \ker(\delta^p_{L \times (G / \Gamma)})$ is the coinduced action from $\Gamma \acts \ker(\delta^p_L)$.
\end{enumerate}
\end{lem}

\begin{proof}
(i). It suffices to show that if $X \subseteq L$ is a $\Gamma$-orbit then $X \times (G / \Gamma) \subseteq L \times (G / \Gamma)$ is a $G$-orbit. It is clear that $X \times (G / \Gamma)$ is $G$-invariant. Fix $x, y \in X$ and $a \Gamma \in G / \Gamma$. Let $\gamma \in \Gamma$ satisfy $\gamma \cdot x = y$. Set $g = r(a) \gamma r(\Gamma)^{-1}$ and note that $r(g) = r(a)$. We have
$$\sigma(\Gamma, g) = r(g)^{-1} g r(\Gamma) = r(a)^{-1} g r(\Gamma) = \gamma.$$
Therefore
$$g \cdot (x, \Gamma) = (\gamma \cdot x, g \Gamma) = (y, a \Gamma).$$
As $x, y \in X$ and $a \in G$ were arbitrary, we conclude that $X \times (G / \Gamma)$ is a $G$-orbit.

(ii). Assume $\Gamma$ acts freely on $L$. From the proof of (i) we see that every $G$-orbit of $L \times (G / \Gamma)$ meets $L \times \{\Gamma\}$. So fix any $x \in L$ and $g \in G$. It suffices to show that if $g \cdot (x, \Gamma) = (x, \Gamma)$ then $g = 1_G$. We immediately have $g \Gamma = \Gamma$ so $g \in \Gamma$. So $\sigma(\Gamma, g) = r(\Gamma)^{-1} g r(\Gamma)$. As $g \cdot (x, \Gamma) = (x, \Gamma)$ we must have $r(\Gamma)^{-1} g r(\Gamma) \cdot x = x$ whence $g = 1_G$ since $g \in \Gamma$ and $\Gamma$ acts freely.

(iii). Note that this statement is unrelated to the actions of $G$ and $\Gamma$. As a simplicial complex, $L \times (G / \Gamma)$ is a countable disjoint union of copies of $L$. So $H^p(L \times (G / \Gamma), \KK) = \prod_{G / \Gamma} H^p(L, \KK)$.

(iv). Write $L$ as an increasing union $L = \cup L_i$ of $\Gamma$-invariant cocompact subcomplexes. Then $L \times (G / \Gamma) = \cup_i L_i \times (G / \Gamma)$ and by (i) each $L_i \times (G / \Gamma)$ is cocompact. The inclusions $L_i \subseteq L_j$ induce surjections of the cochains $\pi^{L_j}_{L_i} : C_{(2)}^*(L_j) \rightarrow C_{(2)}^*(L_i)$ for $j \geq i$. Each $\pi^{L_j}_{L_i}$ descends to a map $\bar{H}^*_{(2)}(L_j) \rightarrow \bar{H}^*_{(2)}(L_i)$. By definition
$$\beta^p_{(2)}(L : \Gamma) = \limsup_{i \rightarrow \infty} \liminf_{j \rightarrow \infty} \dim_\Gamma \overline{\Img}\Big( \bar{H}^p_{(2)}(L_j) \rightarrow \bar{H}^p_{(2)}(L_i) \Big).$$
It is straightforward to see that at the induced level we have $\pi^{L_j \times (G / \Gamma)}_{L_i \times (G / \Gamma)} = \prod_{G / \Gamma} \pi^{L_j}_{L_i}$. Similarly, as a Hilbert $G$-module,
$$\overline{\Img}\Big( \bar{H}^p_{(2)}(L_j \times (G / \Gamma)) \rightarrow \bar{H}^p_{(2)}(L_i \times (G / \Gamma)) \Big)$$
is isomorphic to the Hilbert $G$-module induced from $\Gamma \acts \overline{\Img}(\bar{H}^p_{(2)}(L_j) \rightarrow \bar{H}^p_{(2)}(L_i))$ (see \cite[Section 1.1.5]{Luc} for definitions). So by the reciprocity formula for von Neumann dimension (see \cite[Lemma 1.24]{Luc} or \cite[p. 194 property (1.3)]{CG86})
 we have
$$\dim_G \overline{\Img}\Big( \bar{H}^p_{(2)}(L_j \times (G / \Gamma)) \rightarrow \bar{H}^p_{(2)}(L_i \times (G / \Gamma)) \Big) = \dim_\Gamma \overline{\Img}\Big( \bar{H}^p_{(2)}(L_j) \rightarrow \bar{H}^p_{(2)}(L_i) \Big).$$
Therefore $\beta^p_{(2)}(L \times (G / \Gamma) : G) = \beta^p_{(2)}(L : \Gamma)$ as claimed.

(v). There is a natural homeomorphism $\phi: C^{p-1}(L \times (G / \Gamma), \KK) \rightarrow C^{p-1}(L, \KK)^{G / \Gamma}$ defined by
$$\phi(x)(a \Gamma)(c) = x(c, a \Gamma)$$
for $a \in G$ and $c$ a $(p-1)$-cell of $L$. It is easily seen that this bijection takes $\ker(\delta^p_{L \times (G / \Gamma)})$ to $\ker(\delta^p_L)^{G / \Gamma}$, so we only need to check that this bijection respects the appropriate group actions. Indeed, using the natural shift action of $G$ on $\ker(\delta^p_{L \times (G / \Gamma)})$ we have
$$\phi(g \cdot x)(a \Gamma)(c) = (g \cdot x)(c, a \Gamma) = x(g^{-1} \cdot (c, a \Gamma)) = x(\sigma(a \Gamma, g^{-1}) \cdot c, g^{-1} a \Gamma),$$
while using the coinduced action on $\ker(\delta^p_L)^{G / \Gamma}$ gives
\begin{align*}
[g \cdot \phi(x)](a \Gamma)(c) & = \Big( \sigma(a \Gamma, g^{-1})^{-1} \cdot [\phi(x)(g^{-1} a \Gamma)] \Big) (c)\\
 & = \phi(x)(g^{-1} a\Gamma)(\sigma(a \Gamma, g^{-1}) \cdot c)\\
 & = x(\sigma(a \Gamma, g^{-1}) \cdot c, g^{-1} a \Gamma).
\end{align*}
Thus $\phi$ is a $G$-equivariant bijection between $\ker(\delta^p_{L \times (G / \Gamma)})$ and $\ker(\delta^p_L)^{G / \Gamma}$.
\end{proof}

\begin{thm}[Failure of the Yuzvinsky formula]
\label{thm:fail2}
Let $G$ be a sofic group containing an infinite subgroup $\Gamma$ such that $\Gamma$ has some non-zero $\ell^2$-Betti number $\beta^p_{(2)}(\Gamma) > 0$ and admits a free cocompact action on $p$-connected simplicial complex (for instance, if $\Gamma$ has a finite classifying space). Then $G$ admits an algebraic action and an algebraic factor map that simultaneously violates the Yuzvinsky addition formula for both measured (with respect to Haar probability measures) and topological sofic entropy for all sofic approximations to $G$.
\end{thm}

We remark that a weaker assumption suffices. We only need a finite field $\KK$, a $p \geq 1$, and a free simplicial action $\Gamma \acts L$ satisfying $H^{p-1}(L, \KK) = 0$ (if $p > 1$), $H^p(L, \KK) = 0$, $\beta^p_{(2)}(L : \Gamma) > 0$, and with the $\Gamma$-action on the $p$-skeleton of $L$ cocompact.

\begin{proof}
The assumption that $\Gamma$ is infinite implies $p \geq 1$. Fix a finite field $\KK$. Consider a $p$-connected simplicial free cocompact $\Gamma$-complex $L'$.

First consider the case $p > 1$. Since $L'$ is $p$-connected we have $H^{p-1}(L', \KK) = 0$, $H^p(L', \KK) = 0$, and $\beta^p_{(2)}(L' : \Gamma) = \beta^p_{(2)}(\Gamma) > 0$. Letting $G \acts L$ be the induced simplicial action of $G$, we have that $G$ acts freely, $L$ is cocompact, $H^{p-1}(L, \KK) = 0$, $H^p(L, \KK) = 0$, and $\beta^p_{(2)}(L : G) = \beta^p_{(2)}(L' : \Gamma) > 0$ by Lemma \ref{lem:induce}. Now for any choice of sofic approximation $\Sigma$ to $G$, Corollaries \ref{cor:juzv} and \ref{cor:juzv2} immediately imply that $\delta^p$ violates the Yuzvinsky addition formula for $\Sigma$ for both topological and measure-theoretic entropies.

Now consider the case $p = 1$. Since $L'$ is $1$-connected we have $H^1(L', \KK) = 0$, and $\beta^1_{(2)}(L' : \Gamma) = \beta^1_{(2)}(\Gamma) > 0$. Letting $G \acts L$ be the induced simplicial action of $G$, we have that $G$ acts freely, $L$ is cocompact, $H^1(L, \KK) = 0$, and $\beta^1_{(2)}(L : G) = \beta^1_{(2)}(L' : \Gamma) > 0$ by Lemma \ref{lem:induce}. Since $L'$ is $1$-connected it is connected, so $\ker(\delta^1_{L'})$ consists of the constant functions and $\Gamma$ acts trivially (i.e. it fixes every point in $\ker(\delta^1_{L'})$). It is well known that trivial actions have sofic topological entropy $0$, and it is known that coinduction preserves sofic topological entropy \cite[Prop. 6.21]{H14}. Thus from Lemma \ref{lem:induce}.(v) it follows that $\toph{\Sigma}{G}{\ker(\delta^1_L)} = 0$ for every sofic approximation $\Sigma$ to $G$. So for every choice of sofic approximation $\Sigma$, Corollaries \ref{cor:juzv} and \ref{cor:juzv2} imply that $\delta^1$ violates the Yuzvinsky addition formula for $\Sigma$ for both topological and measure-theoretic entropies.
\end{proof}

With more work, we are able to strengthen the above theorem in the topological case.

\begin{thm}[Failure of the Yuzvinsky formula, topological]
 \label{thm:fail}
Let $G$ be a sofic group containing an infinite subgroup $\Gamma$ with some non-zero $\ell^2$-Betti number $\beta^p_{(2)}(\Gamma) > 0$. Then $G$ admits an algebraic action and an algebraic factor map that simultaneously violates the Yuzvinsky addition formula for topological sofic entropy for all sofic approximations to $G$.
\end{thm}

We remark that a weaker assumption suffices. We only need a finite field $\KK$, a $p \geq 1$, and a free simplicial action $\Gamma \acts L$ satisfying $H^p(L, \KK) = 0$ and $\beta^p_{(2)}(L : \Gamma) > 0$.

\begin{proof}
The assumption that $\Gamma$ is infinite implies $p \geq 1$. Fix any finite field $\KK$ and any sofic approximation $\Sigma$ to $G$. Consider a contractible simplicial free $\Gamma$-complex $L'$. We have $H^p(L', \KK) = 0$ and $\beta^p_{(2)}(L' : \Gamma) = \beta^p_{(2)}(\Gamma) > 0$. Letting $G \acts L$ be the induced simplicial action, we have that $G$ acts freely, $H^p(L, \KK) = 0$, and $\beta^p_{(2)}(L : G) = \beta^p_{(2)}(L' : \Gamma) > 0$ by Lemma \ref{lem:induce}.

For a subcomplex $R \subseteq L$ we write $\delta^i_R$ for the $i^\text{th}$-coboundary map on $R$. We write $C^i(R, \Z)$ and $C^i(R, \KK)$ for the set of $i$-cochains of $R$ with $\Z$ coefficients and $\KK$ coefficients respectively. We write $C^i_{(2)}(R)$ for the space of $\ell^2$-summable $i$-cochains of $R$. When we wish to refer to $C^i(R, \Z)$, $C^i(R, \KK)$, and $C^i_{(2)}(R)$ simultaneously, we simply write $C^i(R)$. If $R \subseteq S$ are two sub-complexes of $L$, then the inclusion $R \subseteq S$ induces morphisms $\pi^S_R$ for the cochains:
\begin{equation} \label{eqn:cochains}
\begin{array}{ccccc}
C^{i-1}(S) & \overset{\delta^i_S}{\longrightarrow} & C^i(S) & \overset{\delta^{i+1}_S}{\longrightarrow} & C^{i+1}(S)\\
[3pt]\pi^S_R \downarrow \ \uparrow \zeta^S_R & \raisebox{4pt}{$\circlearrowleft$} & \pi^S_R \downarrow \ \uparrow \zeta^S_R & \raisebox{4pt}{$\circlearrowleft$} & \pi^S_R \downarrow \ \uparrow \zeta^S_R\\
[-5pt]C^{i-1}(R) & \overset{\delta^i_R}{\longrightarrow} & C^i(R) & \overset{\delta^{i+1}_R}{\longrightarrow} & C^{i+1}(R).
\end{array}
\end{equation}
Since the diagram is commutative, it is easily checked that $\Img(\delta^i_R) = \pi^S_R(\Img(\delta^i_S))$. However for kernels we only have $\pi^S_R(\ker(\delta^i_S)) \subseteq \ker(\delta^i_R)$. For $R \subseteq S$ there is a natural isomorphism of $C^i(R)$ with the collection of elements of $C^i(S)$ that are identically $0$ on the $i$-simplices that are in $S$ but not $R$. We let $\zeta^S_R : C^i(R) \rightarrow C^i(S)$ denote the associated injection. For any two subcomplexes $R$ and $S$, we also let $C^i(S \setminus R)$ denote the set of $i$-cochains on $S$ that are identically $0$ on the $i$-simplices in $R \cap S$.

Write $L$ as an increasing union of $G$-invariant cocompact subcomplexes $L = \cup_i L_i$. Each $\pi^{L_j}_{L_i}$ produces a homomorphism $\bar{H}^{p}_{(2)}(L_j) \to \bar{H}^{p}_{(2)}(L_i)$. Recall that the $p$-th $\ell^2$-Betti number of the action of $G$ on $L$ is
\begin{align*}
\beta^p_{(2)}(L : G) & = \limsup_{i\to \infty} \liminf_{j \to \infty}  \dim_G \overline{\Img}\, \bigl(\bar{H}^{p}_{(2)}(L_j)\to \bar{H}^{p}_{(2)}(L_i)\bigr)\\
 & = \limsup_{i \rightarrow \infty} \liminf_{j \rightarrow \infty} \dim_G \overline{\pi^{L_j}_{L_i}(\ker(\delta^{p+1}_{L_j}))} - \dim_G \overline{\Img}(\delta^p_{L_i}).
\end{align*}
Thus, for every $\epsilon>0$ there is a large enough $i$ such that for all $j\geq i$, 
\begin{equation*}
\dim_G \overline{\pi^{L_j}_{L_i}(\ker(\delta^{p+1}_{L_j}))} \geq \dim_G \overline{\Img}(\delta^p_{L_i}) + \beta^p_{(2)}(L : G)-\epsilon.
\end{equation*}
Fix $0 < \epsilon < \beta^p_{(2)}(L : G)$ and $i$ as above.

Although we will not need this fact, from $H^p(L, \KK) = 0$ one can deduce that (with coefficients in $\KK$)
$$\Img(\delta^{p}_{L_i}) = \pi^L_{L_i}(\Img(\delta^{p}_L)) = \pi^L_{L_i}(\ker \delta^{p+1}_L) = \bigcap_{j \geq i} \pi^{L_j}_{L_i}(\ker \delta^{p+1}_{L_j}).$$
This suggests that we should compute the entropy of each $\pi^{L_j}_{L_i}(\ker \delta^{p+1}_{L_j})$ and apply Corollary \ref{cor:intsct}. Indeed, it is quite plausible that this would relate the entropy of $\Img(\delta^{p}_{L_i})$ to the quantity $\beta^p_{(2)}(L : G) - \epsilon$. However, it is evident from \S\ref{sec:talg} that our entropy methods are best adapted to the case of kernels of maps, rather than \emph{projections} of kernels. We are unable to obtain a lower bound to the entropy of $\pi^{L_j}_{L_i}(\ker \delta^{p+1}_{L_j})$ and thus we will take a different, though conceptually very similar, path.

For a subgroup $N$ of an abelian group $A$, write $[N]$ for the set of $a \in A$ such that there is $k \in \Z$ with $0 \neq k \cdot a \in N$. Note that if $A$ is torsion-free then $A / [N]$ is torsion-free.

Consider a finite subcomplex $R \subseteq L$ and $i \in \N$. Define $D_{i, R}^\Z : C^p(L_i, \Z) \rightarrow C^{p+1}(R, \Z)$ by
$$D_{i, R}^\Z(x) = \delta^{p+1}_R \circ \zeta^R_{R \cap L_i} \circ \pi^{L_i}_{R \cap L_i}(x).$$
Note that $D_{i, R}^\Z$ is continuous as it only depends upon the restriction of $x$ to the finitely many $p$-simplices in $R \cap L_i$. Let $Q_{i, R}^\Z : C^{p+1}(R, \Z) \rightarrow C^{p+1}(R, \Z) / [N_{i, R}^\Z]$ be the quotient map, where
\begin{equation*}
N_{i, R}^\Z = \delta^{p+1}_R(C^p(R \setminus L_i, \Z)).
\end{equation*}
Note that $Q_{i,R}^\Z$ is also continuous, as it is a map between countable discrete groups. The various spaces and maps are pictured below.
\begin{equation*}C^p(L_i, \Z)   \underbrace{\underbrace{\overset{\pi^{L_i}_{R \cap L_i}}{\longrightarrow}  C^{p}(R \cap L_i, \Z)   \overset{\zeta^R_{R \cap L_i}}{\longrightarrow}  C^{p}(R , \Z)   \overset{\delta^{p+1}_R}{\longrightarrow}   }_{D_{i, R}^\Z}
C^{p+1}(R , \Z)  \overset{Q_{i, R}^\Z}{\longrightarrow}}_{Q_{i, R}^\Z \circ D_{i, R}^\Z}   C^{p+1}(R, \Z)/ [N_{i, R}^\Z] 
\end{equation*}
Let $\psi_{i,R}^\Z$ be the $G$-equivariant map induced by $Q_{i, R}^\Z \circ D_{i, R}^\Z$, meaning that 
$$\psi_{i,R}^\Z: \left(\begin{array}{ccl}C^p(L_i, \Z) & \rightarrow & \bigl(C^{p+1}(R, \Z) / [N_{i, R}^\Z]\bigr)^G 
\\
x & \mapsto  & \bigl(\ \underbrace{Q_{i, R}^\Z \circ D_{i, R}^\Z(g^{-1} \cdot x)}_{\psi_{i,R}^\Z(x)(g):=}\ \bigr)_{g\in G}\end{array}\right)$$
Note that continuity of $\psi_{i,R}^\Z$ follows from the continuity of $Q_{i,R}^\Z \circ D_{i,R}^\Z$.

In the $\ell^2$ and $\KK$-coefficient settings, define the analogous objects $D_{i, R}^{(2)}$, $D_{i, R}^\KK$, $Q_{i, R}^{(2)}$, $Q_{i, R}^\KK$, $N_{i, R}^{(2)}$, $N_{i, R}^\KK$, $\psi_{i,R}^{(2)}$, and $\psi_{i,R}^\KK$. Note that $N_{i, R}^{(2)} = [N_{i, R}^{(2)}]$ and $N_{i, R}^\KK = [N_{i, R}^\KK]$. It is easily seen that up to isomorphism
$$C^p(L_i, \KK) = C^p(L_i, \Z) \otimes \KK, \ C^{p+1}(R, \KK) = C^{p+1}(R, \Z) \otimes \KK, \ C^p(R \setminus L_i, \KK) = C^p(R \setminus L_i, \Z) \otimes \KK$$
$$(\delta^{p+1}_R)^\KK = (\delta^{p+1}_R)^\Z \otimes \id_\KK, \text{ and } D_{i,R}^\KK = D_{i,R}^\Z \otimes \id_\KK.$$
Letting $\iota : N_{i,R}^\Z \rightarrow C^{p+1}(R, \Z)$ denote the inclusion map, we have
\begin{align*}
[N_{i,R}^\KK] = N_{i,R}^\KK = \delta^{p+1}(C^p(R \setminus L_i, \KK)) & = (\delta^{p+1} \otimes \id_\KK)(C^p(R \setminus L_i, \Z) \otimes \KK)\\
 & = (\iota \times \id_\KK)(N_{i,R}^\Z \otimes \KK) = (\iota \times \id_\KK)([N_{i,R}^\Z] \otimes \KK),
\end{align*}
where the last equality is due to the fact that $\KK$ is a field. With $\Z$ coefficients we have the exact sequence
\begin{equation*}
\begin{array}{ccccccccc}
0 & \longrightarrow & [N_{i,R}^\Z] & \overset{\iota}{\longrightarrow} & C^{p+1}(R, \Z) & \overset{Q_{i,R}^\Z}{\longrightarrow} & \frac{C^{p+1}(R, \Z)}{[N_{i,R}^\Z]} & \longrightarrow & 0.
\end{array}
\end{equation*}
Since the tensor product is right exact, we obtain the exact sequence
\begin{equation*}
\begin{array}{ccccccc}
[N_{i,R}^\Z] \otimes \KK & \overset{\iota \otimes \id_\KK}{\longrightarrow} & C^{p+1}(R, \Z) \otimes \KK & \overset{Q_{i,R}^\Z \otimes \id_\KK}{\longrightarrow} & \frac{C^{p+1}(R, \Z)}{[N_{i,R}^\Z]} \otimes \KK & \longrightarrow & 0.
\end{array}
\end{equation*}
Thus
$$\ker(Q_{i,R}^\Z \otimes \id_\KK) = (\iota \otimes \id_\KK)([N_{i,R}^\Z] \otimes \KK) = [N_{i,R}^\KK] = \ker(Q_{i,R}^\KK)$$
So we conclude that $Q_{i,R}^\KK = Q_{i,R}^\Z \otimes \id_\KK$. It then follows that $\psi_{i, R}^\KK = \psi_{i, R}^\Z \otimes \id_\KK$ as well.

Let $n(i)$ be the number of $G$-orbits of $p$-simplices in $L_i$ and let $m(i, R)$ be the rank of the free abelian group $C^{p+1}(R, \Z) / [N_{i, R}^\Z]$. By fixing an ordered set of $n(i)$-many $p$-simplices of $L_i$ lying in distinct orbits and by picking a basis for $C^{p+1}(R, \Z) / [N_{i, R}^\Z]$, we obtain $G$-equivariant isomorphisms of $C^p(L_i, \Z)$ with $(\Z^{n(i)})^G$ and of $(C^{p+1}(R, \Z) / [N_{i, R}^\Z])^G$ with $(\Z^{m(i, R)})^G$. Using the same bases, in the $\ell^2$ and $\KK$ settings we obtain isomorphisms with $(\ell^2(G))^{n(i)}$, $(\ell^2(G))^{m(i, R)}$, $(\KK^{n(i)})^G$, and $(\KK^{m(i, R)})^G$. Since $\psi_{i,R}^\Z$ is continuous, we can apply Lemma \ref{lem:equiv} to obtain a matrix
$$M_{i,R} \in \Mat{n(i)}{m(i, R)}{\Z[G]}$$
so that, under these isomorphisms, $\psi_{i,R}^\Z = M_{i, R}^\Z$. It is clear from the definitions that $\psi_{i, R}^{(2)} = M_{i, R}^{(2)}$. Since $\psi_{i, R}^\KK = \psi_{i, R}^\Z \otimes \id_\KK$, we also have $\psi_{i, R}^\KK = M_{i, R}^\KK$.

\vspace{0.2in}
\noindent
\underline{Claim:} With coefficients in $\KK$ we have
\begin{equation} \label{eqn:juzfail1}
\Img(\delta_{L_i}^p) = \bigcap_{\substack{R \subseteq L\\R \text{ finite}}} \ker(\psi_{i, R}^\KK).
\end{equation}

\noindent
\textit{Proof of Claim:}
First note that if $x \in \Img(\delta_{L_i}^p)$ then $\pi^{L_i}_{R \cap L_i}(x) \in \Img(\delta_{R \cap L_i}^p)$. Say $\pi^{L_i}_{R \cap L_i}(x) = \delta_{R \cap L_i}^p(y)$. Set
$$z = \delta^p_R \circ \zeta^R_{R \cap L_i}(y) - \zeta^R_{R \cap L_i} \circ \pi^{L_i}_{R \cap L_i}(x) \in C^p(R \setminus L_i, \KK).$$
Then
$$0 = \delta^{p+1}_R \circ \delta^p_R \circ \zeta^R_{R \cap L_i}(y) = \delta^{p+1}_R(z) + \delta^{p+1}_R \circ \zeta^R_{R \cap L_i} \circ \pi^{L_i}_{R \cap L_i}(x) = \delta^{p+1}_R(z) + D_{i, R}^\KK(x).$$
Therefore $x \in \Img(\delta_{L_i}^p)$ implies $Q_{i, R}^\KK \circ D_{i, R}^\KK (x) = 0$. Thus, by $G$-invariance of $\Img(\delta_{L_i}^p)$, we have
 $\Img(\delta_{L_i}^p) \subseteq \ker(\psi_{i, R}^\KK)$.

Now for the converse direction assume that $x \in \bigcap_R \ker(\psi_{i, R}^\KK)$. Write $L$ as an increasing union $L = \cup_n R_n$ of finite subcomplexes $R_n$. For each $n$ choose $z_n \in C^p(R_n \setminus L_i, \KK)$ with
\begin{equation} \label{eqn:juzfail0}
0 = D_{i, R_n}(x) + \delta^{p+1}_{R_n}(z_n) = \delta^{p+1}_{R_n} \circ \zeta^{R_n}_{R_n \cap L_i} \circ \pi^{L_i}_{R_n \cap L_i}(x) + \delta^{p+1}_{R_n}(z_n).
\end{equation}
By compactness there is an accumulation point $z \in C^p(L \setminus L_i, \KK)$ of $\zeta^L_{R_n}(z_n)$. From (\ref{eqn:juzfail0}) we obtain
$$0 = \delta^{p+1}_L \circ \zeta^L_{L_i}(x) + \delta^{p+1}_L(z) = \delta^{p+1}_L(\zeta^L_{L_i}(x) + z).$$
So $\zeta^L_{L_i}(x) + z \in \ker(\delta^{p+1}_L)$. Since $H^p(L, \KK) = 0$, we must have $\zeta^L_{L_i}(x) + z \in \Img(\delta^p_L)$. Therefore
$$x = \pi^L_{L_i}(\zeta^L_{L_i}(x) + z) \in \pi^L_{L_i}(\Img(\delta^p_L)) = \Img(\delta^p_{L_i}).$$
This proves the claim.\hfill [Proof of Claim] $\blacksquare$

\vspace{0.2in}
\noindent
\underline{Claim:} If $j \geq i$ and $R \subseteq L_j$ then in the $\ell^2$-setting we have $\pi^{L_j}_{L_i}(\ker \delta^{p+1}_{L_j}) \subseteq \ker(\psi_{i, R}^{(2)})$.

\vspace{0.1in}
\noindent
\textit{Proof of Claim:}
Consider $y \in \ker \delta^{p+1}_{L_j}$. Set $x = \pi^{L_j}_{L_i}(y)$ and set
$$z = \pi^{L_j}_R(y) - \zeta^R_{R \cap L_i} \circ \pi^{L_i}_{R \cap L_i}(x) \in C^p_{(2)}(R \setminus L_i).$$
Noting that $\pi^{L_j}_R(y) \in \ker \delta^{p+1}_R$, we have
$$D_{i, R}^{(2)}(x) = \delta^{p+1}_R \circ \zeta^R_{R \cap L_i} \circ \pi^{L_i}_{R \cap L_i}(x) = \delta^{p+1}_R \circ \pi^{L_j}_R(y) - \delta^{p+1}_R(z) = - \delta^{p+1}_R(z) \in [N_{i, R}^{(2)}].$$
Thus $Q_{i, R}^{(2)} \circ D_{i, R}^{(2)}(x) = 0$. We conclude, by $G$-invariance of  $\ker \delta^{p+1}_{L_j}$ and $G$-equivariance of $\pi^{L_j}_{L_i}$, that $\pi^{L_j}_{L_i}(\ker \delta^{p+1}_{L_j}) \subseteq \ker(\psi_{i, R}^{(2)})$.\hfill [Proof of Claim] $\blacksquare$
\vspace{0.1in}

By (\ref{eqn:juzfail1}) and Corollary \ref{cor:intsct} we have
\begin{equation} \label{eqn:juzfail2}
\toph{\Sigma}{G}{\Img(\delta_{L_i}^p)} = \inf_R \toph{\Sigma}{G}{\ker(\psi_{i,R}^\KK)}.
\end{equation}
Since $\psi_{i, R}^\KK = M_{i, R}^\KK$ and $\psi_{i, R}^{(2)} = M_{i, R}^{(2)}$, Lemma \ref{lem:dim} gives
\begin{equation} \label{eqn:juzfail3}
\toph{\Sigma}{G}{\ker(\psi_{i, R}^\KK)} \geq \log|\KK| \cdot \dim_G \ker(\psi_{i, R}^{(2)}).
\end{equation}
Since $R$ is finite, there is a $j \geq i$ with $R \subseteq L_j$. So by the second claim above

\begin{equation} \label{eqn:juzfail4}
\log|\KK| \cdot \dim_G \ker(\psi_{i, R}^{(2)}) \geq \log|\KK| \cdot \dim_G \overline{\pi^{L_j}_{L_i} (\ker(\delta_{L_j}^{p+1}))}.
\end{equation}
By our choice of $i$ we have
\begin{equation} \label{eqn:juzfail5}
\log|\KK| \cdot \dim_G \overline{\pi^{L_j}_{L_i} (\ker(\delta_{L_j}^{p+1}))} \geq \log|\KK| \cdot \dim_G \overline{\Img}(\delta_{L_i}^p) + \log|\KK| \beta^p_{(2)}(L : G) - \epsilon \log|\KK|.
\end{equation}
Putting together (\ref{eqn:juzfail2}, \ref{eqn:juzfail3}, \ref{eqn:juzfail4}, \ref{eqn:juzfail5}), we obtain
\begin{equation} \label{eqn6}
\toph{\Sigma}{G}{\Img(\delta_{L_i}^p)} \geq \log|\KK| \cdot \dim_G \overline{\Img}(\delta_{L_i}^p) + \log|\KK| \beta^p_{(2)}(L : G) - \epsilon \log|\KK|.
\end{equation}

Finally, as $\delta_{L_i}^p$ is given by convolution by a matrix with coefficients in $\Z[G]$, we have $\toph{\Sigma}{G}{\ker(\delta_{L_i}^p)} \geq \log|\KK| \cdot \dim_G \ker(\delta_{L_i}^p)$ by Lemma \ref{lem:dim}. Thus
\begin{align*}
& \toph{\Sigma}{G}{\ker(\delta_{L_i}^p)} + \toph{\Sigma}{G}{\Img(\delta_{L_i}^p)}\\
& \quad \geq \log|\KK| \cdot \dim_G \ker(\delta_{L_i}^p) + \log|\KK| \cdot \dim_G \overline{\Img}(\delta_{L_i}^p) + \log|\KK| \beta^p_{(2)}(L : G) - \epsilon \log|\KK|\\
 & \quad = \log|\KK| \cdot \dim_G C^{p-1}_{(2)}(L_i) + \log|\KK| \beta^p_{(2)}(L : G) - \epsilon \log|\KK|\\
 & \quad = \toph{\Sigma}{G}{C^{p-1}(L_i, \KK)} + \log|\KK| \beta^p_{(2)}(L : G) - \epsilon \log|\KK|\\
 & \quad > \toph{\Sigma}{G}{C^{p-1}(L_i, \KK)}.\qedhere
\end{align*}
\end{proof}

\newpage
\section{Definitions of measured sofic entropy and Rokhlin entropy} \label{sec:ment}

A function $\alpha : X \rightarrow K$ is \emph{finer} than another function $\beta : X \rightarrow L$, written $\alpha \geq \beta$, if there is a map $\beta_\alpha : K \rightarrow L$ such that $\beta = \beta_\alpha \circ \alpha$. In this situation we will abuse notation and also let $\beta_\alpha$ denote the product map $\beta_\alpha^{D} : K^{D} \rightarrow L^{D}$.

Let $G$ be a sofic group, and let $\Sigma = (\sigma_n : G \rightarrow \Sym(D_n))$ be a sofic approximation to $G$. Let $G \acts (X, \mu)$ be a {\pmp} action. We now present the definition of measured sofic entropy. The definition is similar to that of topological entropy, but involves accounting for the frequencies of various patterns. Specifically, a Borel function $\alpha : X \rightarrow K$ and a finite set $F \subseteq G$ 
delivers for every $x\in X$ a pattern $p:F\to K, \ f\mapsto \alpha(f\cdot x)$. Similarly a function  $a : D_n \rightarrow K$ delivers for every $\delta\in D_n$ a pattern $p:F\to K, \ f\mapsto a(\sigma_n(f)(\delta))$.
These data thus define partitions of $X$ (resp. $D_n$), according to the associated pattern, into the following pieces:
\begin{eqnarray*}
U_{\alpha,F}(p)&:=&\{x\in X: \forall f\in F, \ \alpha(f\cdot x)=p(f)\}
\\
U_{a,F,n}(p)&:=&\{\delta\in D_n: \forall f\in F, \ a(\sigma_n(f)(\delta))=p(f)\}
\end{eqnarray*}
We equip the finite set $D_n$ with the normalized counting measure $\mu_n(A)=\frac{| A |}{| D_n |}$, and for each $p$ we will compare the measures of the pieces $\mu(U_{\alpha, F}(p))$ and $\mu_n(U_{a,F,n}(p))$. Roughly speaking, the measured sofic entropy is the exponential growth rate of the number of functions $a\in K^{D_n}$ for which these measures are quite similar.
More precisely, for every  $\epsilon > 0$, let 
\begin{equation}\label{eq:def M mu}
\AM_\mu(\alpha, F, \epsilon, \sigma_n):=\bigl\{a \in K^{D_n}: \forall p \in K^F, \ 
\left| \mu(U_{\alpha, F}(p)) - \mu_n(U_{a,F, n}(p)) 
\right| \leq \epsilon\bigr\}.
\end{equation}
The \emph{measured $\Sigma$-entropy} of $G \acts (X, \mu)$ is then defined to be
\begin{equation*}
\meash{\Sigma}{G}{X}{\mu} = \sup_\beta \inf_{\alpha \geq \beta} \inf_{\epsilon > 0} \inf_{\substack{F \subseteq G\\F \text{ finite}}} \limsup_{n \rightarrow \infty} \frac{1}{\vert D_n \vert} \cdot \log \Big| \beta_\alpha \circ \AM_\mu(\alpha, F, \epsilon, \sigma_n) \Big|,
\end{equation*}
where $\alpha$ and $\beta$ range over finite-valued Borel functions. Alternatively, if $\mathcal{A}$ is an algebra that is generating in the sense that for all $x \neq y$ there is $g \in G$ and $A \in \mathcal{A}$ with $g \cdot x \in A$ and $g \cdot y \not\in A$, then in the definition above one can restrict to finite-valued $\mathcal{A}$-measurable functions $\alpha$ and $\beta$ and obtain the same entropy value \cite[Theorem 2.6]{Ke13}.

The above definition is due to Kerr \cite{Ke13} and is equivalent to the definitions of Kerr--Li \cite{KL11a, KL}, all of which generalize the original definition of measured sofic entropy due to Bowen \cite{B10b}. As with topological sofic entropy, the measured sofic entropy may depend on $\Sigma$, and either $\meash{\Sigma}{G}{X}{\mu} \geq 0$ or else $\meash{\Sigma}{G}{X}{\mu} = - \infty$. Also, when the acting group is amenable, the measured sofic entropy coincides with the classical Kolmogorov--Sinai entropy for all choices of $\Sigma$ \cite{Ba, KL}.

Just as in the classical case, measured sofic entropy and topological sofic entropy are related via the variational principle.

\begin{thm}[Variational principle, Kerr--Li \cite{KL11a}]
\label{thm:varp}
Let $G$ be a sofic group and let $\Sigma$ be a sofic approximation to $G$. Let $X$ be a compact metrizable space and let $G \acts X$ be a continuous action. Then
$$\toph{\Sigma}{G}{X} = \sup_\mu \ \ \meash{\Sigma}{G}{X}{\mu},$$
where the supremum is taken over all $G$-invariant Borel probability measures and the supremum is $- \infty$ if there are no such measures.
\end{thm}

We will soon introduce Rokhlin entropy, but first we must discuss the notion of Shannon entropy for a countable-valued function. Let $(X, \mu)$ be a standard probability space. The \emph{Shannon entropy} of a countable-valued function $\alpha: X \rightarrow K$ is
$$\sH(\alpha) = \sum_{k \in K} - \mu(\alpha^{-1}(k)) \cdot \log \mu(\alpha^{-1}(k)).$$
When we need to clarify the measure being used, we write $\sH_\mu(\alpha)$. If $\cQ$ is a countable partition of $X$ then the \emph{conditional Shannon entropy} of $\alpha$ given $\cQ$ is
\begin{equation}\label{eq:conditional Shannon entropy}
\sH(\alpha | \cQ) = \sum_{Q \in \cQ} \mu(Q) \cdot \sH_Q(\alpha),
\end{equation}
where we write $\sH_Q(\alpha)$ for $\sH_{\mu_Q}(\alpha)$, where $\mu_Q(A) = \frac{\mu(A \cap Q)}{\mu(Q)}$ for Borel $A \subseteq X$. More generally, if $\calF$ is a sub-$\sigma$-algebra of $X$ then the \emph{conditional Shannon entropy} of $\alpha$ given $\calF$ is
\begin{equation}\label{eq:cond Shannon as inf}
\sH(\alpha | \calF) = \inf_{\cQ \subseteq \calF} \sH(\alpha | \cQ),
\end{equation}
where the infimum is over all countable $\calF$-measurable partitions $\cQ$ of $X$. These definitions are not the usual ones, but they are equivalent (see \cite{Do11}). We mention three properties of Shannon entropy that we will need. These properties of Shannon entropy are well known and can be found in \cite{Do11}.

\begin{lem} \label{lem:shan}
Let $\alpha, \beta : X \rightarrow \N$ be Borel functions and let $\calF$ be a sub-$\sigma$-algebra.
\begin{enumerate}
\item[\rm (i)] $\sH(\alpha) \leq \log(k)$ if $\alpha$ takes only $k$-many values.
\item[\rm (ii)] $\sH(\alpha \times \beta | \calF) \leq \sH(\alpha | \calF) + \sH(\beta | \calF)$.
\item[\rm (iii)] If $\beta \subseteq \calF$ then $\sH(\alpha \times \beta | \calF) = \sH(\alpha | \calF)$.
\end{enumerate}
\end{lem}

For a {\pmp} action $G \acts (X, \mu)$, the \emph{Rokhlin entropy} is defined as
$$\rh{G}{X}{\mu} = \inf \Big\{ \sH(\alpha | \cJ) : \alpha \text{ is a countable-valued generating function}\Big\},$$
where $\cJ$ is the $\sigma$-algebra of $G$-invariant sets. In this paper we will only use Rokhlin entropy for ergodic actions, and in this case $\cJ = \{X, \varnothing\}$ up to null sets and $\sH(\alpha | \cJ) = \sH(\alpha)$.

Rokhlin entropy was introduced by the second author and studied in \cite{S12,S14,S14a} and studied with Alpeev in \cite{S14b}.

When $G$ is amenable and the action is free, Rokhlin entropy coincides with the classical Kolmogorov--Sinai measured entropy \cite{ST14,S14b} (for $G = \Z$ this result is due to Rokhlin \cite{Roh67}). For free actions of sofic groups it is an open question if Rokhlin entropy agrees with measured sofic entropy (when the latter is not minus infinity). However, the following inequality is known. The lemma below is due to Bowen \cite{B10b} in the ergodic case and Alpeev--Seward \cite{S14b} in the non-ergodic case.

\begin{lem}[Bowen \cite{B10b}, Alpeev--Seward \cite{S14b}] \label{lem:sofrok}
Let $G$ be a sofic group with sofic approximation $\Sigma$. Then for every {\pmp} action $G \acts (X, \mu)$
$$\meash{\Sigma}{G}{X}{\mu} \leq \rh{G}{X}{\mu}.$$
\end{lem}

We will also need the following relative version of Rokhlin entropy. For a {\pmp} action $G \acts (X, \mu)$ and a $G$-invariant sub-$\sigma$-algebra $\calF$, the \emph{Rokhlin entropy} of $(X, \mu)$ relative to $\calF$ is
$$\rh{G}{X}{\mu | \calF} = \inf \Big\{ \sH(\alpha | \calF \vee \cJ) : \alpha \text{ is a countable-valued generating function}\},$$
where $\cJ$ is the $\sigma$-algebra of $G$-invariant sets. Again, we will only consider Rokhlin entropy for ergodic actions in which case $\cJ = \{X, \varnothing\}$. For some of the fundamental properties of (relative) Rokhlin entropy, see \cite{S14,S14a,S14b}. We will need one such property here.

\begin{thm}[Seward \cite{S14}] \label{thm:relrok}
Let $G$ be a countable group and let $G \acts (X, \mu)$ be an ergodic action with $\mu$ non-atomic. If $G \acts (Y, \nu)$ is a factor of $(X, \mu)$ and $\calF$ is the $G$-invariant sub-$\sigma$-algebra of $X$ associated to the factor map $X \rightarrow Y$ then
$$\rh{G}{X}{\mu} \leq \rh{G}{Y}{\nu} + \rh{G}{X}{\mu | \calF}.$$
\end{thm}

\newpage
\section{Measured entropy of algebraic subshifts} \label{sec:malg}

In this section we show that if $G$ acts by continuous group automorphisms of a profinite group $H$ and if the homoclinic group of $H$ is dense, then the sofic topological entropy of $G \acts H$ is equal to the sofic measured entropy with respect to the Haar probability measure on $H$. This result allows us to primarily focus on topological entropy in the other sections of this paper.

Recall that for a probability space $(X, \mu)$ and a function $f : X \rightarrow \R$, the \emph{expectation} of $f$ is
$$\avg(f) = \int_X f \ d \mu,$$
and the \emph{variance} of $f$ is
$$\var(f) = \int_X (f - \avg(f))^2 \ d \mu.$$
We will need the following well-known fact.

\begin{lem}[Chebyshev's Inequality] \label{lem:cheb}
Let $(X, \mu)$ be a standard probability space, let $f : X \rightarrow \R$ be a function with finite expectation and finite variance, and let $c > 0$. Then
$$\mu \Big( \Big\{ x \in X : |f(x) - \avg(f)| < c \cdot \sqrt{\var(f)} \Big\} \Big) \geq 1 - \frac{1}{c^2}.$$
\end{lem}

We now present the main theorem of this section.

\begin{thm}[Haar measure and topological sofic entropy] 
\label{thm:meas}
Let $G$ be a sofic group and let $\Sigma : (\sigma_n : G \rightarrow \Sym(D_n))$ be a sofic approximation to $G$. Let $H$ be a profinite group with Haar probability measure $\lambda_H$, and let $G$ act on $H$ by continuous group automorphisms. If the homoclinic group of $H$ is dense then
$$\meash{\Sigma}{G}{H}{\lambda_H} = \toph{\Sigma}{G}{H}.$$
\end{thm}

\begin{proof}
We claim that there is a profinite group $K$ and a compact $G$-invariant subgroup $X \leq K^G$ so that $H$ is $G$-equivariantly isomorphic to $X$ as topological groups. Indeed, we can simply take $K = H$ and define $\phi : H \rightarrow K^G$ by
$$\phi(h)(g) = g^{-1} \cdot h.$$
It is easily checked that $\phi$ is $G$-equivariant. Clearly $\phi$ is an injective continuous group homomorphism. Since $H$ is compact, $\phi$ must be a homeomorphism between $H$ and $X = \phi(H)$. In particular, $\phi$ maps the homoclinic group of $H$ to the homoclinic group of $X$. This proves the claim. For the remainder of the proof, we work with $X \leq K^G$.

Let $\lambda_X$ be the Haar probability measure on $X$. Say $K = \varprojlim K_m$ with each $K_m$ a finite group and with homomorphisms $\beta_m : K \rightarrow K_m$ and $\beta_{m, \ell} : K_\ell \rightarrow K_m$, $\ell \geq m$. Let $\alpha : K^G \rightarrow K$ be the tautological generating function and set $\alpha_\ell = \beta_\ell \circ \alpha$. By the variational principle (see Theorem \ref{thm:varp}), Lemma \ref{lem:topent}, and Proposition \ref{prop:tsubgroup} we have
\begin{align}
\meash{\Sigma}{G}{X}{\lambda_X} & \leq \toph{\Sigma}{G}{X}\nonumber\\
 & = \sup_{m \in \N} \inf_{\ell \geq m} \inf_{\substack{F \subseteq G \\ F \text{ finite}}} \limsup_{n \rightarrow \infty} \frac{1}{\vert D_n \vert} \cdot \log | \beta_{m, \ell} \circ \AM_X(\alpha_\ell, F, 0, \sigma_n)|.\label{eqn:meas}
\end{align}
So it suffices to show that $\meash{\Sigma}{G}{X}{\lambda_X}$ is greater than or equal to the right-most expression above.

Let $\mathcal{A}$ be the algebra generated by the maps $\alpha_m$, $m \in \N$. Then $\mathcal{A}$ is a generating algebra and thus in computing the measured sofic entropy of $G \acts (X, \lambda_X)$ we need only consider finite-valued $\mathcal{A}$-measurable functions. Every finite-valued $\mathcal{A}$-measurable function is coarser than some $\alpha_m$, so due to monotonicity properties in the definition of measured sofic entropy we have
$$\meash{\Sigma}{G}{X}{\lambda_X} = \sup_{m \in \N} \inf_{\ell \geq m} \inf_{\epsilon > 0} \inf_{\substack{F \subseteq G\\F \text{ finite}}} \limsup_{n \rightarrow \infty} \frac{1}{\vert D_n \vert} \cdot \log | \beta_{m, \ell} \circ \AM_{\lambda_X}(\alpha_\ell, F, \epsilon, \sigma_n)|.$$
Note that $\AM_X$ (see (\ref{eq:def M Y})) is distinct from $\AM_{\lambda_X}$ (see (\ref{eq:def M mu})).

Fix a finite $F \subseteq G$, $\ell \in \N$, and $\epsilon > 0$. Consider the map $\pi : K^G \rightarrow K_\ell^F$ defined by setting $\pi(z)(f) = \alpha_\ell(f \cdot z)$. Then $\pi$ is a group homomorphism so $P = \pi(X)$ is a subgroup of $K_\ell^F$ and $\pi$ pushes $\lambda_X$ forward to the Haar probability measure $\lambda_P$ on $P$. Therefore, setting $U_p = \pi^{-1}(p) \subseteq K^G$ for $p \in K_\ell^F$, we have
$$\lambda_X(U_p) = \lambda_P(p) = \begin{cases} \frac{1}{|P|} & \text{if } p \in P \\ 0 & \text{otherwise}. \end{cases}$$
For $i \in D_n$, define $\pi_i : K_\ell^{D_n} \rightarrow K_\ell^F$ by $\pi_i(a)(f) = a(\sigma_n(f)(i))$. For $a \in K_\ell^{D_n}$, $p \in K_\ell^F$, and $i\in D_n$ define $\chi_{p,i}(a) \in \{0, 1\}$ by
$$\chi_{p,i}(a) = 1 \Longleftrightarrow \pi_i(a) = p.$$
Next, set
$$C_p(a) = \frac{1}{\vert D_n \vert} \cdot \sum_{i\in D_n} \chi_{p,i}(a)$$
(equivalently, in the notation of (\ref{eq:def M mu}), $C_p(a) = \mu_n(U_{a,F,n}(p))$). By definition we have that $a \in \AM_{\lambda_X}(\alpha_\ell, F, \epsilon, \sigma_n)$ if and only if
\begin{equation} \label{eqn:mker2}
|C_p(a) - \lambda_X(U_p)| \leq \epsilon
\end{equation}
for every $p \in K_\ell^F$. We will use Chebyshev's Inequality (Lemma \ref{lem:cheb}) for the function $C_p: \AM_X(\alpha_\ell, F, 0, \sigma_n) \rightarrow [0, 1]$ to show that, asymptotically in $n$, its values concentrate around $\lambda_X(U_p)$. It will then follow that for large $n$
$$\Big| \AM_X(\alpha_\ell, F, 0, \sigma_n) \cap \AM_{\lambda_X}(\alpha_\ell, F, \epsilon, \sigma_n) \Big| \geq \frac{1}{2} \cdot \Big| \AM_X(\alpha_\ell, F, 0, \sigma_n) \Big|,$$
which will complete the proof. An essential ingredient in our argument is that $\AM_X(\alpha_\ell, F, 0, \sigma_n)$ is itself a (finite) group, and the normalized counting measure on this set coincides with its Haar probability measure. We proceed by estimating the expectation and variance of each function $C_p$.

Since the homoclinic group of $X$ is dense, for every $p \in P$ there is a homoclinic point $x_p \in X$ with $\pi(x_p) = p$. Let $S \subseteq G$ be a finite symmetric set such that $S$ contains $F F^{-1}$ and $\alpha_\ell(g \cdot x_p) = 1_{K_\ell}$ for every $p \in P$ and every $g \not\in S$. For $n \in \N$ let $B_n \subseteq D_n$ be the set of $ i \in D_n$ with either $\sigma_n(g) \circ \sigma_n(h)(i) \neq \sigma_n(g h)(i)$ for some $g, h \in S^2$ or $\sigma_n(g)(i) = \sigma_n(h)(i)$ for some $g \neq h \in S^2$. Since $\Sigma$ is a sofic approximation, we have that $|B_n| < \epsilon \cdot \vert D_n \vert$ for all sufficiently large $n$. Fix such a value of $n$. Define $\Delta_n^S$ to be the set of pairs $(i, i')$ with either $i \in B_n$, $i' \in B_n$, or $\sigma_n(r)(i) = \sigma_n(r')(i')$ for some $r, r' \in S^2$.

We claim that for all $(i, i') \not\in \Delta_n^S$ the function $\pi_i \times \pi_{i'} : K_\ell^{D_n} \rightarrow K_\ell^F \times K_\ell^F$ maps $\AM_X(\alpha_\ell, F, 0, \sigma_n)$ onto $P \times P$. Notice that by definition $\pi_i(a) \in P$ for every $i\in D_n$ and $a \in \AM_X(\alpha_\ell, F, 0, \sigma_n)$. Thus $\pi_i \times \pi_{i'}$ maps $\AM_X(\alpha_\ell, F, 0, \sigma_n)$ into $P \times P$. So fix $(i, i') \not\in \Delta_n^S$ and $p, p' \in P$. Define $a \in K_\ell^{D_n}$ by setting
$$a(\sigma_n(r)(i)) = \alpha_\ell(r \cdot x_p) \text{ for } r \in S^2, \qquad a(\sigma_n(r)(i')) = \alpha_\ell(r \cdot x_{p'}) \text{ for } r \in S^2,$$
and $a(j) = 1_{K_\ell}$ elsewhere. This is well-defined since $(i, i') \not\in \Delta_n^S$. Clearly $\pi_i(a) = p$ and $\pi_{i'}(a) = p'$. So we only need to check that $a \in \AM_X(\alpha_\ell, F, 0, \sigma_n)$. Fix $j \in D_n$. First suppose that $a(\sigma_n(f)(j)) = 1_{K_\ell}$ for all $f \in F$. Then $\pi_j(a)$ is the element of $K_\ell^F$ that is identically $1_{K_\ell}$. As $P$ is a subgroup of $K_\ell^F$, we must have $\pi_j(a) \in P$. Now suppose that there is $r \in F$ with $a(\sigma_n(r)(j)) \neq 1_{K_\ell}$. Since $\alpha_\ell(g \cdot x_p) = \alpha_\ell(g \cdot x_{p'}) = 1_{K_\ell}$ for $g \not\in S$, there must be $s \in S$ with either $\sigma_n(r)(j) = \sigma_n(s)(i)$ or $\sigma_n(r)(j) = \sigma_n(s)(i')$. Without loss of generality, suppose $\sigma_n(r)(j) = \sigma_n(s)(i)$. Then $j = \sigma_n(r^{-1} s)(i)$ since $i \not\in B_n$. Furthermore, since $i \not\in B_n$ it follows that $\sigma_n(f)(j) = \sigma_n(f r^{-1} s)(i)$ for all $f \in F$. Since $F F^{-1} S \subseteq S^2$ we obtain
$$\pi_j(a)(f) = a(\sigma_n(f)(j)) = a(\sigma_n(f r^{-1} s)(i)) = \alpha_\ell(f r^{-1} s \cdot x_p).$$
Therefore $\pi_j(a) \in P$ since $x_p \in X$. We have shown that $\pi_j(a) \in P$ for every $j \in D_n$. We conclude that $a \in \AM_X(\alpha_\ell, F, 0, \sigma_n)$ as claimed.

Note that $\AM_X(\alpha_\ell, F, 0, \sigma_n)$ is a subgroup of $K_\ell^{D_n}$. Let $\lambda_n$ denote the Haar probability measure on $\AM_X(\alpha_\ell, F, 0, \sigma_n)$. Since $\pi_i$ is a group homomorphism, it follows from the previous paragraph that $\pi_i$ pushes $\lambda_n$ forward to $\lambda_P$ for $i \not\in B_n$, and $\pi_i \times \pi_j$ pushes $\lambda_n$ forward to $\lambda_P \times \lambda_P$ for $(i, j) \not\in \Delta_n^S$. Fix $0 < \delta < \epsilon / 2$ with
$$(2 \delta - \delta^2) \cdot \frac{1}{|P|^2} + \delta < \frac{\epsilon^2}{8 |P|}.$$
Let $n$ be large enough so that $|B_n| < \delta \cdot \vert D_n \vert$ and $|\Delta_n^S| < \delta \cdot \vert D_n \vert^2$. For every $p \in P$ and $i \not\in B_n$ the function $\chi_{p,i} : \AM_X(\alpha_\ell, F, 0, \sigma_n) \rightarrow \{0, 1\}$ has expected value
$$\avg(\chi_{p,i}) = \int \chi_{p,i}(a) \ \mathrm{d} \lambda_n(a) = \lambda_P(p) = \frac{1}{|P|}.$$
For $i \in B_n$ we have $0 \leq \avg(\chi_{p,i}) \leq 1$, and therefore $C_p$ has expected value
$$\frac{1}{|P|} - \frac{\epsilon}{2} < (1 - \delta) \cdot \frac{1}{|P|} < \avg(C_p) < (1 - \delta) \cdot \frac{1}{|P|} + \delta < \frac{1}{|P|} + \frac{\epsilon}{2}.$$
Similarly, for all $p, p' \in P$ and all $(i, j) \not\in \Delta_n^S$
$$\int \chi_{p,i}(a) \cdot \chi_{p',j}(a) \ \mathrm{d} \lambda_n(a) = \lambda_P(p) \cdot \lambda_P(p') = \frac{1}{|P|^2}.$$
Therefore $C_p$ has variance
\begin{align*}
\var(C_p) & = \int (C_p(a) - \avg(C_p))^2 \ \mathrm{d} \lambda_n(a) \\
 & = - \avg(C_p)^2 + \int C_p(a)^2 \ \mathrm{d} \lambda_n(a) \\
 & < - (1 - \delta)^2 \cdot \frac{1}{|P|^2} + \frac{1}{\vert D_n \vert^2} \cdot \sum_{i, j} \int \chi_{p,i}(a) \cdot \chi_{p,j}(a) \ \mathrm{d} \lambda_n(a) \\
 & \leq - (1 - \delta)^2 \cdot \frac{1}{|P|^2} + \frac{1}{\vert D_n \vert^2} \cdot \sum_{(i,j) \not\in \Delta_n^S} \frac{1}{|P|^2} + \frac{1}{\vert D_n \vert^2} \cdot |\Delta_n^S| \\
 & < - (1 - \delta)^2 \cdot \frac{1}{|P|^2} + \frac{1}{|P|^2} + \delta \\
 & = (2 \delta - \delta^2) \cdot \frac{1}{|P|^2} + \delta \\
 & < \frac{\epsilon^2}{8 |P|}
\end{align*}
We apply Lemma \ref{lem:cheb} and conclude that for every $p \in P$
\begin{align*}
\lambda_n \Big( & \Big\{a \in \AM_X(\alpha_\ell, F, 0, \sigma_n) : |C_p(a) - \lambda_P(p)| < \epsilon \Big\} \Big) \\
 & \geq \lambda_n \Big( \Big\{a \in \AM_X(\alpha_\ell, F, 0, \sigma_n) : |C_p(a) - \avg(C_p)| < \frac{\epsilon}{2} \Big\} \Big) \\
 & \geq \lambda_n \Big( \Big\{a \in \AM_X(\alpha_\ell, F, 0, \sigma_n) : |C_p(a) - \avg(C_p)| < \sqrt{2 |P|} \cdot \sqrt{\var(C_p)} \Big\} \Big) \\
 & \geq 1 - \frac{1}{2 |P|}
\end{align*}
Thus,
$$\lambda_n \Big( \Big\{a \in \AM_X(\alpha_\ell, F, 0, \sigma_n) : \forall p \in P \ |C_p(a) - \lambda_P(p)| < \epsilon \Big\} \Big) \geq 1 / 2.$$
As we have already remarked, it is immediate from the definitions that $C_p(a) = 0 = \lambda_P(p)$ for every $p \in K_\ell^F \setminus P$ and $a \in \AM_X(\alpha_\ell, F, 0, \sigma_n)$. So we obtain as claimed
$$\Big| \AM_{\lambda_X}(\alpha_\ell, F, \epsilon, \sigma_n) \cap \AM_X(\alpha_\ell, F, 0, \sigma_n) \Big| \geq \frac{1}{2} \cdot \Big| \AM_X(\alpha_\ell, F, 0, \sigma_n) \Big|.$$
Therefore
\begin{align*}
\limsup_{n \rightarrow \infty} \frac{1}{\vert D_n \vert} \log \Big| \beta_{m, \ell} \circ \AM_{\lambda_X}(\alpha_\ell, F, \epsilon, \sigma_n) \Big| & \geq \limsup_{n \rightarrow \infty} \frac{1}{\vert D_n \vert} \log \Big( \frac{1}{2} \Big| \beta_{m, \ell} \circ \AM_X(\alpha_\ell, F, 0, \sigma_n) \Big| \Big) \\
 & = \limsup_{n \rightarrow \infty} \frac{1}{\vert D_n \vert} \log \Big| \beta_{m, \ell} \circ \AM_X(\alpha, F, 0, \sigma_n) \Big|
\end{align*}
By taking the infimum over finite $F \subseteq G$, $\epsilon > 0$, and $\ell \geq m$ and then the supremum over $m \in \N$, we obtain $\meash{\Sigma}{G}{X}{\lambda_X} \geq \toph{\Sigma}{G}{X}$.
\end{proof}

As mentioned in the introduction, in the classical setting there is a close connection between homoclinic groups and entropy. It is interesting to wonder how many of these connections persist in the non-amenable realm.

\begin{question}
Let $K$ be a finite group and let $X \subseteq K^G$ be an algebraic subshift. Suppose that there is precisely one element of $X$ in the homoclinic group (namely the function from $G$ to $K$ that is constantly $1_K$). Does $G \acts X$ have topological entropy $\toph{\Sigma}{G}{X} \in \{-\infty, 0\}$ for every sofic approximation $\Sigma$?
\end{question}

Another important question is if the profinite assumption can be removed from the above theorem.

\begin{question}
Let $H$ be a compact metrizable group, let $G$ act on $H$ by group automorphisms, and let $\lambda_H$ be the Haar probability measure on $H$. If the homoclinic points are dense in $H$, does it follow that $\meash{\Sigma}{G}{H}{\lambda_H} = \toph{\Sigma}{G}{H}$ for every sofic approximation $\Sigma$?
\end{question}

We also deduce the following consequence of Theorem \ref{thm:period}.

\begin{cor}[Measured entropy and fixed points]
\label{cor:meas2}
Let $G$ be a residually finite group, let $(G_n)$ be a chain of finite-index normal subgroups with trivial intersection, and let $\Sigma=\bigl(\sigma_n:G\to \Sym(G_n \backslash G)\bigr)$ be the associated sofic approximation. Let $K$ be a finite group, let $X \subseteq K^G$ be an algebraic subshift of finite type, and let $\lambda_X$ be the Haar probability measure on $X$. If $\Fix{G_n}{X}$ converges to $X$ in the Hausdorff metric and $\lambda_X$ is ergodic then
$$\meash{\Sigma}{G}{X}{\lambda_X} = \toph{\Sigma}{G}{X} = \limsup_{n \rightarrow \infty} \frac{1}{|G : G_n|} \cdot \log \Big| \Fix{G_n}{X} \Big|.$$
\end{cor}

\begin{proof}
For each $n$ let $\mu_n$ be the normalized counting measure on $\Fix{G_n}{X}$. We claim that $\mu_n$ converges to $\lambda_X$ in the weak$^*$-topology. The argument we present here essentially comes from \cite[Lemma 6.2]{B11}.

Let $\mu$ be a limit point of the sequence $\mu_n$. Since $X$ is closed we have $\mu(X) = 1$. So it suffices to show that $\mu$ is invariant under the translation action of $X$ on itself. Fix a compatible metric $d$ on $X$. By the Birkhoff--Kakutani theorem, we may choose $d$ to be right-invariant, meaning $d(x z, y z) = d(x, y)$. Consider a point $x \in X$ and a continuous function $f : K^G \rightarrow \R$. Let $\epsilon > 0$. Since $f$ is continuous and $K^G$ is compact, $f$ is uniformly continuous. So there is $\delta > 0$ such that $|f(y) - f(z)| < \epsilon$ whenever $d(y, z) < \delta$. For $n$ sufficiently large,
$$\left| \int_{K^G} f(z) d \mu_n(z) - \int_{K^G} f(z) d \mu(z) \right| < \epsilon,$$
$$\left| \int_{K^G} f(x \cdot z) d \mu_n(z) - \int_{K^G} f(x \cdot z) d \mu(z) \right| < \epsilon,$$
and there is $x_n \in \Fix{G_n}{X}$ with $d(x, x_n) < \delta$. Clearly $x_n \cdot \Fix{G_n}{X} = \Fix{G_n}{X}$ so that
$$\int_{K^G} f(z) d \mu_n(z) = \int_{K^G} f(x_n \cdot z) d \mu_n(z).$$
By the uniform continuity of $f$, we also have
$$\left| \int_{K^G} f(x \cdot z) d \mu_n(z) - \int_{K^G} f(x_n \cdot z) d \mu_n(z) \right| < \epsilon.$$
Therefore
$$\left| \int_{K^G} f(x \cdot z) d \mu(z) - \int_{K^G} f(z) d \mu(z) \right| < 3 \epsilon.$$
As $f$, $\epsilon$, and $x \in X$ were arbitrary, we conclude from the uniqueness of Haar measure that $\mu = \lambda_X$.

Since $\lambda_X$ is ergodic and the $\mu_n$'s converge to $\lambda_X$, it follows from work of Bowen \cite[Theorem 4.1]{B11} that
$$\limsup_{n \rightarrow \infty} \frac{1}{|G : G_n|} \cdot \log \Big| \Fix{G_n}{X} \Big| \leq \meash{\Sigma}{G}{X}{\lambda_X}.$$
Now applying the variational principle (Theorem \ref{thm:varp}) and Theorem \ref{thm:period} completes the proof.
\end{proof}

\newpage
\section{The Ornstein--Weiss maps and cost} \label{sec:OW}

In this section we prove Theorems \ref{thm:rokcost} and \ref{thm:cost}. Recall that the original Ornstein--Weiss factor map $\theta : (\Z / 2 \Z)^{F_2} \rightarrow (\Z / 2 \Z \times \Z / 2 \Z)^{F_2}$ is defined as
$$\theta(x)(f) = \Big( x(f) - x(fa), \ x(f) - x(f b) \Big) \mod 2,$$
where $\{a, b\}$ is a free generating set for the rank two free group $F_2$. We generalize this map as follows. If $G$ is a countable group with generating set $S$ ($S$ need not be finite), and $K$ is a finite additive abelian group, then we define $\theta^{ow}_S : K^G \rightarrow (K^S)^G$ by
$$\theta^{ow}_S(x)(g)(s) = x(g) - x(g s) \in K.$$
The map $\theta^{ow}_S$ is $G$-equivariant since for $g, h \in G$, $s \in S$, and $x \in K^G$ we have
\begin{align*}
[h^{-1} \cdot \theta^{ow}_S(x)](g)(s) = \theta^{ow}_S(x)(h g)(s) & = x(h g) - x(h g s)\\
 & = (h^{-1} \cdot x)(g) - (h^{-1} \cdot x)(g s) = \theta^{ow}_S(h^{-1} \cdot x)(g)(s).
\end{align*}

An alternative viewpoint on $\theta^{ow}_S$ comes from the Cayley graph $\Cay(G, S)$. Specifically, $K^G$ can be identified with the set of $K$-labellings of the vertices in $\Cay(G, S)$, and $(K^S)^G$ can be identified with the set of $K$-labellings of the directed edges of $\Cay(G, S)$. Then for a vertex labeling $x \in K^G$, $\theta^{ow}_S(x)$ is the edge labeling that labels each directed edge by the difference in $x$-values of its initial vertex and terminal vertex.

In the lemma below and the remainder of this section we implicitly identify $K$ with the subgroup of constant functions in $K^G$.

\begin{lem} \label{lem:owiso}
Let $G$ be a countable group with generating set $S$ and let $K$ be a finite abelian group. Let $\pi : K^G \rightarrow K^G / K$ be the quotient map. Then, up to a $G$-equivariant isomorphism of topological groups, $\theta^{ow}_S$ is identical to $\pi$.
\end{lem}

\begin{proof}
This is immediate from the fact that both $\theta^{ow}_S$ and $\pi$ are continuous $G$-equivariant group homomorphisms having kernel $K$.
\end{proof}

\begin{lem}
Let $G$ be a countably infinite group, let $K$ be a non-trivial finite abelian group, and let $\lambda_{K^G / K}$ denote the Haar probability measure on $K^G / K$. Then the action of $G$ on $K^G / K$ is essentially free with respect to $\lambda_{K^G / K}$.
\end{lem}

\begin{proof}
Fix $1_G \neq g \in G$. We have $g \cdot (x + K) = x + K$ if and only if there is $k \in K$ with $x(g^{-1} h) = x(h) + k$ for all $h \in G$. Since the Haar probability measure $\lambda_{K^G}$ naturally pushes forward to $\lambda_{K^G/K}$, it suffices to show that $\lambda_{K^G}(F_{g,k}) = 0$, where
$$F_{g, k} = \{ x \in K^G : \forall h \in G \ x(g^{-1} h) = x(h) + k\}.$$
If $x \in F_{g, k}$, then, by the invariance of Haar measure under translation, $F_{g, k}$ has the same measure as $F_{g, k} - x = F_{g, 0_K}$. The set $F_{g, 0_K}$ is the collection of points $y \in K^G$ fixed by $g$. Since $(K^G, \lambda_{K^G})$ is a Bernoulli shift and Bernoulli shift actions are essentially free, we have $\lambda_{K^G}(F_{g, 0_K}) = 0$.
\end{proof}

We are now ready for the main theorem of this section. We would like to emphasize that in the theorem below $G$ is not required to be sofic.

\begin{thm}[Rokhlin entropy and cost]
 \label{thm:rokcost}
Let $G$ be a countably infinite group, let $K$ be a finite abelian group, and let $\lambda_{K^G/K}$ denote the Haar probability measure on $K^G / K$. Then
$$\rh{G}{K^G / K}{\lambda_{K^G/K}} \leq \costsup(G) \cdot \log |K|.$$
\end{thm}

We remind the reader that when $G$ is finitely generated, the supremum-cost is realized by any non trivial Bernoulli shift action (Ab{\'e}rt-Weiss~\cite{AW13}).

\begin{proof}
If $K = \{0_K\}$ is trivial then both sides of the above expression are $0$, and if $\costsup(G) = \infty$ then there is nothing to show. So suppose that $K$ is non-trivial and $\costsup(G) < \infty$. Fix $\epsilon > 0$. Since $G \acts K^G / K$ is essentially free, a theorem of Seward and Tucker-Drob \cite{ST14} states that there is an equivariant factor map $f : (K^G / K, \lambda_{K^G / K}) \rightarrow (Y, \nu)$ such that the action of $G$ on $(Y, \nu)$ is essentially free and $\rh{G}{Y}{\nu} < \epsilon / 2$. Fix a countable-valued generating function $\gamma' : Y \rightarrow \N$ with $\sH(\gamma') < \epsilon / 2$. Let $\calF$ be the $G$-invariant sub-$\sigma$-algebra of $K^G / K$ associated to the factor map $f: (K^G / K, \lambda_{K^G / K}) \rightarrow (Y, \nu)$. Also set $\gamma = \gamma' \circ f$.

Since $\costsup(G) < \infty$, we can find a finite graphing $\Phi = \{\phi_i : 1 \leq i \leq m\}$ for $G \acts (Y, \nu)$ with
$$\sum_{i = 1}^m \nu(\dom(\phi_i)) < \costsup(G) + \frac{\epsilon}{2 \log|K|}$$
(that $\Phi$ can be chosen finite follows from \cite{G00}). For $y \in Y$, let $\Phi[y]$ denote the graph on $G$ with a directed edge from $g$ to $h$ if and only if there is $i \in I$ with $g \cdot y \in \dom(\phi_i)$ and $\phi_i(g \cdot y) = h \cdot y$. For $x + K \in K^G / K$, write $\Phi[x + K]$ for $\Phi[f(x + K)]$.  Since $\Phi$ graphs the orbit equivalence relation of $G \acts (Y, \nu)$, the graph $\Phi[x + K]$ is connected for $\lambda_{K^G / K}$-almost-every $x + K \in K^G / K$.

Let $\alpha : K^G \rightarrow K$ be the tautological generating function. For each $1 \leq i \leq m$ define the function $\beta_i : K^G / K \rightarrow K$ by
$$\beta_i(x + K) = \begin{cases}
0_K & \text{if } f(x + K) \not\in \dom(\phi_i)\\
\alpha(x) - \alpha(g \cdot x) & \text{if } \phi_i(f(x + K)) = f(g \cdot x + K).
\end{cases}$$
We set $\beta = \beta_1 \times \cdots \times \beta_m$ and claim that $\beta \times \gamma$ is a generating function mod $\lambda_{K^G / K}$. Intuitively, the function $\beta$ describes a variant of the Ornstein--Weiss map that uses the graphing $\Phi$ instead of the Cayley graph $\Cay(G, S)$. It is precisely the connectivity of the graph that is needed in order for these maps to have kernel $K$. For this reason, one should expect the $G$-translates of $\beta$ to separate points in $K^G / K$, however $\gamma$ is required as well. A key point is that (modulo a null set) if two points $x + K$ and $z + K$ satisfy $\gamma(g \cdot x + K) = \gamma(g \cdot z + K)$ for every $g \in G$ then they must have the same graph $\Phi[x + K] = \Phi[z + K]$. We now check the details of this argument.

Let $X \subseteq K^G / K$ be an invariant conull set such that $\Phi[x + K]$ is connected whenever $x + K \in X$ and such that for all $x_1 + K, x_2 + K \in X$ with $f(x_1 + K) \neq f(x_2 + K)$ there is $g \in G$ with $\gamma(g \cdot x_1 + K) \neq \gamma(g \cdot x_2 + K)$. It suffices to check that if $x + K, z + K \in X$ and $\beta \times \gamma(g \cdot x + K) = \beta \times \gamma(g \cdot z + K)$ for all $g \in G$ then $x + K = z + K$. So suppose that $\beta \times \gamma(g \cdot x + K) = \beta \times \gamma(g \cdot z + K)$ for all $g \in G$. In particular $\gamma(g \cdot x + K) = \gamma(g \cdot z + K)$ for all $g \in G$, so we must have $f(x + K) = f(z + K)$ and hence $\Phi[x + K] = \Phi[z + K]$. Suppose that there is an edge in $\Phi[x + K]$ directed from $g$ to $h$. Then there is $1 \leq i \leq m$ with $f(g \cdot x + K), f(g \cdot z + K) \in \dom(\phi_i)$, $\phi_i(f(g \cdot x + K)) = f(h \cdot x + K)$, and $\phi_i(f(g \cdot z + K)) = f(h \cdot z + K)$. We have
\begin{align*}
- \alpha(h \cdot x) + \alpha(h \cdot z) & = - \alpha(g \cdot x) + [\alpha(g \cdot x) - \alpha(h \cdot x)] - \Big( - \alpha(g \cdot z) + [\alpha(g \cdot z) - \alpha(h \cdot z)] \Big)\\
 & = - \alpha(g \cdot x) + \beta_i(g \cdot x) + \alpha(g \cdot z) - \beta_i(g \cdot z)\\
 & = - \alpha(g \cdot x) + \alpha(g \cdot z).
\end{align*}
So $-\alpha(h \cdot x) + \alpha(h \cdot z) = -\alpha(g \cdot x) + \alpha(h \cdot z)$ whenever $g$ and $h$ are adjacent in $\Phi[x + K]$. Since $\Phi[x + K]$ is connected, it follows that $x + K = z + K$. We conclude that $\beta \times \gamma$ is a generating function.

Now we will bound $\rh{G}{K^G / K}{\lambda_{K^G / K}}$. It may aid the reader to recall Lemma \ref{lem:shan}. By Theorem \ref{thm:relrok} we have
\begin{align}
\rh{G}{K^G / K}{\lambda_{K^G / K}} & \leq \rh{G}{Y}{\nu} + \rh{G}{K^G / K}{\lambda_{K^G / K} | \calF}\nonumber\\
 & < \frac{\epsilon}{2} + \sH(\beta \times \gamma | \calF)\nonumber\\
 & = \frac{\epsilon}{2} + \sH(\beta | \calF).\label{eqn:rh}
\end{align}
For each $1 \leq i \leq m$, if we consider the two-piece partition
$$\cQ_i = \{f^{-1}(\dom(\phi_i)), (K^G / K) \setminus f^{-1}(\dom(\phi_i))\} \subseteq \calF,$$
we have that the restriction of $\beta_i$ to $(K^G / K) \setminus f^{-1}(\dom(\phi_i))$ is trivial and thus
\begin{align*}
\sH(\beta_i | \calF) & \overset{\textrm{eq. }(\ref{eq:cond Shannon as inf})}{\leq} \sH(\beta_i | \cQ_i)\\
 & \overset{\textrm{def. }(\ref{eq:conditional Shannon entropy})}{=} \lambda_{K^G / K}(f^{-1}(\dom(\phi_i))) \cdot \sH_{f^{-1}(\dom(\phi_i))}(\beta_i)\\
 & \qquad + \lambda_{K^G / K}\Big((K^G / K) \setminus f^{-1}(\dom(\phi_i))\Big) \cdot \sH_{(K^G / K) \setminus f^{-1}(\dom(\phi_i))}(\beta_i)\\
 & = \lambda_{K^G / K}(f^{-1}(\dom(\phi_i))) \cdot \sH_{f^{-1}(\dom(\phi_i))}(\beta_i)\\
 & \leq \lambda_{K^G / K}(f^{-1}(\dom(\phi_i))) \cdot \log |K|.
\end{align*}
So we have
\begin{align*}
\sH(\beta | \calF) & \leq \sum_{i = 1}^m \sH(\beta_i | \calF)\\
 & \leq \sum_{i = 1}^m \lambda_{K^G / K}(f^{-1}(\dom(\phi_i))) \cdot \log|K|\\
 & = \sum_{i = 1}^m \nu(\dom(\phi_i)) \cdot \log|K|\\
 & < \costsup(G) \cdot \log |K| + \frac{\epsilon}{2}.
\end{align*}
From the above inequality and equation (\ref{eqn:rh}) we find that $\rh{G}{K^G / K}{\lambda_{K^G / K}} < \costsup(G) \cdot \log |K| + \epsilon$. Now let $\epsilon$ tend to $0$.
\end{proof}

\begin{thm}[Measured and topological entropy vs cost and $\ell^2$-Betti number] 
\label{thm:cost}
Let $G$ be a countably infinite sofic group, let $\Sigma$ be a sofic approximation to $G$, let $\KK$ be a finite field, and let $\lambda_{\KK^G / \KK}$ be Haar probability measure on $\KK^G / \KK$. Then
\begin{align*}
& \ \meash{\Sigma}{G}{\KK^G / \KK}{\lambda_{\KK^G / \KK}} = \toph{\Sigma}{G}{\KK^G / \KK} \leq \costsup(G) \cdot \log |\KK|.\\
\intertext{Furthermore, if $G$ is finitely generated then}
\Big( 1 + \beta_{(2)}^1(G) \Big) \cdot \log |\KK| \leq & \ \meash{\Sigma}{G}{\KK^G / \KK}{\lambda_{\KK^G / \KK}} = \toph{\Sigma}{G}{\KK^G / \KK}.
\end{align*}
\end{thm}

\begin{proof}
If $\KK = \{0_\KK\}$ is trivial then all expressions above have value $0$. So assume that $\KK$ is non-trivial. Let $S$ be a generating set for $G$. By Lemma \ref{lem:owiso} $\KK^G / \KK$ is $G$-equivariantly isomorphic (as topological groups) to $\theta^{ow}_S(\KK^G) \subseteq (\KK^S)^G$. Now $\KK^S$ is a profinite group, and $\theta^{ow}_S(\KK^G)$ has dense homoclinic group since $\KK^G$ has dense homoclinic group. Thus by Theorem \ref{thm:meas}
\begin{equation} \label{eqn:equal}
\toph{\Sigma}{G}{\KK^G / \KK} = \meash{\Sigma}{G}{\KK^G / \KK}{\lambda_{\KK^G / \KK}}.
\end{equation}
The upper bound follows immediately from Theorem \ref{thm:rokcost}:
$$\meash{\Sigma}{G}{\KK^G / \KK}{\lambda_{\KK^G / \KK}} \leq \rh{G}{\KK^G / \KK}{\lambda_{\KK^G / \KK}} \leq \costsup(G) \cdot \log |\KK|.$$

Now we assume that $G$ is finitely generated and consider the lower bound. By redefining $S$ if necessary, we may suppose that $S$ is finite. Let $G = \langle S | R \rangle$ be a presentation for $G$. Let $L = \Cay_R(G, S)$ denote the $2$-dimensional polygonal complex obtained from the Cayley graph $\Cay(G, S)$ by inserting a $2$-dimensional polygon for every relation in $R$. Then the coboundary map
$$\delta^1 : C^0(L, \KK) \rightarrow C^1(L, \KK)$$
coincides with the generalized Ornstein--Weiss map $\theta_S^{ow}$. Since $L$ is simply connected we have $H^1(L, \KK) = 0$. Additionally, since $S$ is finite the action of $G$ on the $1$-skeleton of $L$ is cocompact. So by Corollary \ref{cor:juzv} (see Remark \ref{rem:polygon})
\begin{equation} \label{eqn:1betti}
\begin{aligned}
      & \ \toph{\Sigma}{G}{C^0(L, \KK)} + \beta_{(2)}^1(L : G) \cdot \log |\KK|\\
 \leq & \ \toph{\Sigma}{G}{\ker \theta_S^{ow}} + \toph{\Sigma}{G}{\KK^G / \KK}.
\end{aligned}
\end{equation}
As $L$ is simply connected, we have
\begin{equation} \label{eqn:bb}
\beta_{(2)}^1(L : G) = \beta_{(2)}^1(G).
\end{equation}
Additionally, since $\ker \theta_S^{ow} = \KK$ is finite, one can either compute directly or apply \cite[Prop. 1.7]{H14} to get
\begin{equation} \label{eqn:zero}
\toph{\Sigma}{G}{\ker \theta_S^{ow}} = 0.
\end{equation}
Finally, $C^0(L, \KK)$ is isomorphic to $\KK^G$ and so has entropy $\log |\KK|$. Thus, putting together equations (\ref{eqn:equal}), (\ref{eqn:1betti}), (\ref{eqn:bb}), and (\ref{eqn:zero}) we get
\begin{equation*}
(1 + \beta_{(2)}^1(G)) \cdot \log |\KK| \leq \toph{\Sigma}{G}{\KK^G / \KK} = \meash{\Sigma}{G}{\KK^G / \KK}{\lambda_{\KK^G / \KK}}. \qedhere
\end{equation*}
\end{proof}

\newpage
\section{Discussion} \label{sec:disc}

We mention the possibility of generalizations to Theorems \ref{thm:rokcost} and \ref{thm:cost} that may relate first $\ell^2$-Betti numbers and cost of equivalence relations to entropy. We first discuss a mild generalization. Consider a free {\pmp} action $G \acts (X, \mu)$. For a finite abelian group $K$ one can consider the direct product action $G \acts (X \times (K^G / K), \mu \times \lambda_{K^G/ K})$, where $\lambda_{K^G / K}$ is the Haar probability measure on $K^G / K$. The proof of Theorem \ref{thm:rokcost} is easily modified to prove the following. Specifically, in that proof $(X, \mu)$ would take the role of the small entropy factor $G \acts (Y, \nu)$ that was obtained by invoking the result of Seward--Tucker-Drob. Below we write $\Borel(X)$ for the Borel $\sigma$-algebra of $X$.

\begin{thm}
Let $G$ be a countably infinite group, let $K$ be a finite abelian group, and let $\lambda_{K^G / K}$ be the Haar probability measure on $K^G / K$. Then for any free {\pmp} action $G \acts (X, \mu)$ we have
$$\rh{G}{X \times (K^G / K)}{\mu \times \lambda_{K^G / K} | \Borel(X)} \leq \cost(G \acts (X, \mu)) \cdot \log |K|.$$
\end{thm}

When $G$ is amenable it is always true that $\rh{G}{X \times Y}{\mu \times \nu | \Borel(X)} = \rh{G}{Y}{\nu}$ for free actions. However it is unknown if this property holds for any notion of entropy for any non-amenable group. If this property were to hold in general, then by choosing $G \acts (X, \mu)$ so that $\cost(G) = \cost(G \acts (X, \mu))$, one could strengthen Theorems \ref{thm:rokcost} and \ref{thm:cost} to hold for $\cost(G)$ in place of $\costsup(G)$.

Now consider the more abstract situation of a probability space $(X, \mu)$ and a probability-measure-preserving countable Borel equivalence relation $R$ on $X$. For $x \in X$ write $[x]_R$ for the $R$-class of $x$. Fix a finite additive abelian group $K$ and consider the space
$$Y = \{(x, f + K) \ | \ x \in X \text{ and } f : [x]_R \rightarrow K\}.$$
Here for each $x \in X$ we identify $K$ with the collection of constant functions from $[x]_R$ to $K$. There is a natural projection $\pi : Y \rightarrow X$, and it is easily seen that the fiber of each $x \in X$ is the compact abelian group of all functions $f : [x]_R \rightarrow K$ modulo the constant functions. Thus one can place Haar probability measure on each fiber to obtain a probability measure $\nu$ on $Y$ with the property that $\pi_*(\nu) = \mu$. We define an equivalence relation $R'$ on $Y$ by
$$(x_0, f_0 + K) \ R' \ (x_1, f_1 + K) \Longleftrightarrow x_0 \ R \ x_1 \text{ and } f_0 + K = f_1 + K.$$
It is easy to check that $R'$ is $\nu$-measure-preserving. Also, $R'$ is a class-bijective extension of $R$ in the sense that $(y_0, y_1) \in R'$ and $\pi(y_0) = \pi(y_1)$ implies $y_0 = y_1$.

We suspect that it may be possible to modify the proof of Theorem \ref{thm:rokcost} to obtain
$$\rh{R}{Y}{\nu | \Borel(X)} \leq \cost(R) \cdot \log |K|.$$
Here $\rh{R}{Y}{\nu | \Borel(X)}$ is the relative Rokhlin entropy of the extension of $R$ by $R'$. Equivalently, if $G \acts (X, \mu)$ is any (possibly non-free) action whose orbit equivalence relation coincides with $R$, then one can lift the $G$-action through $\pi$ to obtain an action $G \acts (Y, \nu)$ whose orbit equivalence relation is $R'$, and in this case, for any such action we have \cite{S14a}
$$\rh{R}{Y}{\nu | \Borel(X)} = \rh{G}{Y}{\nu | \Borel(X)}.$$
There is also a possibility of a connection between the first $\ell^2$-Betti number of $R$ and sofic entropy, though this is much more speculative. The $\ell^2$-Betti numbers of {\pmp} countable equivalence relations were defined by the first author in \cite{G02}. If $R$ is a sofic equivalence relation, then one can view $R$ as a sofic {\pmp} groupoid and consider the action $R \acts (Y, \nu)$ whose orbit equivalence relation is $R'$. Using Bowen's definition of sofic entropy for actions of sofic groupoids \cite{Bb}, is it true that, for any sofic approximation $\Sigma$ to $R$,
$$\Big(\beta^1_{(2)}(R) + 1 \Big) \cdot \log |K| \leq \meash{\Sigma}{R}{Y}{\nu}?$$

We mention one final question. Could it be that all of the inequalities mentioned here between first $\ell^2$-Betti numbers, cost, sofic entropy, and Rokhlin entropy are actually equalities? We see no good reason why this should be true, but it is an interesting question. It would be quite strange if the orbit-equivalence invariants of cost and first $\ell^2$-Betti numbers were found to be special cases of entropy.

\bibliographystyle{alpha}

\end{document}